\documentclass[11pt,reqno]{amsart}
\usepackage[margin=1in]{geometry}
\usepackage[foot]{amsaddr}

\usepackage{amsmath, amsthm, amssymb, url, color, hyperref, bbm, graphicx,
comment,enumerate, tikz-cd, subcaption}

\hypersetup{
    colorlinks,
    citecolor=blue,
    filecolor=black,
    linkcolor=black,
    urlcolor=black
}

\allowdisplaybreaks

\newcommand{\Id}{\operatorname{Id}}
\newcommand{\R}{\mathbb{R}}
\newcommand{\C}{\mathbb{C}}
\newcommand{\ZZ}{\mathbb{Z}}
\newcommand{\eps}{\varepsilon}
\newcommand{\N}{\mathcal{N}}
\newcommand{\E}{\mathbb{E}}
\renewcommand{\P}{\mathbb{P}}
\renewcommand{\O}{\mathcal{O}}
\newcommand{\Cov}{\operatorname{Cov}}
\newcommand{\Var}{\operatorname{Var}}
\newcommand{\Tr}{\operatorname{Tr}}
\newcommand{\cT}{\mathcal{T}}
\newcommand{\cA}{\mathcal{A}}
\newcommand{\cB}{\mathcal{B}}
\newcommand{\cC}{\mathcal{C}}
\newcommand{\cS}{\mathcal{S}}
\newcommand{\cE}{\mathcal{E}}

\newcommand{\cP}{\mathcal{P}}
\newcommand{\cI}{\mathcal{I}}
\newcommand{\cJ}{\mathcal{J}}
\newcommand{\cV}{\mathcal{V}}
\newcommand{\1}{\mathbf{1}}

\newcommand{\cO}{\mathcal{O}}
\newcommand{\dist}{\operatorname{dist}}
\newcommand{\Unif}{\operatorname{Unif}}
\newcommand{\const}{\mathrm{const}}

\newcommand{\ii}{\mathbf{i}}
\newcommand{\wtheta}{\widetilde{\theta}}
\newcommand{\wt}{\tilde{t}}
\newcommand{\wvarphi}{\widetilde{\varphi}}

\newcommand{\Arg}{\operatorname{Arg}}
\newcommand{\diag}{\operatorname{diag}}
\renewcommand{\Re}{\operatorname{Re}}
\renewcommand{\Im}{\operatorname{Im}}

\newcommand{\der}{\mathsf{d}}
\newcommand{\cR}{\mathcal{R}}
\newcommand{\trdeg}{\operatorname{trdeg}}
\newcommand{\cK}{\mathcal{K}}

\newcommand{\bI}{\mathbf{1}}

\newcommand{\HS}{{\mathrm{HS}}}
\newcommand{\rO}{\mathrm{O}}

\newcommand{\trunc}{\text{trunc}}
\newcommand{\formal}{\text{formal}}

\newtheorem{theorem}{Theorem}[section]
\newtheorem{lemma}[theorem]{Lemma}

\newtheorem{corollary}[theorem]{Corollary}
\theoremstyle{definition}
\newtheorem{remark}[theorem]{Remark}
\newtheorem*{remark*}{Remark}
\newtheorem*{example*}{Example}
\newtheorem{example}[theorem]{Example}
\newtheorem{definition}[theorem]{Definition}

\numberwithin{equation}{section}
\numberwithin{figure}{section}

\title[Likelihood landscape for orbit recovery]{Likelihood landscape and
maximum likelihood estimation for the discrete orbit recovery model}
\author{Zhou Fan}\email{zhou.fan@yale.edu}
\author{Yi Sun}\email{yisun@math.columbia.edu}
\author{Tianhao Wang}\email{tianhao.wang@yale.edu}
\author{Yihong Wu}\email{yihong.wu@yale.edu}
\date{\today}

\begin{document}

\begin{abstract}
We study the non-convex optimization landscape for maximum likelihood
estimation in the discrete orbit recovery model with Gaussian noise.
This is a statistical model motivated by applications in molecular microscopy
and image processing, where each measurement of an unknown object is subject to
an independent random rotation from a known rotational group. Equivalently, it
is a Gaussian mixture model where the mixture centers belong to a group orbit.

We show that fundamental properties of the likelihood landscape depend on the
signal-to-noise ratio and the group structure. At low noise, this landscape is
``benign'' for any discrete group, possessing no spurious local optima and
only strict saddle points. At high noise, this landscape may develop
spurious local optima, depending on the specific group. We discuss several
positive and negative examples, and provide a general condition that ensures a
globally benign landscape at high noise. For cyclic permutations of coordinates
on $\R^d$ (multi-reference alignment), there may be spurious local optima when
$d \geq 6$, and we establish a correspondence
between these local optima and those of a surrogate function of the phase variables
in the Fourier domain.

We show that the Fisher information matrix transitions from resembling that of
a single Gaussian distribution in low noise to having a graded eigenvalue structure
in high noise, which is determined by the graded algebra of invariant polynomials
under the group action. In a local neighborhood of the true object, where the
neighborhood size is independent of the signal-to-noise ratio, the landscape is
strongly convex in a reparametrized system of variables given by a transcendence
basis of this polynomial algebra. We discuss implications for optimization
algorithms, including slow convergence of expectation-maximization, and possible
advantages of momentum-based acceleration and variable reparametrization for 
first- and second-order descent methods.
\end{abstract}

\maketitle

\tableofcontents

\section{Introduction}

We study statistical estimation of a vector $\theta_* \in \R^d$
from noisy observations, where each observation is subject to a random and
unknown rotation. Letting $G \subseteq \rO(d)$ be a known subgroup of
orthogonal rotations in dimension $d$, we consider the observation model
\begin{equation}\label{eq:orbitmodel}
Y=g \cdot \theta_*+\sigma \eps.
\end{equation}
Here, $g \sim \Unif(G)$ is an unobserved uniform random element of this group,
$\sigma>0$ is the noise level, and $\eps \sim \N(0,\Id)$ is observation noise
that is independent of $g$. This model is sometimes referred to as
multi-reference alignment, the group action channel, or the orbit recovery
problem \cite{bandeira2017optimal,bandeira2017estimation,bendory2017bispectrum,abbe2018multireference,abbe2018estimation,boumal2018heterogeneous,perry2019sample,brunel2019learning}.

Study of this model has largely been motivated by its relevance to the structure
recovery problem arising in single-particle cryo-electron microscopy (cryo-EM)
\cite{dubochet1988cryo,henderson1990model,frank2006three}. Cryo-EM is an
experimental method of determining the 3D structure of a molecule by imaging
many cryogenic samples of the molecule from different and unknown viewing
angles. Due to limitations of electron dose, the individual images
are subject to high levels of measurement noise, and they must be
aligned and averaged to obtain a high-resolution reconstruction of
the molecule. There is extensive literature on computational methods
for this problem, and we refer readers to the recent surveys
\cite{bendory2020single,singer2020computational}.
In our work, we study the simpler model (\ref{eq:orbitmodel}), which omits
many complications in cryo-EM such as a tomographic
projection, the contrast-transfer function, 
and structural heterogeneity. We do this so as to focus our attention on some of the
fundamental features of this reconstruction problem that may
arise due to the latent rotation $g$.

It has been observed since \cite{sigworth1998maximum} that the difficulty of
estimation in the model (\ref{eq:orbitmodel}) has an atypically strong dependence on
the noise level $\sigma$, and this is a common theme in subsequent study
\cite{bandeira2017optimal,bandeira2017estimation,abbe2018estimation,perry2019sample}. Figure
\ref{fig:rotationsexample} contrasts a low-noise and high-noise setting
in a simple example, where $G$ is the group of three-fold
rotations on the plane $\R^2$. Three distinct clusters corresponding to the
orbit points $\{g\theta_*:g \in G\}$ are observed in low noise,
whereas only a single large cluster is apparent in high noise.
The number of samples needed to recover $\theta_*$ and the dependence of this
sample complexity on $\sigma$ were studied in
\cite{bandeira2017estimation,abbe2018estimation}.
In particular, \cite{bandeira2017estimation} showed that
method-of-moments estimators can achieve rate-optimal sample complexity
in $\sigma$, and connected this complexity to properties of the algebra of
$G$-invariant polynomials.

\begin{figure}
\begin{center}
\includegraphics[width=0.5\textwidth]{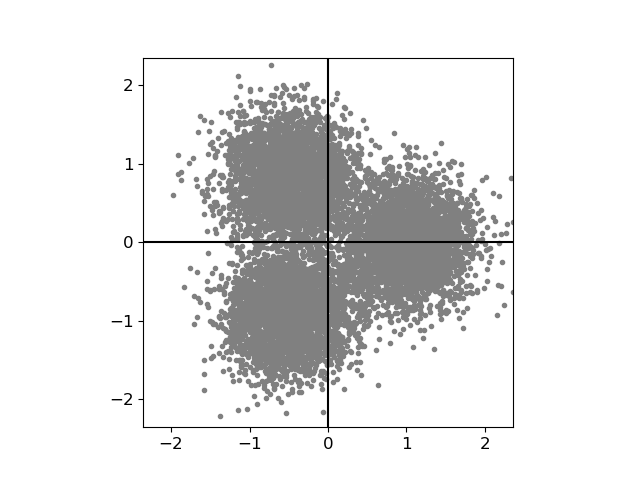}%
\includegraphics[width=0.5\textwidth]{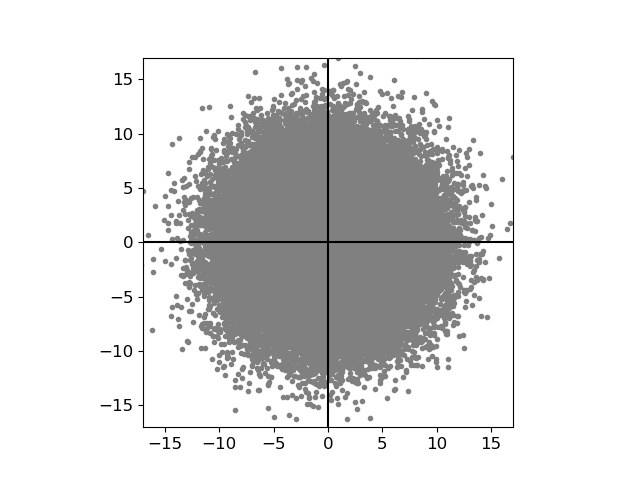}\\
\includegraphics[width=0.35\textwidth]{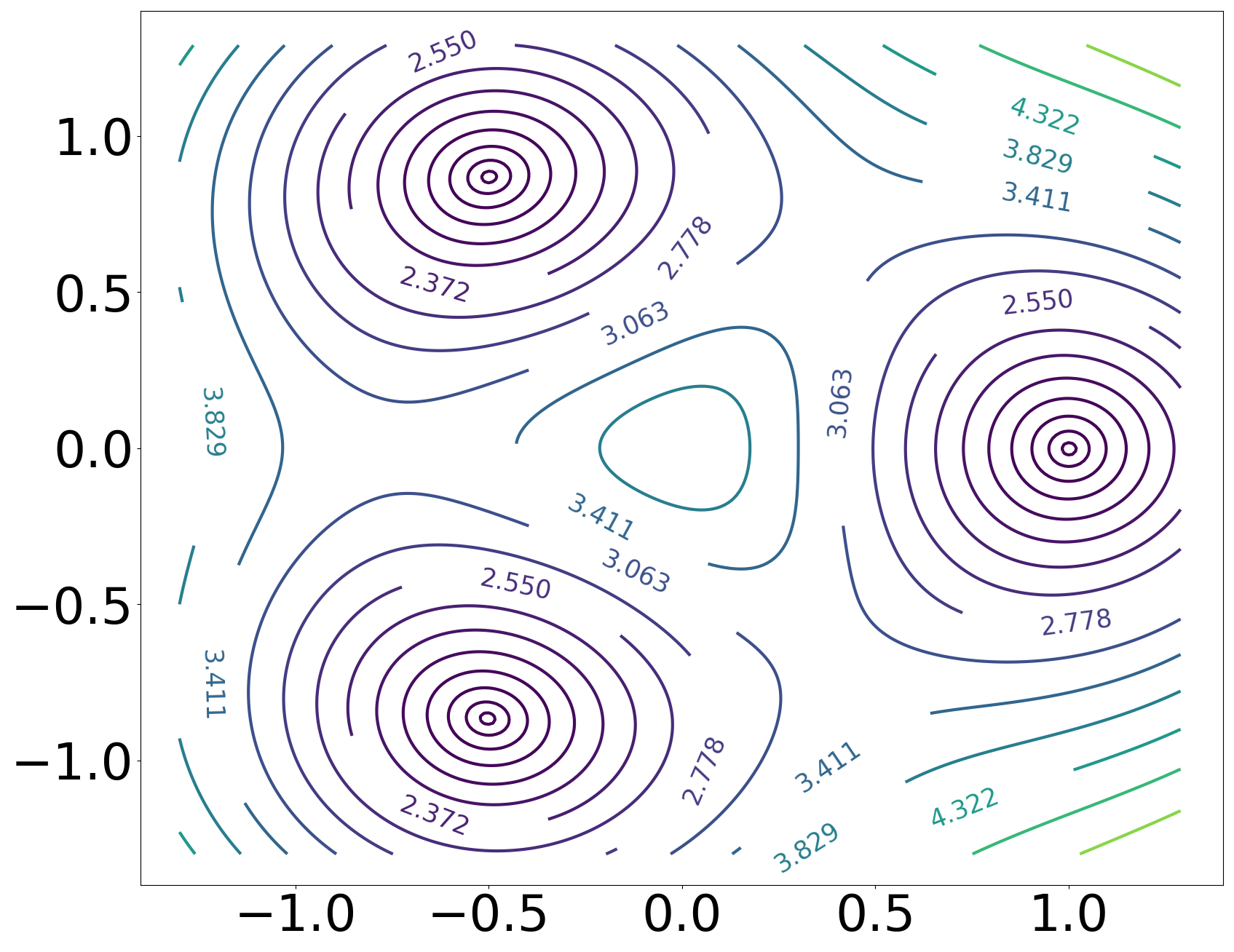}
\hspace{0.8in}
\includegraphics[width=0.35\textwidth]{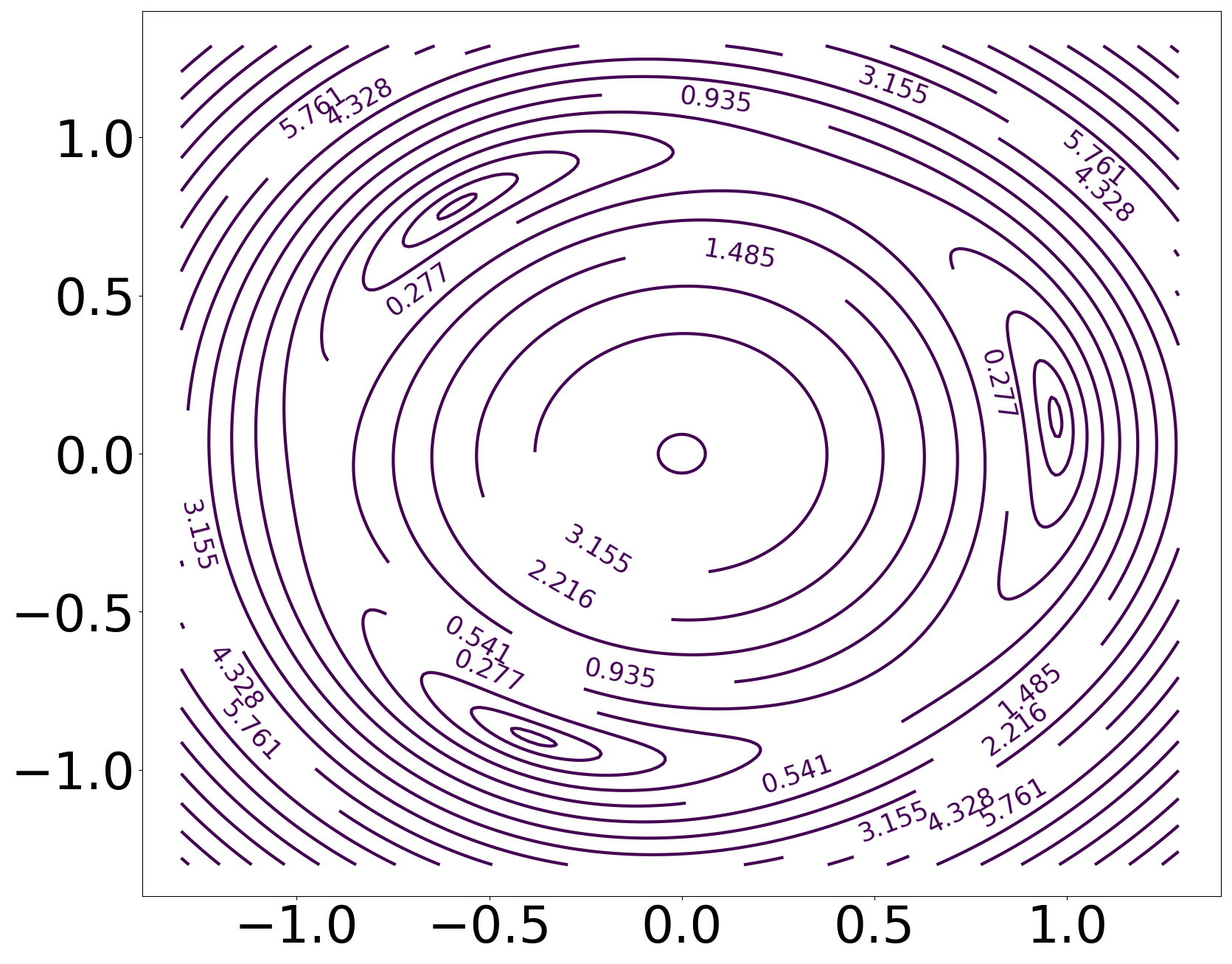}
\end{center}
\caption{Data samples and contours of negative log-likelihood $R_n(\theta)$ for
the group of three-fold rotations on $\R^2$, where $\theta_*=(1,0)$. Left: 10{,}000 samples at noise level $\sigma=0.4$. Right: 100{,}000 samples at noise level
$\sigma=4$. (Note the difference in axis limits between the data plots and
contour plots.) Values in the contour plots are displayed under an affine
transformation for better visualization.}
\label{fig:rotationsexample}
\end{figure}

The focus of our current work is, instead, on maximum likelihood estimation for 
$\theta_*$. Maximum likelihood is a widely used approach in practice, for either
\emph{ab initio} estimation of $\theta_*$ or for iterative refinement of a pilot
estimate obtained by other means
\cite{sigworth1998maximum,scheres2005maximum,scheres2007modeling,scheres2007disentangling}.
Letting $Y_1,\ldots,Y_n$ be i.i.d.\ observations
from the model (\ref{eq:orbitmodel}), the maximum
likelihood estimate (MLE) is a vector $\hat{\theta} \in \R^d$ which
maximizes the log-likelihood function
\[\theta \mapsto \frac{1}{n}\sum_{i=1}^n \log p_\theta(Y_i), \]
where $p_\theta$ is the probability density of $Y$
marginalizing over the latent rotation $g \sim \Unif(G)$.
We denote the negative log-likelihood function by $R_n(\theta)$;
this function is also depicted in
Figure \ref{fig:rotationsexample} for low and high noise.
The success of optimization algorithms for computing the MLE
for ab initio estimation and for iterative refinement depends,
respectively, on the global function landscape of
$R_n(\theta)$ and on its local landscape in a neighborhood of $\theta_*$.

In this work, we study the function landscape of $R_n(\theta)$, assuming that
the true vector $\theta_* \in \R^d$ is suitably generic. We restrict attention
to discrete groups $G$, so that $R_n(\theta)$ has isolated critical points, and
we derive several results. First, we show that the global landscape is ``benign''
for sufficiently low noise, having no spurious local minimizers
for any discrete group. Second, we show that the local landscape in a
$\sigma$-independent neighborhood of $\theta_*$ is also benign at any noise
level $\sigma>0$, and that $R_n(\theta)$ is strongly convex in this
neighborhood after suitable reparametrization. Third, we relate the critical
points of the global landscape in high noise to a sequence of simpler
optimization problems defined by the symmetric moment tensors under $G$.
We show that for discrete rotations in $\R^2$ as in Figure \ref{fig:rotationsexample}, and
for the symmetric group that permutes the coordinates of $\R^d$,
the global landscape is benign also at high noise. In contrast,
for the group of cyclic permutations in $\R^d$,
the global landscape may not be benign for even $d \geq 6$ and odd $d \geq 53$.

Our motivations for studying the MLE and the likelihood landscape are two-fold.
First, classical statistical theory indicates that in the limit $n \to \infty$ for
fixed dimension $d$, the MLE achieves
asymptotic efficiency, meaning that $\hat{\theta}$ converges to $\theta_*$ at an
$O(1/\sqrt{n})$ rate, with asymptotically optimal covariance $I(\theta_*)^{-1}$
(the inverse of the Fisher information matrix) 
matching the Cramer-Rao lower bound (see \cite[Sec.~2.5]{lehmann2006theory}). This need not hold for method-of-moments
estimators as studied in \cite{bandeira2017estimation}.
Our results connect one aspect of
\cite{bandeira2017estimation} regarding the sample complexity for ``list-recovery of
generic signals'' to the MLE, by showing that the eigenstructure of the Fisher information
matrix $I(\theta_*)$ corresponds to a
sequence of transcendence degrees in the graded algebra of $G$-invariant
polynomials.

Second, a body of empirical literature in cryo-EM suggests that $R_n(\theta)$
may have spurious local minimizers.
For ab initio estimation, this has motivated the development of a variety of
optimization algorithms including stochastic hill climbing
\cite{elmlund2013prime}, stochastic gradient descent
\cite{punjani2017cryosparc}, and ``frequency marching'' \cite{barnett2017rapid}.
However, at present, the function landscape of $R_n(\theta)$ is not
theoretically well-understood, even in simple examples of group actions.
For instance, it is unclear how this landscape depends on properties of the
group, and whether the roughness of this landscape is due to insufficient
sample size or is a fundamental aspect of the model even in the $n \to \infty$
limit. Our work takes a step towards understanding these questions,
and our results have concrete implications for descent-based optimization
algorithms in this problem. We discuss these implications
in Section \ref{sec:optimization} below.

\subsection{The orbit recovery model}\label{sec:model}

We study the orbit recovery model (\ref{eq:orbitmodel}) in the setting of a
discrete group. Let $G \subset \rO(d) \subset \R^{d \times d}$
be a discrete subgroup of the orthogonal group in dimension $d$, with finite
cardinality
\[|G|=K.\]
Each observation is modeled as
\[Y=g \cdot \theta_*+\sigma \eps\]
where $g \sim \Unif(G)$, $\eps \sim \N(0,\Id)$, and these are independent. Here,
$\sigma>0$ is the noise level, which we will assume is known.
This is a $K$-component Gaussian mixture model with equal weights, where the centers of the
mixture components are the points of the \emph{orbit} of $\theta_*$ under $G$,
given by
\[\O_{\theta_*}=\{g\theta_*:g \in G\}.\]
The marginal density of $Y$ in this model is the Gaussian mixture density
\begin{equation}\label{eq:mixturedensity}
p_{\theta_*}(Y)=\frac{1}{K}\sum_{g \in G}
\left(\frac{1}{\sqrt{2\pi \sigma^2}}\right)^d
\exp\left(-\frac{\|Y-g\theta_*\|^2}{2\sigma^2}\right).
\end{equation}
For $\theta,\theta' \in \R^d$, note that $p_\theta=p_{\theta'}$
if and only if the $K$ mixture components have the same centers,
i.e. $\O_{\theta'}=\O_\theta$. This means the parameter $\theta_*$ is
statistically identifiable in this model up to its orbit.

Given $n$ independent samples $Y_1,\ldots,Y_n$ distributed
according to (\ref{eq:orbitmodel}), we study the landscape of the
negative log-likelihood \emph{empirical risk}
\begin{equation}\label{eq:empiricalrisk}
R_n(\theta)=-\frac{1}{n}\sum_{i=1}^n \log p_\theta(Y_i)+\const.
\end{equation}
Here, $\const$ denotes a $\theta$-independent value that we
introduce to simplify
the expression for this risk; see (\ref{eq:Rn}) for details.
Our results will apply equally to a setting where the true group
element $g$ in (\ref{eq:orbitmodel}) is not uniform, and we discuss this in
Remark \ref{remark:nonuniform}.

This function $R_n(\theta)$ is non-convex for any non-trivial group $G$.
A maximum likelihood estimator $\hat{\theta} \in \R^d$ is any global
minimizer of $R_n(\theta)$. Note that if $\hat{\theta}$
minimizes $R_n(\theta)$, then all points in its orbit $\O_{\hat{\theta}}$ also
minimize $R_n(\theta)$, so the MLE is also only defined up to its orbit.

Fixing the true parameter $\theta_* \in \R^d$, we denote the mean of $R_n(\theta)$ by
\begin{equation}\label{eq:populationrisk}
R(\theta)=-\E\big[\log p_\theta(Y)\big]+\const,
\end{equation}
where $\E$ is the expectation over both $g$ and $\eps$ in the model
$Y=g \cdot \theta_*+\sigma \eps$. This function $R(\theta)$
depends implicitly on the true parameter
$\theta_*$. We call $R(\theta)$ the \emph{population risk}, and this
may be understood as the $n \to \infty$ limit of $R_n(\theta)$. Note that
\begin{equation}\label{eq:KLdivergence}
R(\theta)=D_{\text{KL}}(p_{\theta_*} \| p_\theta)-\E[\log
p_{\theta_*}(Y)]+\const
\end{equation}
where $D_{\text{KL}}(p\|q)=\int p(y) \log \frac{p(y)}{q(y)} dy$ is the Kullback-Leibler divergence between densities $p$ and $q$,
and the remaining two terms do not depend on $\theta$. Thus, a point $\theta \in \R^d$
is a global minimizer of $R(\theta)$ if and only if
$p_{\theta_*}=p_\theta$, i.e.\ $\theta \in \O_{\theta_*}$.

It was established in
\cite{mei2018landscape} that under mild conditions for empirical risks such as
(\ref{eq:empiricalrisk}), due to concentration of the gradient and
Hessian of $R_n(\theta)$ around those of $R(\theta)$, various
properties of the function landscape of $R(\theta)$ translate to those of
$R_n(\theta)$ for sufficiently large $n$---these
properties include the number of critical points and the number of negative
Hessian eigenvalues at each critical point. Versions of this argument were also
used in the analyses of dictionary learning and phase retrieval in
\cite{sun2016complete,sun2018geometric}. Our analysis will follow a similar
approach, and the core of our arguments will pertain to the population risk
(\ref{eq:populationrisk}) rather than its finite-$n$ counterpart
(\ref{eq:empiricalrisk}).

We will also study properties of the Fisher information matrix in this model.
This is given by
\begin{equation}\label{eq:fisherinformation}
I(\theta_*)=-\E\big[\nabla_\theta^2 \log
p_\theta(Y)\big|_{\theta=\theta_*}\big]
=\nabla_\theta^2 R(\theta_*),
\end{equation}
which is the Hessian of the population risk $R(\theta)$
evaluated at its global minimizer $\theta=\theta_*$. It was shown in
\cite{brunel2019learning} that $I(\theta_*)$ is invertible if and only if all
$K$ points of the orbit $\O_{\theta_*}$ are distinct. We assume this
condition in all of our results, and some of our results will further
restrict $\theta_*$ to satisfy additional \emph{generic}
properties that hold outside the zero set of an analytic function
on $\R^d$. Identifying the MLE $\hat{\theta}$ as the point
in its orbit closest to $\theta_*$, \cite{abbe2018estimation} verified that
$\hat{\theta}$ is an asymptotically consistent estimate for $\theta_*$ as
$n \to \infty$. By the classical theory of maximum likelihood estimation in
parametric models (see \cite[Chapter 5]{van2000asymptotic}),
we then have the convergence in law
\begin{equation}\label{eq:MLEnormal}
\sqrt{n}(\hat{\theta}-\theta_*) \to \N\big(0,I(\theta_*)^{-1}\big).
\end{equation}
Thus the eigenvalues of the Fisher
information matrix determine the coordinate-wise asymptotic variances of the 
MLE in an orthogonal basis for $\R^d$.

\subsection{Overview of results}\label{sec:results}

We will be interested in the geometric properties of the function landscapes of
$R_n(\theta)$ and $R(\theta)$. The most ideal setting for non-convex
optimization is when these landscapes are benign in the following sense.

\begin{definition}
The landscape of a twice continuously-differentiable function $f:\R^d \to \R$
is {\bf globally benign} if the only local minimizers of $f$ are global
minimizers, $f$ is strongly convex at each such local minimizer, and each
saddle point of $f$ is a strict saddle point.
\end{definition}

This is equivalent to saying that the only points $\theta \in \R^d$
where $\nabla f(\theta)=0$ and $\lambda_{\min}(\nabla^2 f(\theta)) \geq 0$ are
the global minimizers of $f$, and $\lambda_{\min}(\nabla^2 f(\theta))>0$
strictly at all such points. This condition has been
discussed in \cite{ge2015escaping,lee2016gradient,jin2017escape}, which show
that randomly-initialized gradient descent converges to a global minimizer
almost surely under this condition, and that gradient descent perturbed with
additive noise can furthermore converge in polynomial time under a
quantitative version of this condition.

In our results, we will fix a generic true parameter $\theta_* \in \R^d$. We study 
low-noise and high-noise regimes, where the low-noise regime is defined by
$\sigma<\sigma_0$ for a sufficiently small $(\theta_*,d,G)$-dependent constant
$\sigma_0>0$, and the high-noise regime by $\sigma>\sigma_0$ for a (different)
sufficiently large $(\theta_*,d,G)$-dependent constant $\sigma_0>0$. It is the
high-noise regime that is of primary interest in applications such as cryo-EM. We
provide results also for low noise, to contrast with the high-noise behavior,
and because these results may be of separate interest in other applications.\\

\noindent {\bf Global landscape and Fisher information at low noise.} We show
in Section \ref{sec:lownoise}
that both $R(\theta)$ and $R_n(\theta)$ are globally benign in the low noise
regime, for any discrete group $G$, any $\theta_*$ whose
orbit points are distinct under $G$, and sufficiently large sample size
$n$. That is, there exists $\sigma_0 \equiv
\sigma_0(\theta_*,d,G)$ for which $R(\theta)$ and $R_n(\theta)$ do not have any
spurious local minimizers when $\sigma<\sigma_0$.

We also show that the Fisher information satisfies $I(\theta_*) \approx
\sigma^{-2}\Id$, where the error of this approximation is exponentially small
in $\sigma^{-2}$. Here, $\sigma^{-2}\Id$ is the Fisher
information of the single Gaussian distribution $\N(\theta_*,\sigma^2\Id)$. Thus
the local geometries of $R(\theta)$ and $R_n(\theta)$ near $\theta_*$ resemble
those of a single Gaussian, and they do not ``feel'' the effects of the other
mixture components.

We remark that the group structure plays an important role in our proof of this
global landscape result, and such a result is not true for general Gaussian mixture models:
For the three-component Gaussian mixture model
\[\frac{1}{3}\N(\theta_1,\sigma^2\Id)+\frac{1}{3}\N(\theta_2,\sigma^2\Id)+\frac{1}{3}\N(\theta_3,\sigma^2\Id),\]
it is known that the negative log-likelihood population risk
as a function of $(\theta_1,\theta_2,\theta_3) \in \R^{3d}$ can have spurious local\
minimizers, even in the $\sigma \to 0$ limit. Similar examples may be constructed for
any number of mixture components $K \geq 3$ \cite{jin2016local}.\\

\noindent {\bf Fisher information at high noise.} As the noise level $\sigma$
increases, a transition
occurs in the structure of the Fisher information matrix $I(\theta_*)$.
We show in Section \ref{sec:locallandscape}
that in the high-noise regime, for any generic $\theta_* \in
\R^d$, there is a decomposition $d=d_1+d_2+\ldots+d_L$ where
\begin{equation}\label{eq:FIinformal}
I(\theta_*) \text{ has } d_\ell \text{ eigenvalues on the order of }
\sigma^{-2\ell} \text{ for each } \ell=1,\ldots,L.
\end{equation}
The number $d_\ell$ is $\trdeg(\cR^G_{\leq \ell})-\trdeg(\cR^G_{\leq \ell-1})$,
where $\trdeg(\cR^G_{\leq \ell})$ is the transcendence degree over $\R$ of the
space of $G$-invariant polynomials having degree $\leq \ell$.
The number $L$ is the smallest integer for which $\trdeg(\cR^G_{\leq L})=d$.

For the group of $K$-fold discrete rotations in $\R^2$, as
in Figure \ref{fig:rotationsexample}, we have $L=K$, $d_2=1$, $d_K=1$, and
$d_\ell=0$ for each other $\ell$.
Thus $I(\theta_*)$ has one eigenvalue of magnitude $\sigma^{-4}$,
corresponding to the curvature of $R(\theta)$ in the radial direction, and one
eigenvalue of magnitude $\sigma^{-2K}$, corresponding to the
direction tangent to the circle
$\{\theta \in \R^2:\|\theta\|=\|\theta_*\|\}$. For the symmetric
group of all permutations in $\R^d$, we have $L=d$ and $d_\ell=1$ for each
$\ell=1,\ldots,d$. For cyclic permutations in $\R^d$, we have $L=3$,
$d_1=1$, $d_2=\lceil \frac{d-1}{2} \rceil$, and $d_3=\lfloor \frac{d-1}{2}
\rfloor$. Here $d_1$ corresponds to the sum $\theta_1+\ldots+\theta_d$, $d_2$ to
the magnitudes of the remaining Fourier coefficients of $\theta$, and $d_3$ to
the phases.

Applying (\ref{eq:FIinformal})
to the classical efficiency result (\ref{eq:MLEnormal}) for the
MLE, this shows that $\hat{\theta}$ estimates $\theta_*$ with an asymptotic
covariance of $O(\sigma^{2L}/n)$. This rate agrees with the results of
\cite{bandeira2017estimation} on list-recovery of generic signals $\theta_*$
by a method-of-moments estimator. More precisely, (\ref{eq:FIinformal})
exhibits a decomposition of $\R^d$ into orthogonal subspaces of
dimensions $d_1,\ldots,d_L$, such that the MLE $\hat{\theta}$ estimates
$\theta_*$ with an asymptotic covariance of $O(\sigma^{2\ell}/n)$ in its
component belonging to the $\ell^\text{th}$ subspace. For any continuously
differentiable function $\psi:\R^d \to \R$, a Taylor expansion of $\psi$ (i.e.\
the statistical delta method) yields also the convergence in law
\begin{equation}\label{eq:deltamethod}
\sqrt{n}\Big(\psi(\hat{\theta})-\psi(\theta_*)\Big)
\to \N\Big(0,\nabla \psi(\theta_*)^\top I(\theta_*)^{-1} \nabla \psi(\theta_*)\Big)
\end{equation}
as $n \to \infty$. We show that when $\psi$ is any $G$-invariant polynomial of
degree $\ell$, the gradient $\nabla \psi(\theta_*)$ belongs to the span of the
first $\ell$ subspaces, so that $\psi(\hat{\theta})$
estimates $\psi(\theta_*)$ with variance $O(\sigma^{2\ell}/n)$.\\

\noindent {\bf Global landscape at high noise.}
Denote by
\begin{equation}\label{eq:momenttensor}
T_\ell(\theta)=\E_g[(g\theta)^{\otimes \ell}] \in (\R^d)^{\otimes \ell}
\end{equation}
the $\ell^\text{th}$ moment tensor of $g\theta$, where $\E_g$ is the expectation
over the uniform law $g \sim \Unif(G)$. The entries of $T_\ell(\theta)$ consist
of all order-$\ell$ mixed moments of entries of the random vector $g\theta \in
\R^d$. Let $\|\cdot\|_\HS$ be the Euclidean
norm of the vectorization of such a tensor in $\R^{d^\ell}$.
We relate the local minimizers of $R(\theta)$ and $R_n(\theta)$
in the high-noise regime to a
sequence of simpler optimization problems, given by successively minimizing
\begin{equation}\label{eq:momentobjective}
P_\ell(\theta)=\|T_\ell(\theta)-T_\ell(\theta_*)\|_\HS^2
\end{equation}
over the variety
\begin{equation}\label{eq:momentvariety}
\cV_{\ell-1}=\Big\{\theta \in \R^d:T_k(\theta)=T_k(\theta_*) \text{ for } k=1,\ldots,\ell-1\Big\},
\end{equation}
for $\ell=1,\ldots,L$. This sequence of optimization problems is related to the
method-of-moments, in that (\ref{eq:momentobjective}) may be interpreted as
matching the order-$\ell$ moments $T_\ell(\theta)$ to $T_\ell(\theta_*)$,
subject to the constraint (\ref{eq:momentvariety}) that the moments of
lower order have already been matched.

We show in Section \ref{sec:globalbenign} that for
generic $\theta_*$, if $\cV_L=\O_{\theta_*}$, each variety
$\cV_\ell$ is non-singular with constant dimension, each restriction
$P_\ell|_{\cV_{\ell-1}}$ satisfies a strict saddle condition,
and the only local minimizers of each restriction $P_\ell|_{\cV_{\ell-1}}$ are
the points $\theta \in \cV_\ell$,
then the global landscape of $R(\theta)$ is also benign in the high-noise
regime. In such examples, the landscape of the empirical risk
$R_n(\theta)$ is then also globally benign with high probability when $n
\gg \sigma^{2L}$. This requirement for $n$ matches the sample complexity for
recovery of generic signals in \cite{bandeira2017estimation}.
We analyze the two concrete examples of $K$-fold rotations in
$\R^2$ and the symmetric group of all permutations in $\R^d$, showing that
the global landscape is benign at high noise in these examples.

The first condition $\cV_L=\O_{\theta_*}$ means that $\theta_*$ is uniquely
specified, up to its orbit, by its first $L$ moment tensors
$T_1(\theta_*),\ldots,T_L(\theta_*)$. These are the examples in
\cite{bandeira2017estimation} where the notions of ``generic list recovery''
and ``generic unique recovery'' coincide. We note that this condition alone is
not sufficient to guarantee a benign landscape. For instance, in the cyclic
permutations example below, we have $L=3$ and $\cV_3=\O_{\theta_*}$ 
for generic points $\theta_* \in \R^d$ in any dimension $d$, but spurious local
minima may exist.\\

\noindent {\bf Spurious local minimizers for cyclic permutations.}
The complexity of the sequence of optimization problems in
(\ref{eq:momentobjective}--\ref{eq:momentvariety}) depends on the
structure of the $G$-invariant polynomial algebra. As a more complex example,
we study in Section \ref{sec:MRA}
the group $G$ of cyclic permutations in $\R^d$. Some authors refer
to this specific action as the multi-reference alignment (MRA) model, and the
invariant polynomial algebra for this group
bears some similarities to the continuous action of $\mathrm{SO}(3)$ that is relevant
for cryo-EM applications
\cite{bandeira2017optimal,bandeira2017estimation,perry2019sample}.

For this group, we have $L=3$, and $P_\ell(\theta)$ does not have spurious local
minimizers over $\cV_{\ell-1}$ for $\ell=1$ and 2. For $\ell=3$ and odd $d$,
denoting $\cI=\{1,2,\ldots,\frac{d-1}{2}\}$, we show in Theorem \ref{thm:MRA} that
minimizing $P_3(\theta)$ over $\cV_2$ is equivalent to minimizing
\[F^+(t_1,\ldots,t_{|\cI|})
=-\frac{1}{6}\mathop{\sum_{i,j,k \in \cI \cup -\cI}}_{i+j+k \equiv 0 \bmod d}
r_{i,*}^2r_{j,*}^2r_{k,*}^2\cos(t_i+t_j+t_k)\]
over phase variables $t_1,\ldots,t_{|\cI|} \in [0,2\pi)$,
where we identify $t_{-i}=-t_i$ and set $r_{i,*}$ as the modulus of the
$i^\text{th}$ Fourier coefficient of $\theta_*$. When $d$ is even,
there is an additional term to this function as well as a second function
$F^-(t_1,\ldots,t_{|\cI|})$, and we refer to Section \ref{sec:MRA} for details.

We show that for high noise and generic $\theta_* \in \R^d$,
local minimizers of $R(\theta)$ are in
correspondence with local minimizers of $F^\pm(t_1,\ldots,t_{|\cI|})$, where the
magnitudes of the Fourier coefficients of any such local minimizer $\theta \in
\R^d$ are close to those of $\theta_*$, and the differences in phases between
the Fourier coefficients of $\theta$ and those of $\theta_*$ are
close to the corresponding local minimizer of $F^{\pm}$. In dimensions $d
\leq 5$, there are no spurious local minimizers, and the landscapes
of $R(\theta)$ and $R_n(\theta)$ are globally benign. In even dimensions $d\geq 6$
and odd dimensions $d \geq 53$, we exhibit an open set $U \subset \R^d$ such
that $R(\theta)$ and $R_n(\theta)$ \emph{do} have spurious local minimizers,
for all $\theta_* \in U$.
This is a phenomenon of the population risk $R(\theta)$ and is not caused
by finite-$n$ behavior, so descent procedures may converge to these
spurious local minimizers even in the limit of infinite sample size.
(We have found via a computer search that spurious local minimizers may exist
for odd $d \geq 13$, but we will not attempt to make this rigorous.)

In the method-of-moments approach to MRA, the Fourier magnitudes of
$\theta$ are recovered from the \emph{power spectrum}, or the set of degree-$2$
polynomial invariants, and the Fourier phases are recovered from
certain degree-$3$ polynomial invariants known as the \emph{bispectrum}. The
above surrogate functions
$F^\pm(t_1, \ldots, t_{|\cI|})$ are functions of the bispectrum, and it
may be checked that they are examples of the non-convex 
bispectrum inversion objective in
\cite[Equation (III.4)]{bendory2017bispectrum}.
The spurious local minima that we exhibit for even $d \geq 6$ correspond to
the local minima also identified in \cite[Page 17]{bendory2017bispectrum}.
The spurious local minima for odd $d$ form a new family, which
demonstrates also that the objective in \cite{bendory2017bispectrum} may not be
globally benign in such settings.\\

\noindent {\bf Local landscape at high noise.} Motivated by the possibility that
$R(\theta)$ and $R_n(\theta)$ are not globally benign, we study also their
local landscapes restricted to a smaller neighborhood of $\theta_*$ in Section
\ref{sec:locallandscape}. We show that there is a \emph{$\sigma$-independent}
neighborhood $U$ of $\theta_*$, and a local reparametrization by an
analytic map $\varphi:\R^d \to \R^d$ that is 1-to-1 on $U$, such that $R$ and
$R_n$ are strongly convex as functions of $\varphi \in \varphi(U)$, with
unique local minimizers in $U$. The coordinates of this map $\varphi$ may be
taken to be $d$ polynomials that form a transcendence basis of the $G$-invariant
polynomial algebra.

We remark that this result does not automatically follow from the
invertibility of the Fisher information $I(\theta_*)$ established in
\cite{brunel2019learning}, as this
invertibility does not preclude the possibility that the size of this
neighborhood $U$ shrinks as $\sigma \to \infty$. In fact, it is not
true that $R(\theta)$ must be convex over $\theta \in U$ for
a $\sigma$-independent neighborhood $U$, and the
reparametrization by $\varphi$ is important to ensure convexity.
For instance, in the high-noise picture of
Figure \ref{fig:rotationsexample}, it is evident from the non-convex level sets
that $R_n(\theta)$ is convex only in a small neighborhood of $\theta_*$.
However, it is convex in a much larger neighborhood of $\theta_*$ when
reparametrized by two coordinates that represent the radius and angle.\\

\noindent {\bf High-noise expansion of the population risk.} Our
results in the high-noise regime are enabled by a series expansion
of the population risk function in $\sigma^{-2}$, given by
\[R(\theta)=\sum_{\ell=1}^\infty \sigma^{-2\ell}S_\ell(\theta)\]
for certain $G$-invariant polynomial functions $S_\ell(\theta)$. For fixed
$\theta_* \in \R^d$, each polynomial $S_\ell(\theta)$ takes the form
\[S_\ell(\theta)=\frac{1}{2(\ell!)}\|T_\ell(\theta)-T_\ell(\theta_*)\|_\HS^2
+Q_\ell(\theta)\]
where $Q_\ell(\theta)$ is in the algebra generated by $G$-invariant polynomials
of degree $\leq \ell-1$. We derive these results and provide a rigorous
interpretation of this expansion in Section \ref{sec:seriesexpansion}.

By the relation (\ref{eq:KLdivergence}), this is equivalent to a series
expansion of the KL-divergence $D_{\text{KL}}(p_{\theta_*}\|p_\theta)$ in
$\sigma^{-2}$. In the works
\cite{bandeira2017optimal,bandeira2017estimation,abbe2018estimation}, analogous
expansions were performed instead for upper and lower bounds to the
KL-divergence, and these were
then used to study the sample complexity of estimating $\theta_*$. To
study the log-likelihood landscape, we must perform this expansion for $R(\theta)$
itself. Our proof of this series expansion does not require $G$ to be discrete
(or $\theta_*$ to be generic), and this result may be used also to study
continuous group actions. Following the initial posting of this work, this
series expansion has recently been extended to more general high-noise Gaussian
mixture models in \cite{katsevich2020likelihood}.

\subsection{Implications for optimization}\label{sec:optimization}

In this section, we discuss some implications of our
results for descent-based optimization algorithms in high-noise settings.\\

\noindent {\bf Slow convergence of expectation-maximization.} One of the most
widely used optimization algorithms for minimizing $R_n(\theta)$
is expectation-maximization (EM) (see \cite{DLR1977}, and
\cite{sigworth1998maximum,SIGWORTH2010263,bendory2020single} for applications in
cryo-EM). Starting from an
initialization $\theta^{(0)} \in \R^d$, the EM algorithm iteratively computes
\[\theta^{(t+1)}=\arg\min_{\theta \in \R^d} Q(\theta \mid \theta^{(t)})\]
where
\[Q(\theta \mid \theta^{(t)})=-\frac{1}{n}\sum_{i=1}^n
\E_{g \mid Y_i,\theta^{(t)}}\left[\log
\left(\left(\frac{1}{\sqrt{2\pi\sigma^2}}\right)^d\exp\left(
-\frac{\|Y_i-g \theta\|^2}{2\sigma^2}\right)\right)\right]\]
is the expectation of the full-data negative log-likelihood
over the posterior law of $g \in G$. For each sample $Y_i$, the density of this posterior law is
\[p(g \mid Y_i,\theta^{(t)})=\exp\left(-\frac{\|Y_i-g
\theta^{(t)}\|^2}{2\sigma^2}\right) \Bigg/ \sum_{h \in G}
\exp\left(-\frac{\|Y_i-h \theta^{(t)}\|^2}{2\sigma^2}\right),\]
leading to the following explicit form of the EM iteration:
\[\theta^{(t+1)}=\frac{1}{n}\sum_{i=1}^n \E_{g \mid Y_i,\theta^{(t)}}[g^\top
Y_i].\]
It is straightforward to verify that this is equivalent to the gradient descent (GD)
update
\[\theta^{(t+1)}=\theta^{(t)}-\eta \cdot \nabla R_n(\theta^{(t)})\]
with a fixed step size $\eta=\sigma^2$.

Our results indicate that in the high-noise regime, this step size
$\eta=\sigma^2$ corresponding to EM may not be
correctly tuned for optimal convergence. For applying
GD to a smooth and strongly convex function $f(\theta)$ where
\[\alpha\Id \preceq \nabla^2 f(\theta) \preceq \beta\Id,\]
the optimal step size is $\eta \asymp 1/\beta$, and GD with this step size
achieves a convergence rate
\begin{equation}\label{eq:GDrate}
\|\theta^{(t)}-\theta^{(0)}\|^2 \leq O\Big((1-c\alpha/\beta)^t\Big)
\end{equation}
for a constant $c>0$
(see \cite[Theorem 2.1.14]{nesterov2013introductory}). For any
mean-zero group $G$, we have (by Lemma \ref{lemma:Smeanzero}) that $d_1=0$
in the decomposition $d=d_1+\ldots+d_L$ in \eqref{eq:FIinformal}, so that 
$\lambda_{\max}(\nabla^2 R_n(\theta)) \asymp \sigma^{-4}$ locally near
$\theta_*$. Thus there is a flattening of the landscape near $\theta_*$, and
GD should instead be tuned with the larger step size $\eta \asymp \sigma^4$ after reaching a
small enough neighborhood of $\theta_*$.

\begin{figure}
\begin{center}
  \begin{subfigure}{.43\textwidth}
    \includegraphics[width=\textwidth]{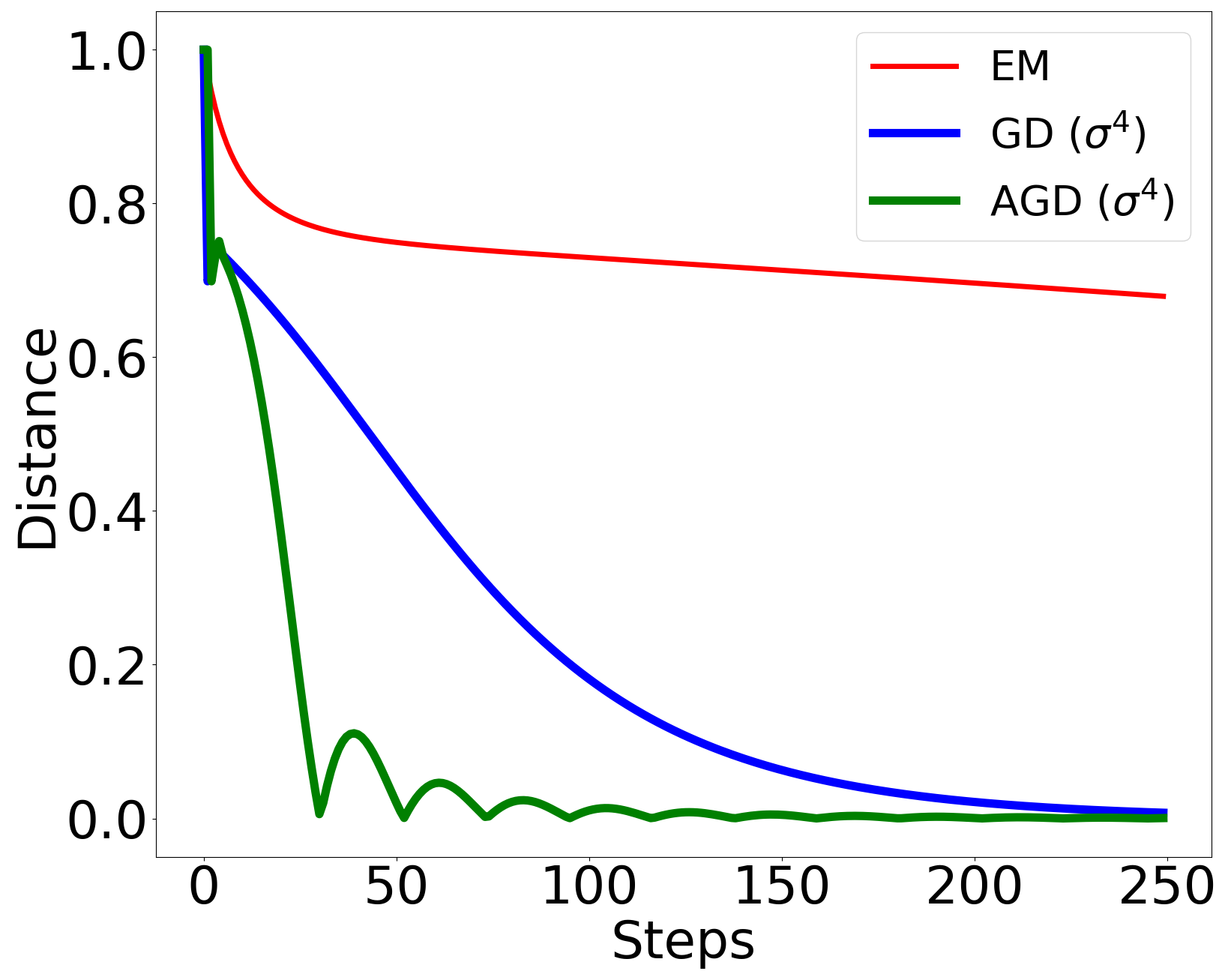}
    \caption{Distances
$\dist(\theta^{(t)},\O_{\theta_*})$ to the orbit of the true parameter
$\theta_*=(1,0)$, for 250 iterates
$\theta^{(1)},\ldots,\theta^{(250)}$ of each algorithm.}
  \end{subfigure}
\hspace{0.2cm}
    \begin{subfigure}{.46\textwidth}
      \includegraphics[width=\linewidth]{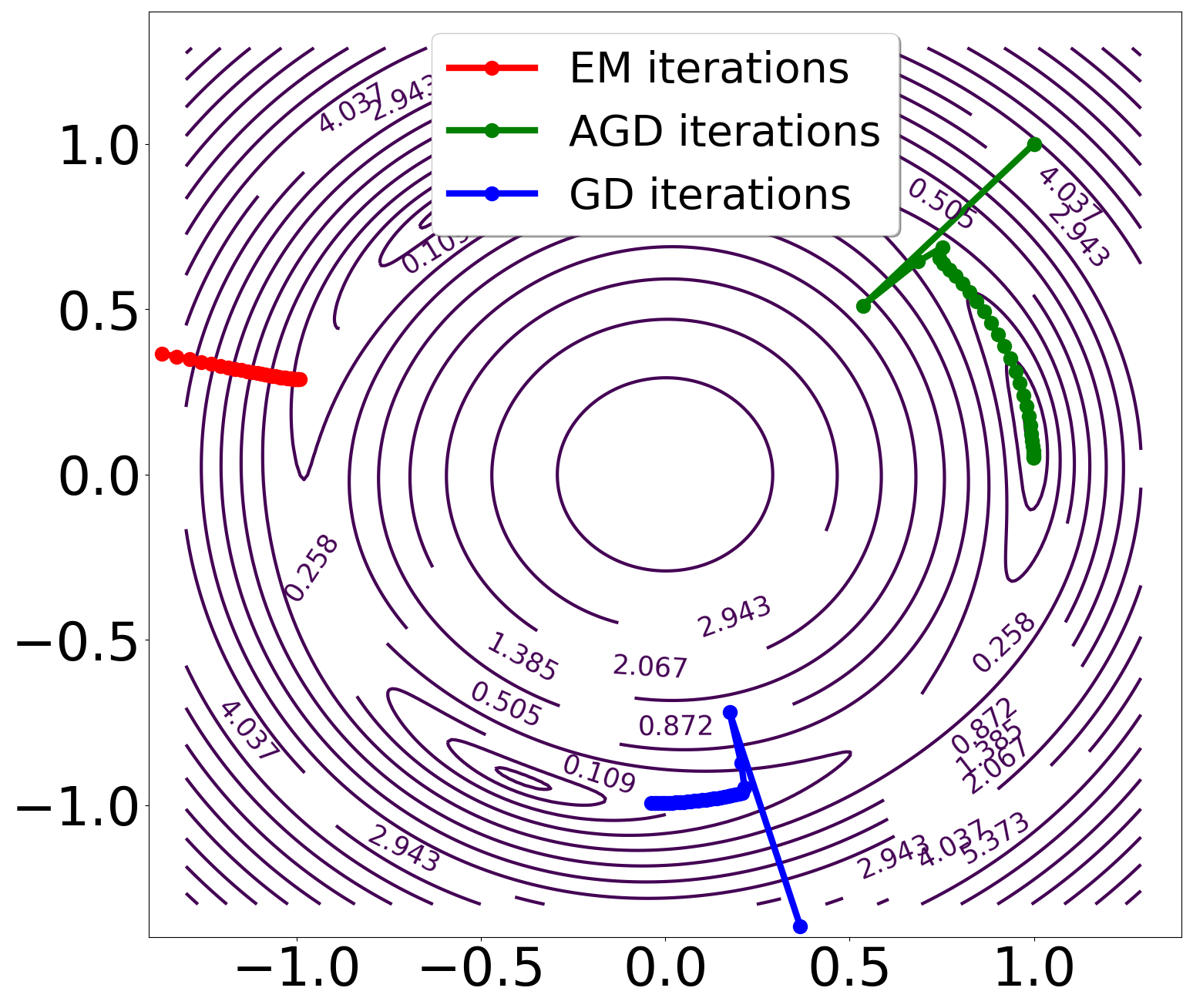}
      \caption{First 30 iterates
for each algorithm, depicted on the contour plot of the negative log-likelihood
function $R_n(\theta)$. Iterates for EM and GD are rotated by angles of $2\pi/3$ and
$4\pi/3$ for easier visualization.}
    \end{subfigure}
\caption{Convergence of expectation-maximization (EM), gradient descent (GD)
with step size $\eta=\sigma^4$, and Nesterov-accelerated gradient descent (AGD)
with step size $\eta=\sigma^4$ on the three-fold rotations example with
$n=100{,}000$ samples and noise level $\sigma=4$. All three algorithms are
initialized at $\theta^{(0)}=(1,1)$.}\label{fig:GDcompare}
\end{center}
\end{figure}

Figure \ref{fig:GDcompare} illustrates this for three-fold rotations in $\R^2$,
comparing 250 iterations of EM
versus GD with step size $\eta=\sigma^4$ on the high-noise example of Figure
\ref{fig:rotationsexample}.
EM converges quite slowly after reaching a vicinity of the circle $\{\theta \in
\R^2:\|\theta\|=\|\theta_*\|\}$, and the improved convergence rate
for step size $\eta=\sigma^4$ is apparent.\\

\noindent {\bf Nesterov acceleration for gradient descent.} The structure
(\ref{eq:FIinformal}) for the eigenvalues of $I(\theta_*)$ also indicates that
the Hessians of the risk functions $R_n(\theta)$ and $R(\theta)$ may be highly
anisotropic and ill-conditioned near $\theta_*$ in high-noise settings. This
poses a known problem for the convergence of gradient descent with any fixed
step size, including EM, as evident from the factor $\alpha/\beta$ in (\ref{eq:GDrate}).

This also suggests that substantial improvements in convergence may be
obtained by using momentum or acceleration methods
\cite{polyak1964some,nesterov2013introductory}. For example,
using the Nesterov acceleration scheme
\begin{align*}
\mu^{(t+1)}&=\theta^{(t)}-\eta \cdot \nabla R_n(\theta^{(t)})\\
\theta^{(t+1)}&=(1+\tau)\mu^{(t+1)}-\tau\mu^{(t)},
\end{align*}
accelerated gradient descent (AGD) can achieve the improved convergence rate
\begin{equation}\label{eq:AGDrate}
\|\theta^{(t)}-\theta^{(0)}\|^2 \leq O\Big((1-c\sqrt{\alpha/\beta})^t\Big),
\end{equation}
see \cite[Theorem 2.2.3]{nesterov2013introductory}.
Figure \ref{fig:GDcompare} also illustrates the convergence of AGD on the same
three-fold rotations example, with step size $\eta=\sigma^4$ and momentum
parameters $\tau \equiv \tau_t$ defined as (see \cite[Section 3.7.2]{bubeck2015convex})
\begin{align*}
\lambda_0=0, \quad \lambda_t=\Big(1+\sqrt{1+4\lambda_{t-1}^2}\Big)\Big/2, \quad
\tau_t=(\lambda_t-1)/\lambda_{t+1}.
\end{align*}
The iterates $\theta^{(t)}$ reach the orbit
$\O_{\theta_*}$ within 30 iterations of AGD, when neither EM nor standard GD
with $\eta=\sigma^4$ is close to having converged.\\

\noindent {\bf Reparametrization for second-order trust region methods.}
Second-order descent procedures may also be applied to minimize $R_n(\theta)$. Since $R_n$
is non-convex, it is possible for its second-order
approximation at an iterate $\theta^{(t)}$ to have a direction
of negative curvature. When this occurs, it is common to apply a trust-region
approach, where the next update $\theta^{(t+1)}$ is constrained to lie within a
fixed-radius ball around
$\theta^{(t)}$
\cite{sun2015nonconvex,sun2016complete,sun2018geometric,mei2018landscape}. This
trust region is used until the iterates $\theta^{(t)}$ reach a neighborhood of
strong convexity around a local minimizer of $R_n(\theta)$, after
which the algorithm naturally transitions to a standard second-order Newton
method for minimizing convex objectives.

At high noise, the region of convexity for $R(\theta)$ and $R_n(\theta)$
around $\theta_*$ may be vanishingly small in $\sigma$, requiring more careful tuning
of this trust-region algorithm and a large number of iterations before reaching
this convex region. However, as mentioned in Section~\ref{sec:results}, our results indicate that
the region of convexity is much larger, and is $\sigma$-independent,
upon reparametrizing by $G$-invariant coordinates $\varphi \equiv
\varphi(\theta)$. This suggests that second-order methods may be more effective
and stable when applied in the parametrization by $\varphi$, rather than the
original parametrization by $\theta$.

\subsection{Notation}
We write $\E_\eps$ for the expectation over $\eps \sim \N(0,\Id)$. We write
\[\E_g[f(g)]=\frac{1}{K}\sum_{g \in G} f(g)\]
for the expectation over the uniform law
$g \sim \Unif(G)$, and $\Var_g$ and $\Cov_g$ for the associated variance and
covariance. Similarly $\E_h$ is the expectation over $h \sim \Unif(G)$, and
$\E_{g_1,g_2}$ is the expectation over independent elements $g_1,g_2 \sim \Unif(G)$
unless stated otherwise.

We consider $\theta_*,d,G$ as constant throughout the paper. We write $C,C',c,c'>0$
for constants that may depend on $\theta_*,d,G$ and change from instance to
instance. These \emph{do not} depend on the noise level $\sigma$, and we will be
explicit about the dependence of our results on $\sigma$.

For a function $f:\R^d \to \R$, we denote its gradient and Hessian by
$\nabla f \in \R^d$ and $\nabla^2 f \in \R^{d \times d}$. More generally, we
denote by $\nabla^k f \in (\R^d)^{\otimes k}$ the symmetric
tensor of its $k^\text{th}$ order partial derivatives. For a coordinate $\theta_i$
of $\theta$, $\partial_{\theta_i} f$ is the partial derivative in $\theta_i$.
For $f:\R^d \to \R^k$, $\der f \in \R^{k \times d}$ is its full derivative
(i.e.\ Jacobian matrix).
When $k=1$, we take the convention that $\nabla f$ is a column vector,
so $\nabla f=\der f^\top$. We write $\nabla_\theta$, $\nabla_\theta^\ell$, and
$\der_\theta$ to clarify that these are taken with
respect to $\theta$, and we write $\nabla_\theta f(\theta_*)$,
$\nabla_\theta^\ell f(\theta_*)$, and $\der_\theta f(\theta_*)$ for their
evaluations at $\theta=\theta_*$. 

For a symmetric matrix $M \in \R^{d \times d}$, 
$\lambda_{\max}(M)$ and $\lambda_{\min}(M)$ are its largest and smallest
eigenvalues, and $\succeq$ and $\succ$ denote the positive-semidefinite and
positive-definite ordering. For $\mu \in \R^d$ and $\rho>0$,
$B_\rho(\mu)$ is the open $\ell_2$ ball of radius $\rho$ around $\mu$.
$\|\cdot\|$ is the $\ell_2$ norm for vectors and $\ell_2 \to \ell_2$ operator
norm (largest singular value) for matrices, $\langle \cdot,\cdot \rangle$ is the $\ell_2$
inner product, and $\|\cdot\|_\HS$ is the vectorized $\ell_2$ norm for higher-order tensors.
$\dist(x,S)=\inf_{y \in S} \|x-y\|$ is the $\ell_2$-distance from $x$ to a set $S$.
$\Id$ is the identity matrix, $\N(\cdot,\cdot)$ denotes the
Gaussian distribution parametrized by mean and variance/covariance,
and $[\ell]=\{1,\ldots,\ell\}$.

For $\alpha=1,2$, denote by
$\|W\|_{\psi_\alpha}=\inf\{t>0:\E_\eps[\exp((|W|/t)^\alpha)] \leq 2\}$ the
sub-exponential and sub-Gaussian norms of the random variable $W$. (See
\cite[Chapter 2]{vershyninHDP}.)

\subsection*{Acknowledgments}
We would like to thank Roy Lederman for helpful conversations at the onset of
this work. Z.~F.~was supported in part by NSF Grant DMS-1916198. Y.~S.~was supported
in part by a Junior Fellow award from the Simons Foundation and NSF Grant DMS-1701654.
Y.~W.~was supported in part by NSF Grant CCF-1900507, NSF CAREER award CCF-1651588,
and an Alfred Sloan fellowship.

\section{Preliminaries}

This section collects several more basic results about the population risk $R(\theta)$ and its empirical counterpart
$R_n(\theta)$, including expressions for their derivatives,
bounds on critical points, and the concentration of $R_n(\theta)$ around $R(\theta)$.

\subsection{The risk, gradient, and Hessian}

Let us first derive some simpler expressions for the risks $R_n(\theta)$ and $R(\theta)$.
We represent each sample $Y$ as
\begin{equation}\label{eq:Yalt}
Y=h(\theta_*+\sigma \eps)
\end{equation}
where $h \in G$, and $\eps \sim \N(0,\Id)$ is independent of $h$.
This is equivalent to the model (\ref{eq:orbitmodel}), by the rotational invariance
of the law of $\eps$. Then the marginal log-likelihood (\ref{eq:mixturedensity}) is
given by
\[-\log p_\theta(Y)=-\log \E_g \left[\left(\frac{1}{\sqrt{2\pi \sigma^2}}\right)^d
\exp\left(-\frac{\|h(\theta_*+\sigma \eps)-g\theta\|^2}{2\sigma^2}\right)\right].\]
Applying $\|h(\theta_*+\sigma \eps)-g\theta\|=\|\theta_*+\sigma \eps-h^\top
g\theta\|$ and the equality in law $h^\top g\overset{L}{=} g$ for any fixed $h \in G$,
we have
\begin{align*}
-\log p_\theta(Y) &=-\log \E_g
\left[\left(\frac{1}{\sqrt{2\pi \sigma^2}}\right)^d
\exp\left(-\frac{\|\theta_*+\sigma \eps-g\theta\|^2}{2\sigma^2}\right)\right]\\
&=\frac{d}{2}\log(2\pi\sigma^2)+\frac{\|\theta_*+\sigma\eps\|^2}{2\sigma^2}
+\frac{\|\theta\|^2}{2\sigma^2}-\log \E_g\left[\exp\left(\frac{\langle \theta_*
+\sigma \eps,g\theta \rangle}{\sigma^2}\right)\right].
\end{align*}

The first two terms above do not depend on $\theta$, and we omit them in the
sequel. We define the empirical risk as
\begin{equation}\label{eq:Rn}
R_n(\theta)=\frac{\|\theta\|^2}{2\sigma^2}-\frac{1}{n}\sum_{i=1}^n
\log \E_g\left[\exp\left(\frac{\langle \theta_*+\sigma \eps_i,g\theta
\rangle}{\sigma^2}\right)\right].
\end{equation}
Then $R_n(\theta)$ is a constant shift of the negative log-likelihood for
independent samples $Y_1,\ldots,Y_n$, as stated in (\ref{eq:empiricalrisk}).
We define the corresponding population risk $R(\theta)=\E[R_n(\theta)]$ by
\begin{equation}\label{eq:R}
R(\theta)=\frac{\|\theta\|^2}{2\sigma^2}-\E_\eps\left[
\log \E_g\left[\exp\left(\frac{\langle \theta_*+\sigma \eps,g\theta
\rangle}{\sigma^2}\right)\right]\right].
\end{equation}

\begin{remark}\label{remark:nonuniform}
The above arguments do not require $h \in G$ to be uniformly distributed.
That is to say, if $h$ is modeled as uniformly distributed,
the law of $p_\theta(Y)$ does not depend on the true distribution of $h$.
Thus our results apply also for non-uniform $h \in G$. Our results
do not describe the landscape if the non-uniformity is incorporated into
the likelihood model. Existing work on method-of-moments suggests
that, in such settings, the Fisher information may have a different dependence
on $\sigma$ in the high-noise
regime \cite{abbe2018multireference, sharon2020method}.
\end{remark}

Next, let us express the gradients, Hessians, and higher-order derivatives of
these risk functions in terms of a reweighted law for $g \in G$.
Given $\theta$ and $\eps$, we introduce the reweighted probability law on $G$ defined by
\begin{equation}\label{eq:gcondlaw}
p(g \mid \eps,\theta)=\exp\left(\frac{\langle \theta_*+\sigma\eps,
g\theta \rangle}{\sigma^2}\right)\Bigg/
\sum_{h \in G} \exp\left(\frac{\langle \theta_*+\sigma\eps,h\theta \rangle}
{\sigma^2}\right).
\end{equation}
We write $\P_g[\cdot \mid \eps, \theta]$, $\E_g[\cdot \mid \eps,\theta]$,
$\Var_g[\cdot \mid \eps,\theta]$, and $\Cov_g[\cdot \mid \eps,\theta]$ for the
probability, expectation, variance, and covariance with respect to
this reweighted law of $g$. We also write $\kappa_g^\ell[\cdot \mid \eps,\theta]$ for
the $\ell^\text{th}$ cumulant tensor with respect to this law; see
Appendix \ref{appendix:cumulants} for the definition.

\begin{lemma}\label{lem:empiricalriskform}
The derivatives of $R_n(\theta)$ take the forms
\begin{align}
\nabla R_n(\theta) &=\frac{1}{\sigma^2}\left(\theta-
\frac{1}{n}\sum_{i=1}^n \E_g\left[g^\top (\theta_*+\sigma \eps_i)
\Big| \eps_i,\theta\right]\right)\label{eq:gradRn}\\
\nabla^2 R_n(\theta) &=\frac{1}{\sigma^2}\left(\Id-
\frac{1}{\sigma^2} \cdot
\frac{1}{n}\sum_{i=1}^n \Cov_g\left[g^\top (\theta_*+\sigma \eps_i)
\Big| \eps_i,\theta\right]\right)\label{eq:hessRn}\\
\nabla^\ell R_n(\theta) &=-\frac{1}{\sigma^{2\ell}}
\cdot \frac{1}{n}\sum_{i=1}^n \kappa^\ell_g\left[g^\top (\theta_*+\sigma \eps_i)
\Big|\eps_i,\theta\right] \quad \text{ for } \ell \geq 3.\label{eq:derRn}
\end{align}
\end{lemma}
\begin{proof}
For any random vector $u \in \R^d$, the derivatives of its cumulant generating
function are given by
\[\nabla_\theta^\ell \log \E[e^{\langle u,\theta \rangle}]
=\kappa^\ell[u \mid \theta]\]
where $\kappa^\ell[u \mid \theta] \in (\R^d)^{\otimes \ell}$
is the $\ell^\text{th}$ cumulant tensor of $u$
under its reweighted law defined by $\E[f(u) \mid \theta]=\E[f(u)e^{\langle
u,\theta\rangle}]/\E[e^{\langle u,\theta\rangle}]$. (See Appendix
\ref{appendix:cumulants}.) In particular, for $\ell=1,2$, these are the mean and
covariance with respect to this law. Then (\ref{eq:gradRn}--\ref{eq:derRn}) follow
from differentiating (\ref{eq:Rn}) in $\theta$, and applying this to the random
vector $u=g^\top(\theta_*+\sigma\eps_i)/\sigma^2$ conditional on $\eps_i$.
\end{proof}

\begin{lemma}
The derivatives of $R(\theta)$ take the forms
\begin{align}
\nabla R(\theta)&=\frac{1}{\sigma^2}
\left(\theta-\E_\eps\left[\E_g\Big[g^\top(\theta_*+\sigma\eps)
\Big\vert \eps,\theta\Big]\right]\right)\label{eq:gradR}\\
&=\frac{1}{\sigma^2}\left(\E_\eps\Big[\E_g[g \mid \eps,\theta]^\top
\E_g[g \mid \eps,\theta]\Big] \,\theta-\E_\eps\Big[\E_g[g \mid \eps,\theta]
\Big]^\top \theta_*\right)\label{eq:gradRalt}\\
\nabla^2 R(\theta)&=\frac{1}{\sigma^2}
\left(\Id-\frac{1}{\sigma^2}\,\E_\eps\Big[\Cov_g
\Big[g^\top (\theta_*+\sigma \eps) \Big\vert \eps,\theta
\Big]\Big]\right)\label{eq:hessR}\\
\nabla^\ell R(\theta)&=-\frac{1}{\sigma^{2\ell}}
\,\E_\eps\Big[\kappa_g^\ell \Big[g^\top (\theta_*+\sigma \eps) \Big\vert \eps,\theta
\Big]\Big] \quad \text{ for } \ell \geq 3\label{eq:derR}
\end{align}
\end{lemma}
\begin{proof}
The identities (\ref{eq:gradR}), (\ref{eq:hessR}), and
(\ref{eq:derR}) are obtained by taking the expectations of
(\ref{eq:gradRn}--\ref{eq:derRn}) over $\eps_1,\ldots,\eps_n$.
(The derivatives of $R(\theta)$ in $\theta$ may be taken inside $\E_\eps$ by a
standard application of the dominated convergence theorem.)

For (\ref{eq:gradRalt}), we apply Gaussian integration by parts to rewrite
the $\E_\eps[\E_g[g^\top \eps \mid \eps,\theta]]$
term in (\ref{eq:gradR}): Denote by $g_{\cdot j}$ the $j^\text{th}$
column of a matrix $g \in G$, and by $g_{ij}$ the $(i,j)$ entry. Then
recalling the density (\ref{eq:gcondlaw}) and applying the integration-by-parts
identity $\E[f(\xi)\xi]=\E[f'(\xi)]$ for $\xi \sim \N(0,1)$, we get
\[\E_\eps\Big[\E_g[g_{\cdot j}^\top \eps \mid \eps,\theta]\Big]
=\sum_{i=1}^d \E_\eps\Big[\E_g[p(g \mid \eps,\theta)g_{ij}]\eps_i\Big]
=\sum_{i=1}^d \E_\eps\Big[\partial_{\eps_i}
\E_g[p(g \mid \eps,\theta)g_{ij}]\Big].\]
Write $(g\theta)_i$ as the $i^\text{th}$ coordinate of $g\theta$, and note that
differentiating (\ref{eq:gcondlaw}) in $\eps_i$ gives
\[\partial_{\eps_i}
p(g \mid \eps,\theta)=\frac{1}{\sigma} \Big(p(g \mid \eps,\theta)(g\theta)_i
-p(g \mid \eps,\theta)\E_h[p(h \mid \eps,\theta)(h\theta)_i]\Big)\]
where $h \sim \Unif(G)$ is independent of $g$. Then
\begin{align*}
\sigma\,\E_\eps\Big[\E_g[g_{\cdot j}^\top \eps \mid \eps,\theta] \Big]
&=\sum_{i=1}^d \E_\eps\Big[\Cov_g[g_{ij},(g\theta)_i \mid
\eps,\theta]\Big]\\
&=\E_\eps\Big[\E_g[g_{\cdot j}^\top g\theta \mid \eps,\theta]
-\E_g[g_{\cdot j} \mid \eps,\theta]^\top\E_g[g\theta \mid \eps,\theta]\Big]\\
&=\theta_j-\E_\eps\Big[\E_g[g_{\cdot j} \mid \eps,\theta]^\top
\E_g[g \mid \eps,\theta]\Big]\theta,
\end{align*}
the last line using $g_{\cdot j}^\top g\theta=\theta_j$ for any fixed orthogonal
matrix $g \in G$. Combining this for $j=1,\ldots,d$,
\[\sigma\,\E_\eps\Big[\E_g[g^\top \eps \mid \eps,\theta]\Big]
=\theta-\E_\eps\Big[\E_g[g \mid \eps,\theta]^\top \E_g[g \mid \eps,\theta]
\Big]\theta.\]
Substituting into (\ref{eq:gradR}) yields (\ref{eq:gradRalt}).
\end{proof}

\subsection{Subgroup decompositions}

If the group $G$ is the product of two groups $G_1$ and $G_2$ acting on orthogonal
subspaces of $\R^d$, then both the empirical and population risks
decompose as a sum corresponding to these two
components. This is stated formally in the following lemma.

\begin{lemma}\label{lemma:productgroup}
Let $V=[V_1 \mid V_2]$ be an orthogonal matrix, where $V_1 \in \R^{d \times d_1}$,
$V_2 \in \R^{d \times d_2}$, and $d_1+d_2=d$.
Suppose that $G \subset \rO(d)$ decomposes as
\[G=\left\{V \begin{pmatrix} g_1 & 0 \\ 0 & g_2 \end{pmatrix} V^\top:
g_1 \in G_1,\,g_2 \in G_2\right\}\]
for subgroups $G_1 \subset \rO(d_1)$ and $G_2 \subset \rO(d_2)$, and write
the corresponding decompositions $\theta_1=V_1^\top \theta$,
$\theta_2=V_2^\top \theta$, $\theta_{1,*}=V_1^\top \theta_*$,
$\theta_{2,*}=V_2^\top \theta_*$. Then
\[R_n(\theta)=R_n^{G_1}(\theta_1)+R_n^{G_2}(\theta_2) \qquad \text{ and }
\qquad R(\theta)=R^{G_1}(\theta_1)+R^{G_2}(\theta_2),\]
where $R_n^{G_1}$ and $R^{G_1}$ denote the empirical and population risks
(\ref{eq:empiricalrisk}) and (\ref{eq:populationrisk}) defined by $G_1$ and
$\theta_{1,*}$ in dimension $d_1$, and similarly for $G_2$.
\end{lemma}
\begin{proof}
Note that $\|\theta\|^2=\|\theta_1\|^2+\|\theta_2\|^2$.
Writing $g \in G$ as $g=V_1g_1V_1^\top+V_2g_2V_2^\top$, we have
\[\langle \theta_*+\sigma\eps_i,g\theta \rangle
=\langle \theta_{1,*}+\sigma V_1^\top \eps_i,g_1\theta_1 \rangle
+\langle \theta_{2,*}+\sigma V_2^\top \eps_i,g_2\theta_2 \rangle.\]
The expectation $\E_g$ may be written as independent expectations over
$g_1 \sim \Unif(G_1)$ and $g_2 \sim \Unif(G_2)$.
Furthermore, $V_1^\top \eps_i$ and $V_2^\top \eps_i$ are independent Gaussian
vectors of dimensions $d_1$ and $d_2$. Applying these to (\ref{eq:Rn})
yields $R_n(\theta)=R_n^{G_1}(\theta_1)+R_n^{G_2}(\theta_2)$.
Taking the expectation yields $R(\theta)=R^{G_1}(\theta_1)+R^{G_2}(\theta_2)$.
\end{proof}

In particular, we may always reduce our study to a group $G$ where $\E_g[g]=0$,
because of the following result. (Here $\E_g[g]$ is the expectation in $\R^{d
\times d}$ when we consider $G \subset \rO(d)$.)

\begin{lemma}\label{lemma:kerneldecomp}
Suppose $\E_g[g]$ has rank $d_1$ where $0<d_1 \leq d$, and set $d_2=d-d_1$.
Let $V=[V_1 \mid V_2]$ be an orthogonal
matrix where the columns of $V_2 \in \R^{d \times d_2}$ span the kernel of
$\E_g[g]$. Then 
\begin{equation}\label{eq:kerneldecomp}
G=\left\{V \begin{pmatrix} \Id & 0 \\ 0 & g_2 \end{pmatrix} V^\top:
g_2 \in G_2\right\}
\end{equation}
where $G_2 \subset \rO(d_2)$ is a subgroup that is group-isomorphic to $G$, and
$\E_{g_2}[g_2]=0$ for $g_2 \sim \Unif(G_2)$.
\end{lemma}
\begin{proof}
Observe that if $g \sim \Unif(G)$, then $g^\top=g^{-1} \sim \Unif(G)$, so
$\E_g[g]=\E_g[g^\top]=\E_g[g]^\top$. Furthermore, if $g,h \sim \Unif(G)$ are
independent, then $gh \sim \Unif(G)$, so
$\E_g[g]=\E_{g,h}[gh]=\E_g[g]\E_h[h]=\E_g[g]^2$. Hence $\E_g[g]$ is symmetric and
idempotent, so it is an orthogonal projection. For any $\theta$ in the range of
this projection, $\theta=\E_g[g]\theta=\E_g[g\theta]$, so
$\|\theta\|^2=\theta^\top \E_g[g\theta]$. As each $g\theta$ is also a vector on
the sphere of radius $\|\theta\|$, we have $\theta^\top g\theta<\|\theta\|^2$
unless $\theta=g\theta$. Thus, $\theta=g\theta$ for every $g \in G$, so $G$ acts
as the identity on the column span of $V_1$. This shows that each $g \in G$ has the 
form (\ref{eq:kerneldecomp}) for some matrix $g_2 \in \rO(d_2)$, and this
1-to-1 mapping from $g$ to $g_2$ must be a group isomorphism between $G$
and $G_2$. Since $G_2$ represents the action of $G$ on the column span of $V_2$,
which is the kernel of $\E_g[g]$, we have $\E_{g_2}[g_2]=0$.
\end{proof}

Combining Lemmas \ref{lemma:productgroup} and \ref{lemma:kerneldecomp}, we may
always decompose $R_n(\theta)=R_n^{\Id}(\theta_1)+R_n^{G_2}(\theta_2)$ and
$R(\theta)=R^{\Id}(\theta_1)+R^{G_2}(\theta_2)$, where $\theta_2$
is the component of $\theta$ in the kernel of $\E_g[g]$.
For $\theta_1$, the risks $R_n^{\Id}(\theta_1)$ and $R^{\Id}(\theta_1)$
correspond to the single Gaussian model $\N(\theta_{1,*},\sigma^2\Id)$. Then
$R^{\Id}(\theta_1)$ and $R_n^{\Id}(\theta_1)$ are strongly convex, and our study
of the landscapes of $R(\theta)$ and $R_n(\theta)$ reduces to studying
$R^{G_2}(\theta_2)$ and $R^{G_2}_n(\theta_2)$ for the mean-zero group $G_2$.

\subsection{Generic parameters and critical points}\label{sec:generic}

Throughout, we will assume that the true parameter $\theta_* \in \R^d$
is generic in the following sense.

\begin{definition}
For a connected open set $U \subseteq \R^d$,
a statement holds for {\bf generic} $\theta_* \in U$ if
it holds for all $\theta_*$ outside the zero set of an analytic
function $f:U \to \R^k$ that is not identically zero on $U$.
\end{definition}

The zero set of any such analytic function has measure zero
(see \cite{mityagin2020zero}), so in particular, a statement that holds for
generic $\theta_* \in \R^d$ holds everywhere outside a measure-zero
subset of $\R^d$.

At a minimum, we will require that the points of the orbit
$\O_{\theta_*}$ are distinct, so $|\O_{\theta_*}|=|G|=K$. This holds for
generic $\theta_*$ because for any $g \neq h \in G$, the condition
$(g-h)\theta_*=0$ defines a subspace of dimension at most $d-1$.

\begin{definition}
For an open domain $U \subseteq \R^d$ and $f:U \to \R$ twice continuously
differentiable, a point $\theta \in U$ is a {\bf critical point} of $f$
if $\nabla f(\theta)=0$. The critical point is {\bf non-degenerate}
if $\nabla^2 f(\theta)$ is non-singular. The function $f$ is {\bf Morse} if all
critical points are non-degenerate. The same definitions apply to $f:M \to \R$ for
any manifold $M$, upon parametrizing $M$ by a local chart.
\end{definition}

A correspondence between non-degenerate critical points of a function $f_1:U \to
\R$ and those of a function $f_2$ uniformly close to $f_1$ was shown
in \cite{mei2018landscape}. We will apply the
following version of this result for only the local minimizers, which
has a more elementary proof.

\begin{lemma}\label{lem:localmincompare}
Let $\theta_0 \in \R^d$, and let $f_1,f_2:B_\eps(\theta_0) \to \R$ be two functions
which are twice continuously differentiable.
Suppose $\theta_0$ is a critical point of $f_1$, and
$\lambda_{\min}(\nabla^2 f_1(\theta))
\geq c_0$ for some $c_0>0$ and all $\theta \in B_\eps(\theta_0)$. If
\[|f_1(\theta)-f_2(\theta)| \leq \delta
\qquad \text{ and } \qquad \|\nabla^2 f_1(\theta)-\nabla^2 f_2(\theta)\| \leq
\delta\]
for some $\delta<\min(c_0,c_0\eps^2/4)$ and all $\theta \in B_\eps(\theta_0)$, then
$f_2$ has a unique critical point in $B_\eps(\theta_0)$, which is a local
minimizer of $f_2$.
\end{lemma}
\begin{proof}
The given conditions imply $\lambda_{\min}(\nabla^2 f_2(\theta))>0$ for all
$\theta \in
B_\eps(\theta_0)$, so $f_2$ is strongly convex and has at most one
critical point. They also imply that for each $\theta \in
B_\eps(\theta_0)$ with $\|\theta-\theta_0\|=r$,
\[f_2(\theta)-f_2(\theta_0) \geq f_1(\theta)-f_1(\theta_0)-2\delta \geq
\frac{c_0r^2}{2}-2\delta.\]
For $r$ sufficiently close to $\eps$, we have $c_0r^2/2-2\delta>0$. Then
$f_2$ must have a local minimizer in $B_r(\theta_0)$.
\end{proof}

\subsection{Bounds for critical points}

A consequence of (\ref{eq:gradRn}) and (\ref{eq:gradR}) is the following
simple bound for critical points of $R(\theta)$ and $R_n(\theta)$.

\begin{lemma}\label{lemma:thetabound}
For $d$-dependent constants $C,C',c>0$, we have
$\sigma^2 \|\nabla R(\theta)\| \geq \|\theta\|-\|\theta_*\|-C\sigma$, and
$\sigma^2 \|\nabla R_n(\theta)\| \geq \|\theta\|-\|\theta_*\|-C\sigma$
with probability at least $1-C'e^{-cn}$. In particular, any critical point
$\theta$ of $R(\theta)$ satisfies $\|\theta\| \leq \|\theta_*\|+C\sigma$,
and the same holds for $R_n(\theta)$ with probability $1-C'e^{-cn}$.
\end{lemma}
\begin{proof}
The bound for $\|\nabla R(\theta)\|$ follows from (\ref{eq:gradR}) and
\[\Big\|\E_\eps\big[\E_g[g^\top (\theta_*+\sigma\eps) \mid \eps,\sigma]\big]\Big\|
\leq \E_\eps[\|\theta_*+\sigma\eps\|] \leq \|\theta_*\|+\sigma\,
\E_\eps[\|\eps\|] \leq \|\theta_*\|+\sigma\sqrt{d}.\]
The bound for $\|\nabla R_n(\theta)\|$ follows similarly from (\ref{eq:Rn}),
on the event
$n^{-1}\sum_{i=1}^n \|\eps_i\| \leq C$ which has probability at least
$1-C'e^{-cn}$ by Hoeffding's inequality for sub-Gaussian random variables (see
\cite[Theorem 2.6.2]{vershyninHDP}). Since $\nabla R(\theta)=0$ at a
critical point $\theta$, and similarly for $R_n(\theta)$, the statements for
critical points follow.
\end{proof}

When $\sigma$ is large, this bound is not sharp in its dependence on $\sigma$. We
will in fact show that any critical point $\theta$ of $R(\theta)$ satisfies
$\|\theta\| \leq C$ for a $\sigma$-independent
constant $C>0$. The following strengthening of Lemma \ref{lemma:thetabound}
first provides the a-priori bound $\|\theta\| \leq C\sigma^{2/3}$. Then,
combined with a series expansion of $R(\theta)$ in $\sigma^{-2}$,
we will improve this to $\|\theta\| \leq C$ in Lemma
\ref{lemma:localizationstrong} of Section \ref{sec:largesigma}.

\begin{lemma}\label{lemma:localization}
For some $(\theta_*,d,G)$-dependent constants $C,c,\sigma_0>0$ and all
$\sigma>\sigma_0$,
\begin{equation}\label{eq:localization}
\sigma^2\|\nabla R(\theta)\|>c\min\left(\frac{\|\theta\|^3}{\sigma^2},
\frac{\|\theta\|}{\sigma^{2/3}}\right)-\|\theta_*\|,
\end{equation}
and every critical point $\theta$ of $R(\theta)$ satisfies
$\|\theta\|<C\sigma^{2/3}$.
\end{lemma}
\begin{proof}
We apply the form of $\nabla R(\theta)$ given in (\ref{eq:gradRalt}).
Denote $\bar{Y}=(\theta_*+\sigma \eps)/\|\theta_*+\sigma \eps\|$
and $\bar{\theta}=\theta/\|\theta\|$. Then
\begin{align}
\sigma^2\|\nabla R(\theta)\|
&\geq \langle \bar{\theta},\sigma^2 \nabla R(\theta) \rangle\nonumber\\
&\geq \bar{\theta}^\top \E_\eps\Big[\E_g[g \mid \eps,\theta]^\top \E_g[g \mid
\eps,\theta]\Big]\theta-\bar{\theta}^\top \E_\eps\Big[\E_g[g \mid
\eps,\theta]^\top\Big]\theta_*\nonumber\\
&=\|\theta\| \cdot
\E_\eps\Big[\|\E_g[g\bar{\theta} \mid \eps,\theta]\|^2\Big]
-\bar{\theta}^\top \E_\eps\Big[\E_g[g \mid
\eps,\theta]^\top\Big]\theta_*\nonumber\\
&\geq \|\theta\| \cdot
\E_\eps\Big[(\bar{Y}^\top \E_g[g\bar{\theta} \mid
\eps,\theta])^2\Big]-\|\theta_*\|\nonumber\\
&=\|\theta\| \cdot \E_\eps\Big[\E_g[\bar{Y}^\top g\bar{\theta} \mid
\eps,\theta]^2\Big]-\|\theta_*\|.\label{eq:gradboundtmp}
\end{align}

We analyze the quantity $\E_g[\bar{Y}^\top g\bar{\theta} \mid \eps,\theta]$
for fixed $\eps$ (and hence fixed $\bar{Y}$):
Note that $|\bar{Y}^\top g\bar{\theta}| \leq 1$. Let
$K(s)$ be the cumulant generating function of $\bar{Y}^\top g\bar{\theta}$ over
the uniform law $g \sim \Unif(G)$, and let $K'(s)$ be its derivative.
Denote
\[t \equiv t(\eps,\theta)=\frac{\|\theta_*+\sigma \eps\|\|\theta\|}{\sigma^2}.\]
Then
\begin{equation}\label{eq:Ktrelation}
\E_g[\bar{Y}^\top g\bar{\theta} \mid \eps,\theta]
=\E_g[p(g \mid \eps,\theta)\bar{Y}^\top g\bar{\theta}]
=\frac{\E_g[\bar{Y}^\top g\bar{\theta} \cdot
e^{t\bar{Y}^\top g\bar{\theta}}]}
{\E_g[e^{t\bar{Y}^\top g\bar{\theta}}]}
=\frac{d}{ds}\log \E_g[e^{s\bar{Y}^\top g\bar{\theta}}]\Big|_{s=t}=K'(t).
\end{equation}
Writing $\kappa_\ell$ as the
$\ell^\text{th}$ cumulant of this law, we have
\begin{equation}\label{eq:Ksseries}
K(s)=\sum_{\ell=1}^\infty \kappa_\ell \frac{s^\ell}{\ell!},
\end{equation}
where this series is absolutely convergent for $|s|<1/e$ by
Lemma \ref{lem:cumulantbounds}. Set
\[t_\sigma \equiv t_\sigma(\eps,\theta)=\min(t(\eps,\theta),\sigma^{-1/3}),\]
where $t_\sigma<1/e$ for $\sigma>\sigma_0$ and large enough $\sigma_0$.
Since $K(0)=0$, 
using the convexity of the cumulant generating function $K$ we can bound its derivative from below by
\[K'(t) \geq K'(t_\sigma) \geq \frac{K(t_\sigma)}{t_\sigma}
=\sum_{\ell=1}^\infty \kappa_\ell \frac{t_\sigma^{\ell-1}}{\ell!}.\]
Applying $|\kappa_\ell| \leq \ell^\ell$ from Lemma \ref{lem:cumulantbounds}
and $\ell! \geq \ell^\ell/e^\ell$,
\[K'(t) \geq \kappa_1+\frac{t_\sigma}{2}\kappa_2-\sum_{\ell=3}^\infty
e^\ell t_\sigma^{\ell-1} \geq \kappa_1+\frac{t_\sigma}{2}\kappa_2-30t_\sigma^2\]
for $\sigma>\sigma_0$ and large enough $\sigma_0$.
Here, $\kappa_1=\E_g[\bar{Y}^\top g\bar{\theta}]$ and
$\kappa_2=\Var_g[\bar{Y}^\top g\bar{\theta}]$.

Now observe that there exists a constant $c_0 \equiv c_0(d)>0$, such that
if $v$ is any random vector on the unit sphere in $\R^d$, then there is a
deterministic vector $u_0$ on the unit sphere for which
\[\min(\E[u_0^\top v],\Var[u_0^\top v])>2c_0.\]
This is because if the mean of $v$ is near 0 and $v$ lies on the sphere,
then the variance of $v$ must be bounded below by a constant
in some direction. Then also
for some $\delta_0>0$ depending only on $c_0$, we have
\[\min(\E[u^\top v],\Var[u^\top v])>c_0 \text{ for all } u \in
B_{\delta_0}(u_0).\]
Let us apply this to the random vector $v=g\bar{\theta}$ under the
uniform law of $g$. (So $u_0$ depends on $G$ and $\theta$.)
Then for $\sigma>\sigma_0$, on the event $\bar{Y} \in B_{\delta_0}(u_0)$, we get
\[K'(t) \geq \frac{c_0}{2}t_\sigma-30t_\sigma^2 \geq \frac{c_0}{3} t_\sigma.\]

Recalling (\ref{eq:Ktrelation}) and applying this to (\ref{eq:gradboundtmp}),
\begin{align*}
\sigma^2\|\nabla R(\theta)\| &\geq \|\theta\|\cdot
\E_\eps\left[\left(\frac{c_0}{3}t_\sigma(\eps,\theta)\right)^2
\1\{\bar{Y} \in B_{\delta_0}(u_0)\}\right]-\|\theta_*\|\\
&\geq \|\theta\|\cdot \E_\eps\left[\left(\frac{c_0}{3}
t_\sigma(\eps,\theta)\right)^2
\1\{\bar{Y} \in B_{\delta_0}(u_0),\|\theta_*+\sigma \eps\|
\geq \sigma\}\right]-\|\theta_*\|.
\end{align*}
On the event $\|\theta_*+\sigma \eps\| \geq \sigma$, we have $t(\eps,\theta)
\geq \|\theta\|/\sigma$, so $t_\sigma(\eps,\theta) \geq
\min(\|\theta\|/\sigma,\sigma^{-1/3})$. Then
\[\sigma^2\|\nabla R(\theta)\| \geq
\frac{c_0^2}{9}\min\left(\frac{\|\theta\|^3}{\sigma^2},\frac{\|\theta\|}{\sigma^{2/3}}\right)
\P\Big[\bar{Y} \in B_{\delta_0}(u_0),\,\|\theta_*+\sigma \eps\| \geq
\sigma\Big]-\|\theta_*\|.\]
Recalling the definition $\bar{Y}=(\theta_*+\sigma \eps)/\|\theta_*+\sigma
\eps\|$, as $\sigma \to \infty$, we have
\[\P\Big[\bar{Y} \in B_{\delta_0}(u_0),\,\|\theta_*+\sigma \eps\|
\geq \sigma\Big]
\to \P\Big[\eps/\|\eps\| \in B_{\delta_0}(u_0),\,\|\eps\| \geq 1\Big].\]
Since $\eps/\|\eps\|$ is uniformly distributed on the sphere, the
limit is a positive constant depending only on the dimension $d$ and $\delta_0$.
Furthermore, for fixed $\theta_*$, this convergence is uniform over $u_0$ on the
unit sphere. Thus we obtain
\[\P\Big[\bar{Y} \in B_{\delta_0}(u_0),\,\|\theta_*+\sigma \eps\| \geq
\sigma\Big] \geq c\]
for a constant $c \equiv c(d)$ and all $\sigma>\sigma_0(\theta_*,d,G)$. This
yields (\ref{eq:localization}). For a large enough constant 
$C \equiv C(\theta_*,d,G)>0$, this implies $\|\nabla R(\theta)\|>0$ when
$\|\theta\| \geq C\sigma^{2/3}$, so any critical point
satisfies $\|\theta\|<C\sigma^{2/3}$.
\end{proof}

\subsection{Concentration of the empirical risk}

We establish uniform concentration of $R_n(\theta)$, $\nabla R_n(\theta)$, and
$\nabla^2 R_n(\theta)$ around their expectations. This will
allow us to translate results about the population landscape of $R(\theta)$ to
the empirical landscape of $R_n(\theta)$.

\begin{lemma}\label{lem:concentration}
There exist $(\theta_*,d,G)$-dependent constants $C,c>0$ such that 
for any $r,t>0$, denoting $B_r \equiv B_r(0)=\{\theta \in \R^d:\|\theta\|<r\}$,
\begin{align}
\P\Big[\sup_{\theta\in B_r}|R_n(\theta)-R(\theta)|\geq t\Big]&\leq 
\left(\tfrac{Cr(1+\sigma)}{\sigma^2 t}\right)^d
\exp\left(-cn\,\tfrac{\sigma^2t^2}{r^2}\right)+Ce^{-cn}
\label{eq:concRn}\\
\P\Big[\sup_{\theta\in B_r}\|\nabla R_n(\theta)-\nabla R(\theta)\|\geq
t\Big]&\leq \left(\tfrac{Cr(1+\sigma^2)}{\sigma^4
t}\right)^d\exp\left(-cn\,\tfrac{\sigma^4t^2}{1+\sigma^2}\right)+Ce^{-cn}\label{eq:concgradRn}\\
\P\Big[\sup_{\theta\in B_r}\|\nabla^2 R_n(\theta)-\nabla^2 R(\theta)\|\geq
t\Big]&\leq
\left(\tfrac{Cr(1+\sigma^3)}{\sigma^6t}\right)^d\exp\left(-cn\min\left(\tfrac{\sigma^8t^2}{1+\sigma^4},\tfrac{\sigma^4t}{1+\sigma^2}\right)\right)+Ce^{-cn^{2/3}}.\label{eq:conchessRn}
\end{align}
\end{lemma}

We prove this by first showing pointwise concentration in Lemma
\ref{lem:pointwise_concentration}, then establishing Lipschitz continuity of these
risks, gradients, and Hessians in Lemma \ref{lem:local_lipschitz}, and finally
applying a covering net argument.

\begin{lemma}\label{lem:pointwise_concentration}
For some $(\theta_*,d,G)$-dependent constants $C,c>0$, any $\theta \in \R^d$, and
any $t>0$,
\begin{align}
\P\left[\left|R_n(\theta)-R(\theta)\right|\geq t\right]&\leq C \exp\left(-cn\,
\tfrac{\sigma^2t^2}{\|\theta\|^2}\right)
\label{eq:value_pointwise}\\
\P\left[\|\nabla R_n(\theta)-\nabla R(\theta)\| \geq t\right]&\leq
C\exp\left(-cn\,\tfrac{\sigma^4t^2}{1+\sigma^2}\right)\label{eq:gradient_pointwise}\\
\P\left[\|\nabla^2 R_n(\theta) - \nabla^2 R(\theta)\| \geq t\right]&\leq
C\exp\left(-cn\min\left(\tfrac{\sigma^8t^2}{1+\sigma^4},\,
\tfrac{\sigma^4t}{1+\sigma^2}\right)\right).\label{eq:hessian_pointwise}
\end{align}
\end{lemma}

\begin{proof}
We apply the Bernstein and Hoeffding inequalities. 
Recall that for $\alpha=1$ or 2, 
$\|f(\eps)\|_{\psi_\alpha}$ denotes the
sub-exponential or sub-Gaussian norm of the random variable $f(\eps)$ over the
law $\eps \sim \N(0,\Id)$.

For $R_n(\theta)$, recall the form (\ref{eq:Rn}). Set
\[f_1(\eps)=\log \E_g\left[\exp \left(\frac{\langle \theta_*+\sigma \eps,g\theta
\rangle}{\sigma^2}\right)\right].\]
Then $\nabla_\eps f_1(\eps)=\E_g[g\theta \mid \eps,\theta]/\sigma$, so
$\|\nabla_\eps f_1(\eps)\| \leq \E_g[\|g\theta\| \mid
\eps,\theta]/\sigma \leq \|\theta\|/\sigma$
and $f_1$ is $\|\theta\|/\sigma$-Lipschitz. By Gaussian concentration of measure and
Hoeffding's inequality (see \cite[Theorems 2.6.2, 5.2.2]{vershyninHDP}),
for constants $C,c>0$ and any $t>0$,
\[\|f_1(\eps)-\E_\eps f_1(\eps)\|_{\psi_2} \leq \frac{C\|\theta\|}{\sigma},
\quad \P\bigg[\bigg|\frac{1}{n}\sum_{i=1}^n f_1(\eps_i)-\E_\eps[f_1(\eps)]\bigg|
\geq t \bigg] \leq 2\exp\left(-cn\,\frac{\sigma^2t^2}{\|\theta\|^2}\right).\]
Applying this to (\ref{eq:Rn}) yields (\ref{eq:value_pointwise}).

For $\nabla R_n(\theta)$, recall (\ref{eq:gradRn}). Denote by $g_{\cdot j}$ the
$j$th column of $g$. Momentarily fixing $j$, denote
\[f_2(\eps)=\E_g\left[g_{\cdot j}^\top (\theta_*+\sigma \eps)\Big|
\eps,\theta\right], \qquad f_{2,g}(\eps)=g_{\cdot j}^\top (\theta_*+\sigma
\eps)\]
where $f_{2,g}$ is defined for each fixed $g \in G$. Then
\[\|f_2(\eps)\|_{\psi_2} 
=\left\|\sum_{g \in G} p(g \mid \eps,\theta) f_{2,g}(\eps)\right\|_{\psi_2}
\leq K \cdot \max_{g \in G} \Big\|p(g \mid \eps,\theta) f_{2,g}(\eps)\Big\|_{\psi_2}
\leq K \cdot \max_{g \in G} \|f_{2,g}(\eps)\|_{\psi_2},\]
the last inequality applying $|p(g \mid \eps,\theta)| \leq 1$ and the definition
of the sub-Gaussian norm. For each fixed $g \in G$, we have
$\|f_{2,g}(\eps)\|_{\psi_2} \leq C(1+\sigma)$. Then by Hoeffding's inequality,
\[\P\bigg[\bigg|\frac{1}{n}\sum_{i=1}^n f_2(\eps_i)-\E_\eps[f_2(\eps)]\bigg|>t
\bigg] \leq 2\exp\left(-cn\,\frac{t^2}{(1+\sigma)^2}\right).\]
This establishes concentration of the $j$th coordinate of $R_n(\theta)$.
Applying a union bound over indices $j=1,\ldots,d$ and replacing $t$ by
$\sigma^2 t$ yields (\ref{eq:gradient_pointwise}).

For $\nabla^2 R_n(\theta)$, recall (\ref{eq:hessRn}). Momentarily
fixing the indices $j$ and $k$, denote
\begin{align*}
f_3(\eps)&=\Cov_g\left[g_{\cdot j}^\top (\theta_*+\sigma\eps),\;
g_{\cdot k}^\top (\theta_*+\sigma\eps)\Big|\eps,\theta\right]\\
&=\E_g\left[g_{\cdot j}^\top (\theta_*+\sigma\eps) \cdot
g_{\cdot k}^\top (\theta_*+\sigma\eps)\Big|\eps,\theta\right]
-\E_g\left[g_{\cdot j}^\top (\theta_*+\sigma\eps) \Big|\eps,\theta\right]
\cdot \E_g\left[g_{\cdot k}^\top (\theta_*+\sigma\eps)\Big|\eps,\theta\right]
\end{align*}
Using the same argument as above, we have the bounds
\[\left\|\E_g\left[g_{\cdot j}^\top (\theta_*+\sigma\eps) \cdot
g_{\cdot k}^\top (\theta_*+\sigma\eps)\Big|\eps,\theta\right]\right\|_{\psi_1}
\leq C(1+\sigma^2),
\quad \left\|\E_g\left[g_{\cdot j}^\top (\theta_*+\sigma\eps)
\Big|\eps,\theta\right]\right\|_{\psi_2} \leq C(1+\sigma).\]
Together with the inequality $\|XY\|_{\psi_1}
\leq \|X\|_{\psi_2}\|Y\|_{\psi_2}$, this yields $\|f_3(\eps)\|_{\psi_1} \leq
C(1+\sigma^2)$. Then by Bernstein's inequality (see \cite[Theorem
2.8.1]{vershyninHDP}),
\[\P\bigg[\bigg|\frac{1}{n}\sum_{i=1}^n
f_3(\eps_i)-\E_\eps[f_3(\eps)]\bigg|>t\bigg] \leq 2\exp\left(
-cn\min\left(\frac{t^2}{(1+\sigma^2)^2},\frac{t}{1+\sigma^2}\right)\right).\]
This establishes concentration of the $(j,k)$ entry
of $\nabla^2 R_n(\theta)$. Taking a union bound over $j,k \in \{1,\ldots,d\}$ and
replacing $t$ by $\sigma^4t$ yields (\ref{eq:hessian_pointwise}).
\end{proof}

\begin{lemma}\label{lem:local_lipschitz}
For a $(\theta_*,d,G)$-dependent constant $C'>0$, as functions over $\theta \in
\R^d$,
\begin{enumerate}[(a)]
\item $R(\theta)-\|\theta\|^2/(2\sigma^2)$ is $C'(1+\sigma)/\sigma^2$-Lipschitz.
\item Each entry of $\nabla R(\theta)-\theta/\sigma^2$ is $C'(1+\sigma^2)/\sigma^4$-Lipschitz.
\item Each entry of $\nabla^2 R(\theta)-\Id/\sigma^2$ is
$C'(1+\sigma^3)/\sigma^6$-Lipschitz.
\end{enumerate}
For $d$-dependent constants $C,c>0$, statements (a) and (b) also hold for
$R_n(\theta)-\|\theta\|^2/(2\sigma^2)$ and $\nabla
R_n(\theta)-\theta/\sigma^2$ with probability at least $1-Ce^{-cn}$, and (c) holds for
$\nabla^2 R_n(\theta)-\Id/\sigma^2$ with probability at least $1-Ce^{-cn^{2/3}}$.
\end{lemma}

\begin{proof}
To prove the desired Lipschitz property, it suffices to bound the first three derivatives of $R(\theta)$.
Recall the expressions (\ref{eq:gradR}), (\ref{eq:hessR}), and (\ref{eq:derR})
for $\nabla^\ell R(\theta)$.
Note that $\|g^\top(\theta_*+\sigma\eps)\|=\|\theta_*+\sigma\eps\|$. Thus, under the law 
\eqref{eq:gcondlaw}, each entry of $g^\top(\theta_*+\sigma\eps)$ has magnitude at most $\|\theta_*+\sigma\eps\|$. 
Invoking Lemma \ref{lem:cumulantbounds}(b), we conclude that for each $\ell \geq
1$ and some constant $C \equiv C(\ell,d,\|\theta_*\|)$,
\[\|\kappa_g^\ell[g^\top(\theta_*+\sigma\eps) \mid \eps,\theta]\|_\HS \leq
C(1+\sigma^\ell\|\eps\|^\ell)\]
where $\ell=1,2$ for the mean and covariance.
 
Applying these bounds to (\ref{eq:gradR}),
(\ref{eq:hessR}), and (\ref{eq:derR}) and taking the expectation over $\eps
\sim \N(0,\Id)$ yields the Lipschitz properties for the population risk
$R(\theta)$. Recalling the
forms (\ref{eq:gradRn}--\ref{eq:derRn}), this also shows the
Lipschitz properties for the empirical risk $R_n(\theta)$ on the events
\[\cE^\alpha=\left\{\frac{1}{n}\sum_{i=1}^n \|\eps_i\|^\alpha
\leq C_0 \right\}\]
for $\alpha=1,2,3$ respectively, where $C_0>0$ is any fixed constant. 
For $\alpha=1,2$ and a sufficiently large constant $C_0>0$,
we have $\P[\cE^\alpha] \geq
1-Ce^{-cn}$ by the Hoeffding and Bernstein inequalities.
For $\alpha=3$, we show in Appendix \ref{appendix:sumcubes}
using the result of \cite{adamczak2015concentration} that
\begin{equation}\label{eq:sumcubes}
\P\left[\frac{1}{n}\sum_{i=1}^n \|\eps_i\|^3 \leq C_0\right]
\geq 1-Ce^{-cn^{2/3}}
\end{equation}
for a sufficiently large constant $C_0>0$. (Note that
this bound is optimal, by considering the deviation of a single summand
$n^{-1}\|\eps_i\|^3$.) This concludes the proof.
\end{proof}

\begin{proof}[Proof of Lemma \ref{lem:concentration}]
Denote $\bar{R}_n(\theta)=R_n(\theta)-\|\theta\|^2/(2\sigma^2)$ and
$\bar{R}(\theta)=R(\theta)-\|\theta\|^2/(2\sigma^2)$. Note that concentration of
$R_n(\theta),\nabla R_n(\theta),\nabla^2 R_n(\theta)$ is equivalent to that
of $\bar{R}_n(\theta),\nabla \bar{R}_n(\theta),\nabla^2 \bar{R}_n(\theta)$.

For $\bar{R}_n(\theta)$, we take a $\delta$-net $N$ of $B_r$
having cardinality $|N| \leq (Cr/\delta)^d$.
Applying (\ref{eq:value_pointwise}) and a union bound over $N$,
\[\P\left[\sup_{\mu \in N}\left|\bar{R}_n(\mu)
-\bar{R}(\mu)\right| \geq t/3\right]\leq
\left(\frac{Cr}{\delta}\right)^d \exp\left(-cn\,\frac{\sigma^2t^2}{r^2}\right).\]
By the Lipschitz bounds for $\bar{R}(\theta)$ and $\bar{R}_n(\theta)$ in
Lemma \ref{lem:local_lipschitz}, picking $\delta=c\sigma^2 t/(1+\sigma)$ for a
small enough constant $c>0$ ensures on an event of probability $1-Ce^{-cn}$ that
$|\bar{R}(\theta)-\bar{R}(\mu)| \leq t/3$ and $|\bar{R}_n(\theta)-\bar{R}_n(\mu)| \leq t/3$
for each point $\theta \in B_r$ and the
closest point $\mu \in N$. Combining these shows (\ref{eq:concRn}).
The bounds (\ref{eq:concgradRn}) and (\ref{eq:conchessRn}) are obtained similarly.
\end{proof}

\section{Landscape analysis for low noise}\label{sec:lownoise}

In this section, we analyze the function landscapes of $R(\theta)$ and
$R_n(\theta)$ in the low-noise regime $\sigma<\sigma_0(\theta_*,d,G)$.
Section \ref{sec:lownoiselocal} analyzes the local landscapes in
a neighborhood of $\theta_*$, as well as the Fisher information
$I(\theta_*)=\nabla_\theta^2 R(\theta_*)$, and Theorem \ref{thm:lownoiselocal}
shows that these behave similarly to a
single-component Gaussian model $\N(\theta_*,\sigma^2\Id)$.
Section \ref{sec:lownoiseglobal} analyzes the global landscapes, and Theorem
\ref{thm:landscapelownoise} and Corollary \ref{cor:landscapelownoise}
show that these are globally benign for small $\sigma$ and large $n$.

\subsection{Local landscape and Fisher information}\label{sec:lownoiselocal}

\begin{theorem}\label{thm:lownoiselocal}
For any $\theta_* \in \R^d$ where $|\O_{\theta_*}|=|G|=K$, there exist
$(\theta_*,d,G)$-dependent constants $\sigma_0,c,\rho>0$ such that
as long as $\sigma<\sigma_0$, every $\theta \in B_\rho(\theta_*)$ satisfies
\begin{equation}\label{eq:Ccondition}
\|\nabla^2 R(\theta)-\sigma^{-2}\Id\|<e^{-c/\sigma^2}.
\end{equation}
In particular, the Fisher information satisfies
$\|I(\theta_*)-\sigma^{-2}\Id\|<e^{-c/\sigma^2}$.
\end{theorem}
Note that by rotational symmetry of $R(\theta)$,
the same statements hold for $B_\rho(\mu)$ and each $\mu \in \O_{\theta_*}$.

\begin{proof}
Since the $K$ points of $\O_{\theta_*}$ are distinct and have the same norm,
we must have $\|\theta_*\|^2>\theta_*^\top \mu$ for each $\mu \in \O_{\theta_*}$
different from $\theta_*$.
Pick ($\theta_*$-dependent) constants $c_0,\rho>0$ such that
$(\theta_*-\mu)^\top \theta_*>3c_0$ and 
$\|\theta_*-\mu\|\rho<c_0$ for all such $\mu$, and also
$\rho<\|\theta_*\|/2$. Define
\begin{equation}\label{eq:goodweights}
\cE=\{\eps \in \R^d: 2\sigma \|\eps\|\|\theta\| \leq c_0\}.
\end{equation}

Consider $\theta \in B_\rho(\theta_*)$, and
recall the form (\ref{eq:hessR}) for $\nabla^2 R(\theta)$.
For any unit vector $v \in \R^d$, we have
\begin{align*}
v^\top \E_\eps\Big[\Cov_g[g^\top (\theta_*+\sigma \eps)
\mid \eps,\theta]\Big]v
&=\E_\eps\Big[\Var_g[\langle v, g^\top (\theta_*+\sigma \eps) \rangle
\mid \eps,\theta]\Big]\\
&=\E_\eps\Big[\Var_g[\langle gv, \theta_*+\sigma \eps \rangle
\mid \eps,\theta]\Big]
\leq \E_\eps\Big[\E_g[ \langle gv-v,\,\theta_*+\sigma \eps
\rangle^2 \mid \eps,\theta ]\Big].
\end{align*}
Let us decompose the last line as $\mathrm{I}+\mathrm{II}$ where
\begin{align*}
\mathrm{I}&=\E_\eps\Big[\1\{\eps \notin \cE\}
\E_g[\langle gv-v,\,\theta_*+\sigma \eps \rangle^2 \mid \eps,\theta]\Big],\\
\mathrm{II}&=\E_\eps\Big[\1\{\eps \in \cE\}
\E_g[\langle gv-v,\,\theta_*+\sigma\eps \rangle^2 \mid \eps,\theta]\Big].
\end{align*}
For $\mathrm{I}$, we have $\|\theta\| \leq \|\theta_*\|+\rho$. Applying the
chi-squared tail bound $\P[\|\eps\|^2>t]<e^{-ct}$ for all $t>C$,
we get $\P[\eps \notin \cE]<e^{-c/\sigma^2}$. Then by Cauchy-Schwarz,
\begin{align*}
\mathrm{I} &\leq
\P[\eps \notin \cE]^{1/2}\E_\eps[\E_g[\langle gv-v,\,\theta_*+\sigma \eps
\rangle^2 \mid \eps,\theta]^2]^{1/2}\\
&\leq \P[\eps \notin \cE]^{1/2}\E_\eps[(2\|\theta_*+\sigma \eps\|)^4]^{1/2}
<e^{-c'/\sigma^2}
\end{align*}
for constants $c',\sigma_0>0$ and all $\sigma<\sigma_0$.
For $\mathrm{II}$, let us bound $\P_g[g \neq \Id \mid \eps,\theta]$ when $\eps
\in \cE$: For any $g \neq \Id$, letting $\mu=g^\top \theta_*$,
\[\langle \theta_*+\sigma\eps,\,\theta-g\theta \rangle
 \geq (\theta_*-\mu)^\top \theta-2\sigma \|\eps\|\|\theta\| \geq
(\theta_*-\mu)^\top \theta_*-2\sigma \|\eps\|\|\theta\|-\|\theta_*-\mu\|\rho>c_0.\]
Then recalling (\ref{eq:gcondlaw}),
$p(\Id \mid \eps,\theta)/p(g \mid \eps,\theta)>e^{c_0/\sigma^2}$
and so
\begin{equation}\label{eq:weightnearone}
p(\Id \mid \eps,\theta)>e^{c_0/\sigma^2}/(e^{c_0/\sigma^2}+K-1)
>1-e^{-c/\sigma^2}
\end{equation}
for constants $c,\sigma_0>0$ and all $\sigma<\sigma_0$.
Thus $\P_g[g \neq \Id \mid \eps,\theta]=1-p(\Id \mid \eps,\theta)<e^{-c/\sigma^2}$, so
\[\mathrm{II} \leq
\E_\eps[\1\{\eps \in \cE\} \P_g[g \neq \Id \mid \eps,\theta] \cdot
(2\|\theta_*+\sigma\eps\|)^2]<e^{-c'/\sigma^2}.\]
Combining these, we get
$v^\top \E_\eps[\Cov_g[g^\top (\theta_*+\sigma\eps) \mid \eps,\theta]]v
<e^{-c/\sigma^2}$ for any unit vector $v \in \R^d$.
Then (\ref{eq:Ccondition}) follows from (\ref{eq:hessR}).
Specializing to $\theta=\theta_*$ yields the statement for
$I(\theta_*)$.
\end{proof}

The following corollary then shows that with high probability when
$n \gg \sigma^{-1}\log \sigma^{-1}$, the empirical risk $R_n(\theta)$ is strongly
convex with a unique local minimizer in $B_\rho(\theta_*)$. By rotational
symmetry, the same statement holds for $B_\rho(\mu)$ and each $\mu \in
\O_{\theta_*}$.

\begin{corollary}\label{cor:lownoiselocal}
For some $(\theta_*,d,G)$-dependent constants $C,c,\sigma_0>0$, if
$\sigma<\sigma_0$, then with probability at least
$1-Ce^{-cn^{2/3}}-\sigma^{-C}e^{-c\sigma n}$,
$\lambda_{\min}(\nabla^2 R_n(\theta)) \geq 1/(2\sigma^2)$ for all $\theta \in
B_\rho(\theta_*)$, and $R_n(\theta)$ has a unique local minimizer and critical
point in $B_\rho(\theta_*)$.
\end{corollary}
\begin{proof}
This follows from Lemma \ref{lem:localmincompare} and Theorem
\ref{thm:lownoiselocal} if we can show that
\[\sup_{\theta \in B_\rho(\theta_*)} \|R_n(\theta)-R(\theta)\| \leq c_1/\sigma^2
\qquad
\text{ and } \qquad
\sup_{\theta \in B_\rho(\theta_*)} \|\nabla^2 R_n(\theta)-\nabla^2 R(\theta)\| \leq
c_1/\sigma^2\]
for a small enough constant $c_1>0$. Applying (\ref{eq:concRn}) with
$r=\|\theta_*\|+\rho$ and $t=c_1/\sigma^2$, we obtain
$\sup_{\theta \in B_\rho(\theta_*)} \|R_n(\theta)-R(\theta)\| \leq c_1/\sigma^2$ with
probability $1-Ce^{-cn}$. Applying (\ref{eq:conchessRn}), we also obtain
$\sup_{\theta \in B_\rho(\theta_*)} \|\nabla^2 R_n(\theta)-\nabla^2 R(\theta)\|
\leq c_1/\sigma^2$ with probability $1-\sigma^{-C}e^{-c\sigma^4n}
-Ce^{-cn^{2/3}}$. To reduce $\sigma^4$ to $\sigma$ in this probability bound, let us
derive a sharper concentration inequality for $\nabla^2 R_n(\theta)$
than the general result provided by (\ref{eq:hessian_pointwise}),
when $\theta \in B_\rho(\theta_*)$ and $\sigma<\sigma_0$.

Recall the set $\cE$ in
(\ref{eq:goodweights}) and the form for $\nabla^2 R_n(\theta)$ in (\ref{eq:hessRn}).
Let us write this as
\begin{equation}\label{eq:hessRndecomp}
\nabla^2 R_n(\theta)=\frac{1}{\sigma^2}\Id
-\frac{1}{\sigma^4} \cdot \frac{1}{n}\sum_{i=1}^n (X_i+Y_i)
-\frac{1}{\sigma^2} \cdot \frac{1}{n}\sum_{i=1}^n Z_i
\end{equation}
where $X_i,Y_i,Z_i \in \R^{d \times d}$ are given by
\begin{align*}
X_i&=\left(\Cov_g\left[g^\top(\theta_*+\sigma\eps_i)\Big| \eps_i,\theta\right]
-\sigma^2\Cov_g\left[g^\top \eps_i \Big| \eps_i,\theta\right]\right)
\1\{\eps_i \in \cE\}\\
Y_i&=\left(\Cov_g\left[g^\top(\theta_*+\sigma\eps_i)\Big| \eps_i,\theta\right]
-\sigma^2\Cov_g\left[g^\top \eps_i \Big| \eps_i,\theta\right]\right)
\1\{\eps_i \notin \cE\}\\
Z_i&=\Cov_g\left[g^\top \eps_i \Big| \eps_i,\theta\right].
\end{align*}

Observe that since $\|Z_i\| \leq \|\E_g[g^\top \eps_i\eps_i^\top g \mid
\eps_i,\theta]\| \leq \|\eps_i\|^2$, and $\|\eps_i\|^2$ has constant
sub-exponential norm, each entry of $Z_i$ also has constant sub-exponential
norm (where constants may depend on $d$). Applying Bernstein's inequality
entrywise and taking a union bound over all entries, for constants $C,c>0$ and
any $t>0$,
\begin{equation}\label{eq:Zibound}
\P\left[\left\|\frac{1}{n}\sum_{i=1}^n Z_i-\E[Z_i]\right\| \geq t\right]
\leq Ce^{-cn\min(t,t^2)}.
\end{equation}

For $X_i$, note that $p(\Id \mid \eps,\theta)>1-e^{-c/\sigma^2}$ when $\eps \in
\cE$, as shown in (\ref{eq:weightnearone}). Then for any unit vector $v \in \R^d$,
\begin{align*}
|v^\top X_iv|&=\left|\Var_g\Big[\langle gv,\theta_*+\sigma\eps_i \rangle \Big|
\eps_i,\theta\Big]-\sigma^2\Var_g\Big[\langle gv,\eps_i \rangle
\Big|\eps_i,\theta\Big]\right|\1\{\eps_i \in \cE\}\\
&\leq \left(\E_g\Big[\langle gv-v,\theta_*+\sigma\eps_i \rangle^2 \Big|
\eps_i,\theta\Big]
+\sigma^2 \E_g\Big[\langle gv-v,\eps_i \rangle^2 \Big|\eps_i,\theta\Big]\right)
\1\{\eps_i \in \cE\}\\
&\leq \P_g[g \neq \Id \mid \eps_i,\theta] \Big(4\|\theta_*+\sigma\eps_i\|^2
+4\sigma^2\|\eps_i\|^2\Big)\1\{\eps_i \in \cE\} \leq Ce^{-c/\sigma^2}.
\end{align*}
Thus $\|X_i\| \leq Ce^{-c/\sigma^2}$ for each $i=1,\ldots,n$. Applying
Hoeffding's inequality entrywise to $X_i$ and taking a union bound over all
entries,
\begin{equation}\label{eq:Xibound}
\P\left[\left\|\frac{1}{n}\sum_{i=1}^n X_i-\E[X_i]\right\| \geq t\sigma^2\right]
\leq C\exp\left(-ne^{c'/\sigma^2}t^2\right).
\end{equation}

For $Y_i$, let us fix indices $j,k \in \{1,\ldots,d\}$ and consider
$\sum_i (Y_i)_{jk}$. Let $W_1,\ldots,W_m$ be i.i.d.\ random
variables whose law is that of $(Y_i)_{jk}$ conditional on $\eps_i
\notin \cE$. We apply Hoeffding's inequality for $W_1,\ldots,W_m$: 
Observe that since the two quadratic terms in $\eps_i$ cancel in the
definition of $Y_i$, we have $|(Y_i)_{jk}| \leq C(1+\sigma\|\eps_i\|)$ for a constant
$C=C(\|\theta_*\|)>0$. Then
\begin{align*}
\E\left[\exp\left(\frac{W_i^2}{t^2}\right)\right]
&\leq \E_\eps\left[\exp\left(\frac{C(1+\sigma^2\|\eps\|^2)}{t^2}\right) \bigg|
\eps \notin \cE\right]\\
&=e^{C/t^2} \cdot \E_\eps\left[\exp\left(\frac{C\sigma^2\|\eps\|^2}{t^2}\right)
\Bigg|\|\eps\|^2>\left(\frac{c_0}{2\sigma\|\theta\|}\right)^2\right].
\end{align*}
Specializing \cite[Eq.\ (2.9)]{coffey2000properties} to the chi-squared
distribution, we obtain
\[\E\left[\exp(s\|\eps\|^2) \mid \|\eps\|^2>x\right]
=\frac{\P[\|\eps\|^2>x(1-2s)]}{\P[\|\eps\|^2>x]}(1-2s)^{-d/2}\]
for $s<1/2$. Here $\P[\|\eps\|^2>x]=\Gamma(d/2,x/2)/\Gamma(d/2)$ where
$\Gamma(a,y)$ is the upper-incomplete Gamma function which satisfies
$\Gamma(a,y)/y^{a-1}e^{-y} \to 1$ as $y \to \infty$, for fixed $a$
(see \cite[Eq.\ (6.5.32)]{abramowitz1948handbook}). Then
\[\frac{\P[\|\eps\|^2>x(1-2s)]}{\P[\|\eps\|^2>x]}
\cdot (1-2s)^{-d/2+1}e^{-xs} \to 1\]
as $x \to \infty$, uniformly over $s \in (0,1/2)$. Setting
$x=c_0^2/(2\sigma\|\theta\|)^2$ and $t=C_1$ for a large enough constant $C_1>0$,
we obtain that $C/t^2<0.05$, $s \equiv C\sigma^2/t^2<0.05/x$, and hence
$\E[\exp(W_i^2/t^2)] \leq 2$ when $\sigma<\sigma_0$ for small enough
$\sigma_0>0$. Thus $\|W_i\|_{\psi_2} \leq C_1$, and Hoeffding's inequality
yields, for a constant $c>0$ and any $s \geq 0$,
\[\P\left[\left|\frac{1}{m}\sum_{i=1}^m W_i-\E[W_i]\right| \geq s\right]
\leq 2e^{-cms^2}.\]

Returning to $(Y_i)_{jk}$, let $S=\{i \in [n]:\eps_i \notin \cE\}$.
The above shows that, conditional on $S$,
\[\P\left[\left|\frac{1}{|S|}\sum_{i \in S} (Y_i)_{jk}\right| \geq s+|\E W_i|
\;\Bigg|\;S\right] \leq 2e^{-c|S|s^2}.\]
Noting that $(Y_i)_{jk}=0$ when $i \notin S$, this implies
\[\P\left[\left|\frac{1}{n}\sum_{i=1}^n (Y_i)_{jk}-\E[(Y_i)_{jk}]\right| \geq
\big(s+|\E W_i|\big)\frac{|S|}{n}+\big|\E (Y_i)_{jk}\big|
\;\Bigg|\;S\right] \leq 2e^{-c|S|s^2}.\]
We have $\P[\eps_i \notin \cE] \leq e^{-c/\sigma^2}$, by a chi-squared tail
bound. From the bound $\|W_i\|_{\psi_2} \leq C_1$, we have $|\E W_i| \leq C$.
Then also $\E (Y_i)_{jk}=(\E W_i) \cdot \P[\eps_i \notin \eps] \leq
Ce^{-c/\sigma^2}$. Setting $t\sigma^2=s|S|/n$,
\[\P\left[\left|\frac{1}{n}\sum_{i=1}^n (Y_i)_{jk}-\E[(Y_i)_{jk}]\right| \geq
t\sigma^2+Ce^{-c'/\sigma^2} \;\Bigg|\;S\right]
\leq 2e^{-cn^2\sigma^4t^2/|S|}\]
for some constants $C,c,c'>0$. On the event $|S| \leq n\sigma^3$, we obtain the
bound $2e^{-cn\sigma t^2}$. By a Chernoff bound, $\P[|S|>n\sigma^3] \leq
\exp(-n\,D_{\text{KL}}(\sigma^3||e^{-c/\sigma^2}))$ for the Bernoulli
relative entropy
\[D_{\text{KL}}(\sigma^3||e^{-c/\sigma^2})=\sigma^3
\log\frac{\sigma^3}{e^{-c/\sigma^2}}
+(1-\sigma^3)\log\frac{1-\sigma^3}{1-e^{-c/\sigma^2}}
\geq c'\sigma.\]
Combining these, we obtain unconditionally that
\begin{equation}\label{eq:Yibound}
\P\left[\left|\frac{1}{n}\sum_{i=1}^n (Y_i)_{jk}-\E[(Y_i)_{jk}]\right| \geq
t\sigma^2+Ce^{-c'/\sigma^2}\right] \leq Ce^{-cn\sigma t^2}.
\end{equation}

Picking a sufficiently small constant $t$ in (\ref{eq:Zibound}),
(\ref{eq:Xibound}), and (\ref{eq:Yibound}) and applying this to
(\ref{eq:hessRndecomp}), we obtain $\|\nabla^2 R_n(\theta)-\nabla^2 R(\theta)\|
\leq c_1/(2\sigma^2)$ with probability at least $1-Ce^{-c\sigma n}$. This is a
pointwise bound for each $\theta \in B_\rho(\theta_*)$. Taking a union bound over a
$\delta$-net of this ball for $\delta=c\sigma^4$, and applying the Lipschitz
continuity of $\nabla^2 R(\theta)$ and $\nabla^2 R_n(\theta)$ from Lemma
\ref{lem:local_lipschitz}, we get the uniform bound
$\sup_{\theta \in B_\rho(\theta_*)}
\|\nabla^2 R_n(\theta)-\nabla^2 R(\theta)\| \leq c_1/\sigma^2$
with probability $1-Ce^{-cn^{2/3}}-\sigma^{-C}e^{-c\sigma n}$ as desired.
\end{proof}

\subsection{Global landscape}\label{sec:lownoiseglobal}

\begin{theorem}\label{thm:landscapelownoise}
Let $\theta_* \in \R^d$ be such that $|\O_{\theta_*}|=|G|=K$.
There exists a $(\theta_*,d,G)$-dependent constant $\sigma_0>0$
such that as long as $\sigma<\sigma_0$, the landscape of $R(\theta)$ is globally
benign.

More quantitatively, let $\rho$ be as in Theorem \ref{thm:lownoiselocal}.
Then there is a $(\theta_*,d,G)$-dependent constant $c>0$ and a decomposition
$\R^d \setminus \bigcup_{\mu \in \O_{\theta_*}} B_\rho(\mu) \equiv \cA \sqcup \cB$,
where for $\theta \in \cA$
\begin{equation}\label{eq:Acondition}
\lambda_{\min}(\nabla^2 R(\theta))<-c/\sigma^3,
\end{equation}
and for $\theta \in \cB$
\begin{equation}\label{eq:Bcondition}
\|\nabla R(\theta)\|>c/\sigma^2.
\end{equation}
\end{theorem}

Let us provide some intuition for the proof:
Recall the reweighted law (\ref{eq:gcondlaw}) for $g \in G$. We enumerate
\[G=\{g_1,\ldots,g_K\},\]
fix a small constant $\tau>0$, and divide the space of $\eps \in \R^d$ into the regions
\begin{align}
\cE_i(\theta,\tau)&=\Big\{\eps \in \R^d:p(g_k \mid \eps,\theta) \leq \tau \text{ for all }
k \in \{1,\ldots,K\} \setminus \{i\}\Big\},\label{eq:Ei}\\
\cE_{ij}(\theta,\tau)&=\Big\{\eps \in \R^d:p(g_i \mid \eps,\theta)>\tau
\text{ and } p(g_j \mid \eps,\theta)>\tau\Big\}.\label{eq:Eij}
\end{align}
Here, for $\tau$ small enough, $\cE_i(\theta,\tau)$ is the space of noise
vectors $\eps$ for which the $\eps$-dependent
distribution (\ref{eq:gcondlaw}) places nearly all
of its weight on $g_i$, and $\cE_{ij}(\theta,\tau)$ is the space of $\eps$ for
which this distribution ``straddles'' its weight between at least two points
$g_i \neq g_j \in G$.

We will choose the set $\cB$ in Theorem \ref{thm:landscapelownoise}
to be those vectors $\theta \in \R^d$ for which
$\P[\eps \in \cE_i(\theta,\tau)] \approx 1$ for some $i \in \{1,\ldots,K\}$.
Thus, for some fixed $g_i \in G$, with high probability over $\eps$, the
law (\ref{eq:gcondlaw}) places nearly all of its weight on the single element $g_i$.
Intuitively, from the form (\ref{eq:gcondlaw}),
these are the points $\theta \in \R^d$ which are closer to
$g_i^\top \theta_*$ than to the other points $g_j^\top \theta_*$ for $j \neq i$.

The remaining points $\R^d \setminus \cB$ will constitute $\cA$.
A key step of the proof is to show that if $\theta \notin \cB$,
then there must be a pair $i \neq j$ for which
$\P[\eps \in \cE_{ij}(\theta,\tau)] \gtrsim \sigma$. That is,
with some small probability of order $\sigma$, the law (\ref{eq:gcondlaw})
straddles its weight between $g_i$ and $g_j$.
(Note that this is not tautological from the definitions, as we must rule out the
possibility, e.g., that $\P[\eps \in \cE_i(\theta,\tau)]=1/2$ and
$\P[\eps \in \cE_j(\theta,\tau)]=1/2$ for some $i \neq j$, but
$\P[\eps \in \cE_{ij}(\theta,\tau)]=0$. Indeed, from the
form of (\ref{eq:gcondlaw}), we see that even if $\theta$ is exactly equidistant
from $g_i^\top \theta_*$ and $g_j^\top \theta_*$, the probability over $\eps$ is only
$O(\sigma)$ that $p(g_i \mid \eps,\theta)$
and $p(g_j \mid \eps,\theta)$ are comparable.)
We prove this claim using a Gaussian isoperimetric argument in Lemma
\ref{lemma:boundary} below.

\begin{lemma}\label{lemma:boundary}
Fix any $\theta \neq 0$ and $\tau \in (0,(K+9)^{-1})$, and define
$\cE_i,\cE_{ij}$ by (\ref{eq:Ei}) and (\ref{eq:Eij}). Suppose, for some
$i \in \{1,\ldots,K\}$ and $p \in (0,1/2]$, that
\[p \leq \P[\eps \in \cE_i] \leq 1/2.\]
Then for some $j \in \{1,\ldots,K\} \setminus \{i\}$,
\[\P[\eps \in \cE_{ij}] \geq \frac{p}{(K-1)\sqrt{2\pi}}
\min\left(\frac{\sigma}{\|\theta\|},1\right).\]
\end{lemma}

\begin{proof}
Let $\cE_i^t=\{\eps \in \R^d:\dist(\eps,\cE_i)<t\}$.
We first claim that if $\eps \in \cE_i^t \setminus \cE_i$
for $t=\sigma/\|\theta\|$, then
there exists some $j \neq i$ for which $\eps \in \cE_{ij}$. For this, note that
\[\nabla_\eps [\,\log p(g \mid \eps,\theta)\,]
=\frac{1}{\sigma}\Big(g\theta-\E_h[ h\theta \mid \eps,\theta]\Big),\]
so $\eps \mapsto \log p(g_i \mid \eps,\theta)$ has the Lipschitz bound
$\|\nabla_\eps \log p(g_i \mid \eps,\theta)\|
\leq 2\|\theta\|/\sigma$.
Suppose that $\eps \in \cE_i^t \setminus \cE_i$.
Then there is $\eps' \in \cE_i$ with $\|\eps-\eps'\|<\sigma/\|\theta\|$, so
$\log p(g_i \mid \eps',\theta)-\log p(g_i \mid \eps,\theta) \leq 2$ and
\[p(g_i \mid \eps',\theta)/p(g_i \mid \eps,\theta) \leq e^2<8.\]
Since $p(g_1 \mid \eps',\theta)+\ldots+p(g_K \mid \eps',\theta)=1$
and $(K+9)\tau<1$, when $\eps' \in \cE_i$
we must have $p(g_i \mid \eps',\theta) \geq 1-(K-1)\tau>8\tau$.
Then the above implies $p(g_i \mid \eps,\theta)>\tau$.
Since $\eps \notin \cE_i$, by definition of $\cE_i$
we must also have $p(g_j \mid \eps,\theta)>\tau$ for
some $j \neq i$, so that $\eps \in \cE_{ij}$ as desired. Note that this index
$j \in \{1,\ldots,K\}$ may depend on $\eps$. However, this shows that
for at least one \emph{fixed} index $j \in \{1,\ldots,K\} \setminus \{i\}$,
\begin{equation}\label{eq:ZiZjbound}
\P[\eps \in \cE_{ij}]
\geq \frac{\P[\eps \in \cE_i^t \setminus \cE_i]}{K-1}.
\end{equation}

We now apply the Gaussian isoperimetric inequality to lower bound the right
side: For $\Phi$ the standard normal distribution function,
\[\Phi^{-1}(\P[\eps \in \cE_i^t]) \geq \Phi^{-1}(\P[\eps \in \cE_i])+t,\]
see \cite[Theorem 10.15]{boucheron2013concentration}.
Then, denoting by $\phi$ the standard normal density,
\begin{align*}
\P[\eps \in \cE_i^t \setminus \cE_i]
=\P[\eps \in \cE_i^t]-\P[\eps \in \cE_i]
&\geq \Phi(\Phi^{-1}(\P[\eps \in \cE_i])+t)
-\Phi(\Phi^{-1}(\P[\eps \in \cE_i]))\\
&=\int_{\Phi^{-1}(\P[\eps \in \cE_i])}^{\Phi^{-1}(\P[\eps \in \cE_i])+t}
\phi(r)dr.
\end{align*}
Applying $\P[\eps \in \cE_i] \in [p,1/2]$ by assumption,
we get $\Phi^{-1}(\P[\eps \in \cE_i]) \in [\Phi^{-1}(p),0]$. Then
there is always an interval of values for $r$, having length
$\min(t,1)$ and contained in the above range of integration,
for which $\phi(r) \geq \min(\phi(\Phi^{-1}(p)),\phi(1))$ over this
interval. Applying the tail bound $\Phi(x) \leq e^{-x^2/2}$ for all
$x \leq 0$, we get $\Phi^{-1}(p) \geq -\sqrt{2\log 1/p}$ and
$\phi(\Phi^{-1}(p)) \geq p/\sqrt{2\pi}$. For $p \leq 1/2$ we have
$p/\sqrt{2\pi}<\phi(1)$. Combining these observations gives
\[\P[\eps \in \cE_i^t \setminus \cE_i] \geq \min(t,1)\cdot
\frac{p}{\sqrt{2\pi}}.\]
Recalling $t=\sigma/\|\theta\|$
and combining with (\ref{eq:ZiZjbound}) yields the lemma.
\end{proof}

\begin{proof}[Proof of Theorem \ref{thm:landscapelownoise}]
Let us fix two positive constants
\begin{equation}\label{eq:tau}
\tau<\min\left(\frac{1}{K+9},\;\frac{\rho}{8\|\theta_*\|K}\right)
\end{equation}
and
\begin{equation}\label{eq:p}
p<\left(\frac{\rho}{12\|\theta_*\|}\right)^2\Bigg/K.
\end{equation}
Define $\cE_i(\theta,\tau)$ and $\cE_{ij}(\theta,\tau)$ by (\ref{eq:Ei})
and (\ref{eq:Eij}) with this choice of $\tau$, and set
\begin{align*}
\cA&=\left\{\theta \in \R^d \setminus \cC:
\P[\eps \in \cE_{ij}(\theta,\tau)]>\frac{p}{K\sqrt{2\pi}} \cdot
\frac{\sigma}{3\|\theta_*\|} \text{ for some } i \neq j\right\},\\
\cB&=\left\{\theta \in \R^d \setminus \cC:
\P[\eps \in \cE_{ij}(\theta,\tau)] \leq \frac{p}{K\sqrt{2\pi}} \cdot
\frac{\sigma}{3\|\theta_*\|} \text{ for all } i \neq j\right\}.
\end{align*}

To check (\ref{eq:Acondition}) when $\theta \in \cA$, recall the form of
$\nabla^2 R(\theta)$ in (\ref{eq:hessR}). We apply
$\P[\eps \in \cE_{ij}(\theta,\tau)]>c\sigma$ for a constant $c>0$ and
some $i \neq j$, by the definition of $\cA$. Choose a constant $c_0>0$ such that
$\|g_i^\top \theta_*-g_j^\top \theta_*\|>3c_0$.
Then a chi-squared tail bound yields
\begin{equation}\label{eq:goodepsilon}
\P\Big[\|\eps\| \leq c_0/\sigma \text{ and }
\eps \in \cE_{ij}(\theta,\tau) \Big]>c'\sigma
\end{equation}
for a different constant $c'<c$ and all $\sigma<\sigma_0$.
For $\eps$ satisfying (\ref{eq:goodepsilon}), we have
\[\|g_i^\top (\theta_*+\sigma \eps)-g_j^\top(\theta_*+\sigma \eps)\|
\geq \|g_i^\top \theta_*-g_j^\top \theta_*\|-2\sigma \|\eps\| \geq c_0,\]
and also $p(g_i \mid \eps,\theta)>\tau$ and $p(g_j \mid \eps,\theta)>\tau$.
Then for such $\eps$, denoting $\mu=\E_g[g^\top (\theta_*+\sigma \eps) \mid
\eps,\theta]$, we have
\begin{align*}
\Tr \Cov_g[g^\top (\theta_*+\sigma \eps) \mid \eps,\theta]
&=\E_g[\|g^\top (\theta_*+\sigma \eps)-\mu\|^2 \mid \eps,\theta]\\
& \geq \tau \cdot \|g_i^\top (\theta_*+\sigma \eps)-\mu\|^2
+\tau \cdot \|g_j^\top (\theta_*+\sigma \eps)-\mu\|^2>c.
\end{align*}
Combining this with (\ref{eq:goodepsilon}) implies that
\[\lambda_{\max}\left(
\E_\eps\Big[\Cov_g\Big[g^\top (\theta_*+\sigma\eps) \;\Big|\;
\eps,\theta\Big]\Big]\right)>c\sigma.\]
Then (\ref{eq:Acondition}) follows from (\ref{eq:hessR}).

To check (\ref{eq:Bcondition}) when $\theta \in \cB$, note that if
$\|\theta\| \geq 3\|\theta_*\|$, then (\ref{eq:Bcondition}) follows from
Lemma \ref{lemma:thetabound}. For $\theta \in \cB$ such that
$\|\theta\|<3\|\theta_*\|$, the definition of $\cB$ and
Lemma \ref{lemma:boundary} imply that either $\P[\eps \in \cE_i(\theta,\tau)]<p$
or $\P[\eps \in \cE_i(\theta,\tau)]>1/2$ for every $i \in \{1,\ldots,K\}$.
Note that since $K\tau<1$, we must have:
\begin{itemize}
\item $\cE_1(\theta,\tau),\ldots,\cE_K(\theta,\tau)$ are disjoint.
\item $\{\cE_i(\theta,\tau)\}_{i=1}^K$ and $\{\cE_{ij}(\theta,\tau)\}_{i \neq
j}$ together cover all of $\R^d$.
\end{itemize}
The first observation implies that
$\P[\eps \in \cE_i(\theta,\tau)]>1/2$ for at most one index
$i \in \{1,\ldots,K\}$, so we must have
$\P[\eps \in \cE_j(\theta,\tau)]<p$ for all other $j \neq i$.
Combining this with the second observation,
\[1 \leq \P[\eps \in \cE_i(\theta,\tau)]+\sum_{j:j \neq i} \P[\eps \in
\cE_j(\theta,\tau)]+\sum_{j \neq k} \P[\eps \in \cE_{jk}(\theta,\tau)]
\leq \P[\eps \in \cE_i]+(K-1)p+\binom{K}{2}c\sigma.\]
For $\sigma<\sigma_0$ and sufficiently small $\sigma_0$,
this implies $\P[\eps \in \cE_i(\theta,\tau)] \geq 1-Kp$.

Recall the form (\ref{eq:gradR}) for $\nabla R(\theta)$. For this
index $i$, let us write
\[\E_\eps\Big[\E_g[g^\top (\theta_*+\sigma\eps) \mid \eps,\theta]\Big]-g_i^\top \theta_*=\mathrm{I}+\mathrm{II}+\mathrm{III}\]
where
\begin{align*}
\mathrm{I}&=\E_\eps\Big[\1\{\eps \notin \cE_i\}\Big(\E_g[g^\top
(\theta_*+\sigma \eps) \mid \eps,\theta]-g_i^\top \theta_*\Big)\Big],\\
\mathrm{II}&=\E_\eps\Big[\1\{\eps \in \cE_i\}\Big(\E_g[\1\{g \neq g_i\}g^\top
(\theta_*+\sigma \eps) \mid \eps,\theta]\Big)\Big],\\
\mathrm{III}&=\E_\eps\Big[\1\{\eps \in \cE_i\}\Big(\E_g[\1\{g=g_i\}g^\top
(\theta_*+\sigma\eps) \mid \eps,\theta]-g_i^\top \theta_*\Big)\Big].
\end{align*}
Applying Cauchy-Schwarz, the above bound $\P[\eps \in \cE_i(\theta,\tau)]
\geq 1-Kp$,
and the condition (\ref{eq:p}) for $p$, we get for $\sigma<\sigma_0$ and small
enough $\sigma_0$ that
\[\|\mathrm{I}\| \leq \P[\eps \notin \cE_i]^{1/2}
\E_\eps[(\|\theta_*+\sigma\eps\|+\|\theta_*\|)^2]^{1/2}
\leq (Kp)^{1/2} \cdot 3\|\theta_*\|<\rho/4.\]
When $\eps \in \cE_i$, we have $\P_g[g=g_i \mid \eps,\theta]
=p(g_i \mid \eps,\theta)>1-K\tau$. Then by the condition (\ref{eq:tau}) for
$\tau$, for $\sigma<\sigma_0$,
\[\|\mathrm{II}\| \leq \E_\eps\Big[\1\{\eps \in \cE_i\}
\P_g[g \neq g_i \mid \eps,\theta] \cdot \|\theta_*+\sigma \eps\|\Big]
\leq K\tau \cdot 2\|\theta_*\|<\rho/4.\]
For $\mathrm{III}$, we cancel $g_i^\top \theta_*$ to get the bound
\[\|\mathrm{III}\| \leq \E_\eps\Big[\E_g[\1\{g=g_i\}
\|g^\top (\sigma \eps)\| \mid \eps,\theta]\Big] \leq \sigma\,\E_\eps[\|\eps\|]
<\rho/4.\]
Combining these with (\ref{eq:gradR}) yields
\[\|\sigma^2 \nabla R(\theta)-(\theta-g_i^\top \theta_*)\|<3\rho/4,\]
and (\ref{eq:Bcondition}) follows since
$\|\theta-g_i^\top \theta_*\| \geq \rho$ because $\theta \notin \bigcup_{\mu \in
\O_{\theta_*}} B_\rho(\mu)$.
These conditions (\ref{eq:Acondition}), (\ref{eq:Bcondition}), and Theorem
\ref{thm:lownoiselocal} together show that the landscape of $R(\theta)$ is
globally benign.
\end{proof}

The following then shows that the landscape of $R_n(\theta)$ is also globally
benign with high probability, when $n \gg \sigma^{-2} \log \sigma^{-1}$.

\begin{corollary}\label{cor:landscapelownoise}
In the setting of Theorem \ref{thm:landscapelownoise}, the same statements hold
for the empirical risk $R_n(\theta)$ with probability at least
$1-\sigma^{-C}e^{-c\sigma^2 n}-Ce^{-cn^{2/3}}$.
\end{corollary}
\begin{proof}
For $\sigma<\sigma_0$ and small enough $\sigma_0$, with probability
$1-Ce^{-cn}$, we have $\|\nabla R_n(\theta)\|
\geq c/\sigma^2$ for all $\theta$ such that $\|\theta\|>3\|\theta_*\|$ by Lemma
\ref{lemma:thetabound}. Applying the concentration result
(\ref{eq:concgradRn}) with $t=c_0/\sigma^2$, and (\ref{eq:conchessRn}) with
$t=c_0/\sigma^3$, over the ball $B_r$ for $r=3\|\theta_*\|$, for small enough
$c_0$ we obtain (\ref{eq:Acondition}) and (\ref{eq:Bcondition}) also for 
the empirical risk $R_n(\theta)$, with probability $1-\sigma^{-C}e^{-c\sigma^2
n}-Ce^{-cn^{2/3}}$. The result then follows from combining with
Corollary \ref{cor:lownoiselocal}.
\end{proof}

\begin{remark}
For $\theta \in
\R^d$ roughly equidistant to multiple points of the orbit $\O_{\theta_*}$, the
weights $p(g \mid \eps,\theta)$ do not concentrate with high probability on a
single deterministic rotation $g \in G$, so we do not obtain the same
refinement of the concentration probability as in
Corollary \ref{cor:lownoiselocal} for the local analysis near $\theta_*$.
\end{remark}

\section{Landscape analysis for high noise}\label{sec:largesigma}

In this section, we analyze the function landscapes of $R(\theta)$ and
$R_n(\theta)$ in the high-noise regime $\sigma>\sigma_0(\theta_*,d,G)$. Our
results relate to the algebra of $G$-invariant polynomials and systems of
reparametrized coordinates in local neighborhoods, which we first review in
Section \ref{sec:invariants}.

Our analysis for high noise is based on showing that truncations of
the formal $\sigma^{-1}$-series
\begin{equation}\label{eq:formalseries}
\sum_{\ell=1}^\infty \sigma^{-2\ell} S_\ell(\theta)
\end{equation}
provide asymptotic estimates for the population risk $R(\theta)$. We derive this in
Lemma \ref{lemma:seriesexpansion} using the series
expansion of the cumulant generating function
$\log \E_g \exp(\langle \theta_*+\sigma \eps,g\theta \rangle/\sigma^2)$
in (\ref{eq:R}).
We quantify the accuracy of the approximation to $R(\theta)$ by
bounding its deviation from the first $k$ terms of its
formal series for any fixed $k$ as $\sigma \to \infty$. To analyze the
concentration of the empirical risk, we provide a similar series expansion for
$R_n(\theta)$ in Lemma \ref{lemma:Rnexpansion}.

The functions $S_\ell(\theta)$ in (\ref{eq:formalseries}) do not depend on
$\sigma$, and we analyze the form of these terms also in Section
\ref{sec:seriesexpansion}. We show in Section
\ref{sec:localreparam} that the local landscape of $R(\theta)$ around any
point $\wtheta \in \R^d$ may be understood, for large $\sigma$, by analyzing the
successive landscapes of these functions $S_\ell(\theta)$ in a reparametrized
system of coordinates near $\wtheta$.

In Section \ref{sec:locallandscape}, we apply this at $\wtheta=\theta_*$ to analyze
the local landscape near $\theta_*$. Theorem \ref{thm:locallargenoise} and
Corollary \ref{cor:locallargenoise} show that $R(\theta)$ is
strongly convex in a $\sigma$-independent neighborhood of $\theta_*$,
when reparametrized by a transcendence basis of the $G$-invariant
polynomial algebra. The same holds with high probability for
$R_n(\theta)$ when $n \gg \sigma^{2L}$, where $L$ is the smallest
integer for which $\trdeg(\cR_{\leq L}^G)=d$.
Theorem \ref{thm:locallargenoise} also shows that $I(\theta_*)$
has a certain graded structure, where the magnitudes of its eigenvalues
correspond to a sequence of transcendence degrees in this algebra.

In Section \ref{sec:globalbenign}, we patch together the local results
of Section \ref{sec:localreparam} to study the global
landscapes of $R(\theta)$ and $R_n(\theta)$.
Theorems \ref{thm:rotations}, \ref{thm:permutations} and Corollaries
\ref{cor:rotations}, \ref{cor:permutations} establish globally benign landscapes
for $K$-fold discrete rotations on $\R^2$ and the symmetric group of all
permutations on $\R^d$, for large $\sigma$ and large $n$. Theorem \ref{thm:globalbenign} then
generalizes this to a more abstract condition, in terms of minimizing the sequence of
polynomials $P_\ell(\theta)$ in (\ref{eq:momentobjective}) over the sequence of
moment varieties $\cV_{\ell-1}$ in (\ref{eq:momentvariety}), and shows
that the empirical landscape of $R_n(\theta)$ inherits the benign property of
$R(\theta)$ also when $n \gg \sigma^{2L}$.

Finally, in Section \ref{sec:MRA}, we analyze the global landscape for
cyclic permutations on $\R^d$ (i.e.\ multi-reference
alignment). Theorem \ref{thm:MRA} and Corollary \ref{cor:MRA} show that the local
minimizers of $R(\theta)$ and $R_n(\theta)$ are in correspondence with those of a
minimization problem in phase space. Corollary \ref{cor:MRAbad} shows that their
landscapes are benign in dimensions $d \leq 5$ (for large $\sigma$ and large
$n$), but may not be benign even for generic $\theta_*$ for even $d \geq 6$
and odd $d \geq 53$.

\subsection{Invariant polynomials and local reparametrization}\label{sec:invariants}

\begin{definition}
For a subgroup $G \subseteq \rO(d)$, a polynomial function $\varphi:\R^d \to \R$ is
{\bf $G$-invariant} if $\varphi(g\theta)=\varphi(\theta)$ for all $g \in G$.
We denote by $\cR^G$ the algebra (over $\R$) of all $G$-invariant polynomials
on $\R^d$, and by $\cR^G_{\leq \ell} \subset \cR^G$ the vector space of such
polynomials having degree $\leq \ell$.
\end{definition}

\begin{definition}
Polynomials $\varphi_1,\ldots,\varphi_k:\R^d \to \R$ are
{\bf algebraically independent} (over $\R$) if there is no non-zero polynomial
$P:\R^k \to \R$ for which $P(\varphi_1(\theta),\ldots,\varphi_k(\theta))$ is
identically 0 over $\theta \in \R^d$.
For a subset $A \subseteq \cR^G$, its {\bf transcendence degree}
$\trdeg(A)$ is the maximum number of algebraically independent elements in $A$.
\end{definition}

One may construct a transcendence basis of $d$ such polynomials according to the
following lemma; we provide a proof for convenience in Appendix \ref{appendix:transcendence}.

\begin{lemma}\label{lemma:phiconstruction}
For any finite subgroup $G \subset \rO(d)$, there exists a smallest integer
$L \geq 1$ for which $\trdeg(\cR_{\leq L})=d$. Writing $d=d_1+\ldots+d_L$ where
\[d_\ell=\trdeg(\cR_{\leq \ell}^G)-\trdeg(\cR_{\leq \ell-1}^G),\]
there also exist $d$ algebraically independent $G$-invariant polynomials
$\varphi=(\varphi^1,\ldots,\varphi^L)$, where each subvector $\varphi^\ell$
consists of $d_\ell$ polynomials having degree exactly $\ell$.
\end{lemma}

It was shown in \cite{bandeira2017estimation} that this number $L$ is
the highest-order moment needed for a moment-of-moments estimator to recover a
generic signal $\theta_*$ in the model (\ref{eq:orbitmodel}), up to a finite
list of possibilities including (but not necessarily limited to) the orbit 
points $\O_{\theta_*}$, and that the number of samples required for this type of
recovery scales as $\sigma^{2L}$.

In our local analysis around a point $\wtheta \in \R^d$, we will switch to a
system of reparametrized coordinates. Let us specify our notation for such a
reparametrization.

\begin{definition}
A function $\varphi:\R^d \to \R^d$ is a {\bf local reparametrization} in an open
neighborhood $U$ of $\wtheta \in \R^d$ if $\varphi$ is
1-to-1 on $U$ with inverse function $\theta(\varphi)$,
and $\varphi(\theta)$ and $\theta(\varphi)$ are analytic respectively on $U$ and
$\varphi(U)$.
\end{definition}

If $\varphi$ is a local reparametrization, then $\der_\theta \varphi$ is
non-singular and equal to $(\der_\varphi \theta)^{-1}$ at each $\theta \in U$.
Conversely, by the inverse function theorem, if $\varphi(\theta)$ is analytic and 
$\der_\theta \varphi(\wtheta)$ is non-singular, then there is such an open 
neighborhood $U$ of $\wtheta$ on which $\varphi$ defines a local reparametrization.

To ease notation, we write (with a slight abuse) $f(\varphi)$ for
$f(\theta(\varphi))$ when the meaning is clear, and we write
$\nabla_\varphi f(\varphi)$, $\nabla_\varphi^2 f(\varphi)$, and
$\partial_{\varphi_i} f(\varphi)$ for the 
gradient, Hessian, and partial derivatives of $f(\varphi)$ with respect to $\varphi$.
For a decomposition $\varphi=(\varphi^1,\ldots,\varphi^L)$ of dimensions
$d_1,\ldots,d_L$, we denote by $\nabla_{\varphi^\ell} f(\varphi) \in \R^{d_\ell}$ and
$\nabla_{\varphi^\ell}^2 f(\varphi) \in \R^{d_\ell \times d_\ell}$ the
subvectors and submatrices of $\nabla_\varphi f(\varphi)$ and $\nabla^2_\varphi
f(\varphi)$ corresponding to the coordinates in $\varphi^\ell$.

Recalling $\nabla_\theta f(\theta)=\der_\theta f(\theta)^\top$, by the chain rule
and product rule, we have
\begin{align}
\nabla_\theta f(\theta)&=(\der_\theta \varphi)^\top
\nabla_\varphi f(\varphi)\label{eq:gradreparam}\\
\nabla_\theta^2 f(\theta)&=
(\der_\theta \varphi)^\top \cdot \nabla_\varphi^2 f(\varphi) \cdot \der_\theta \varphi
+\sum_{i=1}^d \partial_{\varphi_i}f(\varphi) \cdot \nabla_\theta^2
\varphi_i\label{eq:hessreparam}
\end{align}
Note that $\nabla_\theta f(\wtheta)=0$ if and only if $\nabla_\varphi
f(\wvarphi)=0$ for $\wvarphi=\varphi(\wtheta)$, i.e.\ critical points do not depend 
on the choice of parametrization. At a critical point $\wtheta$ of $f(\theta)$,
letting $\wvarphi=\varphi(\wtheta)$, the identity (\ref{eq:hessreparam}) simplifies to
just the first term,
\[\nabla_\theta^2 f(\wtheta)=(\der_\theta \varphi(\wtheta))^\top \cdot
\nabla_\varphi^2 f(\wvarphi) \cdot \der_\theta \varphi(\wtheta),\]
so that the rank and signs of the eigenvalues of $\nabla_\theta^2 f(\wtheta)$
also do not depend on the choice of parametrization. This may be
false when $\theta$ is not a critical point---in particular, strong convexity
of $f(\varphi)$ as a function of $\varphi \in \varphi(U)$ does not imply
strong convexity of $f(\theta)$ as a function of $\theta \in U$.

For analyzing specific groups, we will explicitly describe
our reparametrization $\varphi$. For more general results, we will reparametrize by
the transcendence basis of polynomials $\varphi$ in Lemma \ref{lemma:phiconstruction}.
The following clarifies the relationship between algebraic independence of these
polynomials and linear independence of their gradients, and implies
in particular that $\varphi$ is a local reparametrization at generic points
of $\R^d$. We provide a proof also in Appendix \ref{appendix:transcendence}.

\begin{lemma}\label{lemma:polynomialreparam}
Let $G \subset \rO(d)$ be any
subgroup, and let $\varphi_1,\ldots,\varphi_k$ be polynomials in $\cR^G$.
\begin{enumerate}[(a)]
\item If $\varphi_1,\ldots,\varphi_k$ are algebraically independent, then
$\nabla \varphi_1,\ldots,\nabla \varphi_k$ are linearly
independent at generic points $\theta \in \R^d$.
\item If $\nabla \varphi_1,\ldots,\nabla \varphi_k$ are linearly
independent at any point $\theta \in \R^d$, then $\varphi_1,\ldots,\varphi_k$
are algebraically independent.
\item If $\nabla \varphi_1,\ldots,\nabla \varphi_k$ are linearly
independent at a point $\wtheta \in \R^d$, and $\varphi_1,\ldots,\varphi_k \in
\cR^G_{\leq \ell}$ with $k=\trdeg(\cR^G_{\leq \ell})$, then there is an open
neighborhood $U$ of $\wtheta$ such that for every polynomial $\psi \in \cR^G_{\leq
\ell}$, there is an analytic function $f:\R^k \to \R$ for which
$\psi(\theta)=f(\varphi_1(\theta),\ldots,\varphi_k(\theta))$ for all $\theta \in U$.
\end{enumerate}
\end{lemma}

\subsection{Series expansion of the population risk}\label{sec:seriesexpansion}

For any partition $\pi$ of $[\ell+m] \equiv \{1,\ldots,\ell+m\}$, denote by
$|\pi|$ the number of sets in $\pi$, and label these sets as
$1,\ldots,|\pi|$. For each $i \in [\ell+m]$, denote by
$\pi(i) \in \{1,\ldots,|\pi|\}$ the index of the set containing element $i$. 
For $0 \leq m \leq \ell$, define
\begin{equation}\label{eq:Ml}
M_{\ell,m}(\pi \mid \theta,\theta_*)=\E_{g_1,\ldots,g_{|\pi|}}
\left[\prod_{j=1}^m \Big\langle g_{\pi(2j-1)}\theta,\;g_{\pi(2j)}\theta
\Big\rangle \cdot \prod_{j=2m+1}^{\ell+m} \Big\langle \theta_*,
\;g_{\pi(j)}\theta \Big\rangle\right]
\end{equation}
where the expectation is over independent group elements
$g_1,\ldots,g_{|\pi|} \sim \Unif(G)$.

\begin{example}\label{ex:Mlm}
Consider $\ell=3$, $m=1$, and $\pi=\{\{1,2\},\{3,4\}\}$. For this partition
$\pi$, we have $|\pi|=2$ and $(\pi(1),\pi(2),\pi(3),\pi(4))=(1,1,2,2)$.
Letting $g_1,g_2 \sim \Unif(G)$ be two independent and uniformly distributed
group elements,
\begin{align}
M_{3,1}(\pi \mid \theta,\theta_*)&=
\E_{g_1,g_2}\left[\langle g_1\theta,g_1\theta \rangle
\langle \theta_*,g_2\theta \rangle^2\right]. \label{eq:Mlm1}
\end{align}
For $\pi=\{\{1,3\},\{2\},\{4\}\}$, we have $|\pi|=3$ and
$(\pi(1),\pi(2),\pi(3),\pi(4))=(1,2,1,3)$. Then
\begin{align}
M_{3,1}(\pi \mid \theta,\theta_*)&=
\E_{g_1,g_2,g_3}\left[\langle g_1\theta,g_2\theta \rangle
\langle \theta_*,g_1\theta \rangle \langle \theta_*, g_3\theta \rangle\right].
\label{eq:Mlm2}
\end{align}
Similarly, for $\pi=\{\{1,3,4\},\{2\}\}$, we have
\begin{align}
M_{3,1}(\pi \mid \theta,\theta_*)&=
\E_{g_1,g_2}\left[\langle g_1\theta,g_2\theta \rangle
\langle \theta_*,g_1\theta \rangle^2\right].
\label{eq:Mlm3}
\end{align}\qed
\end{example}

Define the set
\begin{equation}\label{eq:goodpartition}
    \cP(\ell,m)=\Big\{\;\text{partitions } \pi \text{ of } [\ell+m]:
\;\pi(2j-1) \neq \pi(2j)  \text{ for all } j=1,\ldots,m\; \Big\}.
\end{equation}
That is, partitions $\pi \in \cP(\ell,m)$ separate each pair of elements
$\{1,2\},\{3,4\},\ldots,\{2m-1,2m\}$.  Define the quantity
\begin{equation}\label{eq:Sl}
S_\ell(\theta)=\frac{1}{\ell!}
\sum_{m=0}^\ell \frac{1}{2^m}\binom{\ell}{m} \sum_{\pi \in \cP(\ell,m)}
(|\pi|-1)!(-1)^{|\pi|}M_{\ell,m}(\pi \mid \theta,\theta_*)
\end{equation}
and the corresponding $k$-term expressions
\begin{equation}\label{eq:Rk}
R^k(\theta)=\sum_{\ell=1}^k \sigma^{-2\ell}S_\ell(\theta).
\end{equation}
The following is our rigorous result corresponding to (\ref{eq:formalseries}),
which states that $R(\theta)$ may be approximated by $R^k(\theta)$ for
$\|\theta\| \ll \sigma/\log \sigma$ and fixed $k$, as $\sigma \to \infty$.
We provide its proof at the end of this section.

\begin{lemma}\label{lemma:seriesexpansion}
Fix any function $r:(0,\infty) \to [1,\infty)$ such that
$r(\sigma) \cdot (\log \sigma)/\sigma \to 0$ as
$\sigma \to \infty$. For each $k \geq 1$, there exist $(\theta_*,d,G)$-dependent
constants $C,\sigma_0>0$ depending also on $k$,
such that for all $\sigma>\sigma_0$ and all $\theta \in \R^d$
with $\|\theta\|<r(\sigma)$,
\begin{align*}
\left|R(\theta)-R^k(\theta)\right|
&\leq \left(\frac{C\log \sigma}{\sigma}\right)^{2k+2}
(\|\theta\| \vee 1)^{2k+2}\\
\left\|\nabla R(\theta)-\nabla R^k(\theta)\right\| &\leq
\left(\frac{C\log \sigma}{\sigma}\right)^{2k+2}(\|\theta\| \vee 1)^{2k+1}\\
\left\|\nabla^2 R(\theta)-\nabla^2 R^k(\theta)\right\| &\leq
\left(\frac{C\log \sigma}{\sigma}\right)^{2k+2}(\|\theta\| \vee 1)^{2k}.
\end{align*}
\end{lemma}

From the definition in (\ref{eq:Ml}), we observe that for any fixed $\theta_*
\in \R^d$, the term $M_{\ell,m}(\pi \mid \theta,\theta_*)$ is a $G$-invariant
polynomial function of $\theta$. Counting the number of occurrences of $\theta$,
$M_{\ell,m}(\pi \mid \theta,\theta_*)$ has degree $\ell+m$ in $\theta$.
Hence, $S_\ell(\theta)$ is a $G$-invariant polynomial of degree $2\ell$.
The following shows that, in fact, $S_\ell(\theta)$ is in the algebra
generated by the polynomials $\cR^G_{\leq \ell}$ of degree at most $\ell$. (That
is, $S_\ell$ is a polynomial function of elements of $\cR^G_{\leq \ell}$.)
Furthermore, its dependence on the polynomials of degree $\ell$ has an explicit
form in terms of the moment tensor $T_\ell(\theta)=\E_g[(g\theta)^{\otimes \ell}]$
from (\ref{eq:momenttensor}). These properties will allow us to understand the
dependence of $S_\ell(\theta)$ on the transcendence basis for $\cR^G_{\leq
\ell}$ constructed in Lemma \ref{lemma:phiconstruction}.

\begin{lemma}\label{lemma:Selldegell}
For each fixed $\theta_* \in \R^d$ and each $\ell \geq 1$, we have
\begin{equation}
S_\ell(\theta)=\frac{1}{2(\ell!)}\big\|T_\ell(\theta)-T_\ell(\theta_*)
\big\|_\HS^2+Q_\ell(\theta)
\label{eq:Selldegell}
\end{equation}
where $Q_\ell(\theta)$ is a polynomial (with coefficients depending on $\theta_*$)
in the algebra generated by $\cR^G_{\leq \ell-1}$.
In particular, $S_\ell(\theta)$ is in the algebra generated by $\cR^G_{\leq \ell}$.
\end{lemma}
\begin{proof}
We consider the terms $M_{\ell,m}(\pi \mid \theta,\theta_*)$ which constitute
$S_\ell(\theta)$. For each $\pi \in \cP(\ell,m)$, applying
the constraint that $\pi(2j-1) \neq \pi(2j)$ for $j=1,\ldots,m$, we observe
that each set in the partition $\pi$ has cardinality at most $\ell$, and hence
each distinct group element $g_i$ for $i=1,\ldots,|\pi|$ appears at most
$\ell$ times inside the expectation in (\ref{eq:Ml}).

If each set in $\pi$ has cardinality at most $\ell-1$ (e.g.~\eqref{eq:Mlm1} and \eqref{eq:Mlm2} in Example~\ref{ex:Mlm}), then we claim that
$M_{\ell,m}(\pi \mid \theta,\theta_*)$ is in the generated algebra of
$\cR^G_{\leq \ell-1}$. To see this,
observe that for any $k \leq \ell-1$ and tensor $A \in (\R^d)^{\otimes k}$,
we may write
\[\E_g\left[\sum_{i_1,\ldots,i_k=1}^d
\left(\prod_{j=1}^k (g\theta)_{i_j}\right)A_{i_1,\ldots,i_k}\right]
=\E_g\Big[\big\langle (g\theta)^{\otimes k},\,A \big\rangle\Big]
=\Big\langle T_k(\theta),\,A \Big\rangle.\]
Each entry of the moment tensor $T_k(\theta)$ is a $G$-invariant
polynomial of degree $k$, and hence belongs to $\cR^G_{\leq k}$.
Applying this identity once for each distinct element $g_1,\ldots,g_{|\pi|}$ in
(\ref{eq:Ml}), and using that each such element appears $k \leq \ell-1$ times,
we get that $M_{\ell,m}(\pi \mid \theta,\theta_*)$ belongs to the algebra
generated by $\cR^G_{\leq \ell-1}$.
Absorbing the contributions of these terms $M_{\ell,m}(\pi \mid
\theta,\theta_*)$ into $Q_\ell(\theta)$, it remains to consider those partitions
$\pi \in \cP(\ell,m)$ where some set in $\pi$ has cardinality $\ell$.

Without loss of generality, let us order the sets of $\pi$ so that its first set has
cardinality $\ell$. Then $g_1$ appears $\ell$ times in (\ref{eq:Ml}), so
exactly one of $\{\pi(2j-1),\pi(2j)\}$ must be 1 for each $j=1,\ldots,m$, and
every $\pi(j)$ must be 1 for $j=2m+1,\ldots,\ell+m$. For notational convenience,
consider $\pi$ such that $\pi(2j-1)=1$ for each $j=1,\ldots,m$ (e.g.~\eqref{eq:Mlm3} in Example~\ref{ex:Mlm}). For such $\pi$,
we have
\begin{equation}\label{eq:Mellspecial}
M_{\ell,m}(\pi \mid \theta,\theta_*)=\E_{g_1,\ldots,g_{|\pi|}}
\big[\langle g_1\theta,g_{\pi(2)}\theta \rangle \ldots
\langle g_1\theta,g_{\pi(2m)}\theta \rangle \langle g_1\theta,\theta_*
\rangle^{\ell-m}\big].
\end{equation}
Suppose now that there is a second set of $\pi$ which has cardinality at
most $\ell-1$, corresponding to the element $g_2$. Then $g_2$ appears between 1
and $\ell-1$ times in
$g_{\pi(2)},g_{\pi(4)},\ldots,g_{\pi(2m)}$. We may decouple the corresponding
$g_1$'s by introducing a new independent variable $\tilde{g}_1 \sim \Unif(G)$,
setting $\tilde{g}_2=\tilde{g}_1g_1^{-1}g_2$, and writing
\[\langle g_1\theta,g_2\theta \rangle=\langle \theta,g_1^{-1}g_2\theta \rangle
=\langle \tilde{g}_1\theta, \tilde{g}_2\theta \rangle.\]
The expectation over the uniform random pair $(g_1,g_2)$ may be
replaced by that over the uniform random triple $(g_1,\tilde{g}_1,\tilde{g}_2)$,
reducing (\ref{eq:Mellspecial}) into an expectation where each distinct
group element now appears $\leq \ell-1$ times. Then by the argument for the previous case,
we also have that $M_{\ell,m}(\pi \mid \theta,\theta_*)$ belongs to the algebra
generated by $\cR^G_{\leq \ell-1}$ in this case, and these terms may
be absorbed into $Q_\ell(\theta)$.

The only partitions that remain are those where every set in $\pi$ has
cardinality $\ell$. One such partition corresponds to $m=0$, where
$\pi=\{\{1,2,\ldots,\ell\}\}$. For this $\pi$, we have
\[M_{\ell,m}(\pi \mid \theta,\theta_*)=\E_g[\langle \theta_*,g\theta
\rangle^\ell]=\E_{g_1,g_2}[\langle g_1\theta_*,g_2\theta \rangle^\ell] = \langle T_\ell(\theta_*),T_\ell(\theta) \rangle.\]
The remaining $2^{\ell-1}$ such partitions correspond to $m=\ell$ and $|\pi|=2$, where we may
assume without loss of generality that $1 \in \pi(1)$ and $2 \in \pi(2)$,
and take one element of each remaining pair $\{\pi(2j-1),\pi(2j)\}$ for
$j=1,\ldots,\ell$ to belong to $\pi(1)$ and the other to belong to $\pi(2)$.
For these partitions $\pi$, we have
\[M_{\ell,m}(\pi \mid \theta,\theta_*)=\E_{g_1,g_2}[\langle g_1 \theta,g_2
\theta \rangle^\ell]
= \|T_\ell(\theta)\|_\HS^2 
.\]
Applying the above two displays to (\ref{eq:Sl}), we obtain
\[S_\ell(\theta)=-\frac{1}{\ell!} \langle T_\ell(\theta_*),T_\ell(\theta) \rangle +\frac{1}{2(\ell!)}\|T_\ell(\theta)\|_\HS^2 +Q_\ell(\theta)\]
for some $Q_\ell$ in the algebra generated by $\cR^G_{\leq \ell-1}$. 
Completing the square yields
$S_\ell(\theta)=\frac{1}{2(\ell!)} \|T_\ell(\theta)-T_\ell(\theta_*)\|_\HS^2 - \frac{1}{2(\ell!)} \|T_\ell(\theta_*)\|_\HS^2  +Q_\ell(\theta)$,
where $\|T_\ell(\theta_*)\|_\HS^2$ does not depend on $\theta$ and can be absorbed into $Q_\ell(\theta)$. We thus arrive at the stated
form of $S_\ell(\theta)$ in \eqref{eq:Selldegell}. 
Since the entries of $T_\ell(\theta)$
belong to $\cR^G_{\leq \ell}$, we obtain also that $S_\ell$ belongs to the
algebra generated by $\cR^G_{\leq \ell}$.
\end{proof}

The following computation of the first three terms of (\ref{eq:formalseries})
will be useful in our analysis of specific group actions. By Lemma
\ref{lemma:kerneldecomp}, we assume without loss of generality that $\E_g[g]=0$.

\begin{lemma}\label{lemma:Smeanzero}
If $\E_g[g]=0$, then
\begin{align*}
S_1(\theta)&=0\\
S_2(\theta)&=-\tfrac{1}{2}\E_g[\langle \theta_*,g\theta \rangle^2]
+\tfrac{1}{4}\E_g[\langle \theta,g\theta \rangle^2]\\
S_3(\theta)&=-\tfrac{1}{6}\E_g[\langle \theta_*,g\theta \rangle^3]
+\tfrac{1}{12}\E_g[\langle \theta,g\theta \rangle^3]\\
&\hspace{0.5in}+\tfrac{1}{2}\E_{g_1,g_2}[\langle g_1\theta,g_2\theta \rangle
\langle \theta_*,g_1\theta \rangle\langle \theta_*,g_2\theta \rangle]
-\tfrac{1}{3}\E_{g_1,g_2}[\langle g_1\theta,g_2\theta \rangle
\langle \theta,g_1\theta \rangle \langle \theta,g_2\theta \rangle].
\end{align*}
\end{lemma}
\begin{proof}
If $\E_g[g]=0$, then by (\ref{eq:Ml}), any $\pi \in \cP(\ell,m)$ which has a singleton
yields $M_{\ell,m}(\pi \mid \theta,\theta_*)=0$.

For $\ell=1$ and $m \in \{0,1\}$, every $\pi \in \cP(\ell,m)$ has a singleton,
so $S_1(\theta)=0$.

For $\ell=2$ and $m \in \{0,1,2\}$, the only partitions
$\pi \in \cP(\ell,m)$ which do not have a singleton are
$\{\{1,2\}\}$ for $m=0$ and $\{\{1,3\},\{2,4\}\}$ and $\{\{1,4\},\{2,3\}\}$
for $m=2$. We get
\begin{align*}
S_2(\theta)&=-\tfrac{1}{2}M_{2,0}(\{\{1,2\}\})
+\tfrac{1}{8}M_{2,2}(\{\{1,3\},\{2,4\}\})+\tfrac{1}{8}M_{2,2}(\{\{1,4\},
\{2,3\}\})\\
&=-\tfrac{1}{2}\E_g[\langle \theta_*,g\theta \rangle^2]
+\tfrac{1}{4}\E_{g_1,g_2}[\langle g_1\theta,g_2\theta \rangle^2]\\
&=-\tfrac{1}{2}\E_g[\langle \theta_*,g\theta \rangle^2]
+\tfrac{1}{4}\E_g[\langle \theta,g\theta \rangle^2],
\end{align*}
the last line applying the equality in law $g_1^\top g_2\overset{L}{=}g_1$.

For $\ell=3$, grouping together
$\pi \in \cP(\ell,m)$ that yield the same value of
$M_{\ell,m}(\pi \mid \theta,\theta_*)$ by symmetry, we may check that
\begin{align*}
S_3(\theta)&=-\tfrac{1}{6}M_{3,0}(\{\{1,2,3\}\})
+2 \cdot \tfrac{1}{4}M_{3,1}(\{1,3\},\{2,4\})
+4 \cdot \tfrac{1}{8}M_{3,2}(\{1,3,5\},\{2,4\})\\
&\hspace{0.5in}+4 \cdot \tfrac{1}{48}M_{3,3}(\{1,3,5\},\{2,4,6\})
-8 \cdot \tfrac{1}{24}M_{3,3}(\{1,3\},\{2,5\},\{4,6\})\\
&=-\tfrac{1}{6}\E_g[\langle \theta_*,g\theta \rangle^3]
+\tfrac{1}{2}\E_{g_1,g_2}[\langle g_1\theta,g_2\theta \rangle
\langle \theta_*,g_1\theta \rangle \langle \theta_*,g_2\theta \rangle]
+\tfrac{1}{2}\E_{g_1,g_2}[\langle g_1\theta,g_2\theta \rangle^2 \langle
\theta_*,g_1\theta \rangle]\\
&\hspace{0.5in}+\tfrac{1}{12}\E_{g_1,g_2}[\langle g_1\theta,g_2\theta \rangle^3]
-\tfrac{1}{3}\E_{g_1,g_2,g_3}[\langle g_1\theta,g_2\theta \rangle
\langle g_1\theta,g_3\theta \rangle \langle g_2\theta,g_3\theta \rangle].
\end{align*}
By the equality in joint law $(g_1^\top g_2,g_1)
\overset{L}{=}(g_2,g_1)$, the third term vanishes because
\[\E_{g_1,g_2}[\langle g_1\theta,g_2\theta \rangle^2 \langle
\theta_*,g_1\theta \rangle]
=\E_{g_1,g_2}[\langle \theta,g_2\theta \rangle^2 \langle \theta_*,g_1\theta
\rangle]=\E_{g_2}[\langle \theta,g_2\theta \rangle^2]
\E_{g_1}[\langle \theta_*,g_1\theta \rangle]=0.\]
Applying $g_1^\top g_2
\overset{L}{=}g$ and $(g_1^\top g_2,g_1^\top g_3,g_2^\top g_3)
\overset{L}{=}(g_1^\top g_2,g_1^\top,g_2^\top)$ to the remaining terms yields
the form of $S_3$.
\end{proof}

We now prove Lemma \ref{lemma:seriesexpansion}. We will first show
the expansion (\ref{eq:formalseries}) of $R(\theta)$ formally
in Lemma \ref{lem:formal-expansion} below, and then prove
quantitative estimates on the truncation error.  
Recalling the form of $R(\theta)$ in (\ref{eq:R}), we define the formal series
\[
R_\formal (\theta) = \frac{\|\theta\|^2}{2} \sigma^{-2}
- \sum_{k = 1}^\infty \frac{1}{k!} \E_\eps\left[\kappa_k\Big(\langle \sigma^{-2} \theta_* + \sigma^{-1} \eps, g\theta\rangle\Big) \right]
\]
using the cumulant generating function
\begin{equation}\label{eq:cumulantexpansionfg}
\log \E_g[e^{f(g)}] = \sum_{k=1}^\infty \frac{1}{k!}\,\kappa_k(f(g))
\end{equation}
for $f(g) = \langle \sigma^{-2} \theta_* + \sigma^{-1} \eps,g\theta \rangle$, where $\kappa_k(f(g))$
is the $k^\text{th}$ cumulant of $f(g)$ over the law $g \sim \Unif(G)$, conditional on $\eps$.
See Appendix \ref{appendix:cumulants} for definitions.

\begin{lemma} \label{lem:formal-expansion}
As formal power series in $\sigma^{-1}$, we have the equality
\[
R_\formal(\theta) = \sum_{\ell=1}^\infty \sigma^{-2\ell}S_\ell(\theta).
\]
\end{lemma}
\begin{proof}
For notational convenience, set $z = \sigma^{-1}$.
In the rest of the proof, we treat all series expansions formally and take
termwise expectations $\E_\eps$. We now rewrite $R_\formal(\theta)$ using the cumulant
tensors of $g$: Define the order-$k$ moment tensor $\cT_k(g)$ of $g$ by
\begin{equation}\label{eq:gtensor}
\cT_k(g)=\E_g[g^{\otimes k}]
\end{equation}
where $g^{\otimes k} \in (\R^{d \times d})^{\otimes k}$ is the $k$-fold tensor
product of the linear map
$g:\R^d \to \R^d$, acting on $(\R^d)^{\otimes k}$ via
$g^{\otimes k}(v_1 \otimes \ldots \otimes
v_k)=gv_1 \otimes \ldots \otimes gv_k$.
Define the order-$k$ cumulant tensor $\cK_k(g)$ by the moment-cumulant relation
\begin{equation}\label{eq:tensormomentcumulant}
\cK_k(g)=\sum_{\text{partitions } \pi \text{ of } [k]}\;\;
(|\pi|-1)!(-1)^{|\pi|-1} \bigotimes_{S \in \pi} \cT_S(g),
\end{equation}
which is analogous to the usual moment-cumulant relation for scalar random
variables in (\ref{eq:scalarmomentcumulant}).
Here $\cT_S(g)$ is the order-$|S|$ moment tensor of $g$ acting on
$(\R^d)^{\otimes S}$, corresponding to the $|S|$ coordinates belonging to $S$.
For vectors $v_i,w_i \in \R^d$, we have the relation
\[\left\langle \bigotimes_{i \in S} v_i,\;
\cT_S(g) \left(\bigotimes_{i \in S} w_i\right) \right\rangle
=\E_g\left[ \left\langle \bigotimes_{i \in S} v_i,\;
\bigotimes_{i \in S} (gw_i) \right\rangle \right]
=\E_g\left[\prod_{i\in S} \langle v_i,gw_i \rangle\right].\]
Applying this, (\ref{eq:tensormomentcumulant}), and
(\ref{eq:scalarmomentcumulant}), we obtain
\begin{equation}\label{eq:cumtensorrelation}
\left\langle \bigotimes_{i=1}^k v_i,\;
\cK_k(g) \left(\bigotimes_{i=1}^k w_i\right) \right\rangle
=\kappa_k\Big(\langle v_1,gw_1 \rangle,\ldots,\langle v_k,gw_k \rangle\Big).
\end{equation}
Recall that $\kappa_k(f(g))=\kappa_k(f(g),\ldots,f(g))$,
where the latter mixed cumulant function is multi-linear and permutation
invariant in its arguments.
Applying \eqref{eq:cumtensorrelation} followed by a binomial expansion, we get
\[
\kappa_k\Big(\langle z^2\theta_*+z\eps,\;g\theta \rangle\Big)
=\Big\langle (z^2\theta_*+z\eps)^{\otimes k},\;\cK_k(g) \theta^{\otimes k} \Big\rangle
=\sum_{j=0}^k z^{2k-j} \binom{k}{j} \Big\langle \eps^{\otimes j} \otimes
\theta_*^{\otimes (k-j)},\;\cK_k(g) \theta^{\otimes k} \Big\rangle.\]
So as formal series we find
\begin{equation}\label{eq:multilinearexpansion}
R_\formal(\theta)=\frac{\|\theta\|^2}{2}z^2-\sum_{k=1}^\infty \sum_{j=0}^k \frac{z^{2k-j}}{j!(k-j)!}
\Big\langle \E_\eps[\eps^{\otimes j}]
\otimes \theta_*^{\otimes (k-j)},\;\cK_k(g)\theta^{\otimes k}\Big\rangle.
\end{equation}
Note that $\E_\eps[\eps^{\otimes j}]=0$ if $j$ is odd.
Reparametrizing the terms for even $j$ by $j=2m$ and $\ell=k-m$, it may be
checked that $\{(k,j): k \geq 1, 0 \leq j \leq k\}$ is in bijection with
$\{(\ell,m):\ell \geq 1,0 \leq m \leq \ell\}$. Thus, we obtain
\begin{equation}\label{eq:Q}
R_\formal(\theta)=\sum_{\ell=1}^\infty z^{2\ell}R_\ell(\theta),
\end{equation}
where
\begin{equation} \label{eq:rl-def}
R_\ell(\theta)=\1\{\ell=1\}\frac{\|\theta\|^2}{2}-\sum_{m=0}^\ell
\frac{1}{(2m)!(\ell-m)!} \Big\langle \E\big[\eps^{\otimes 2m}\big] 
\otimes \theta_*^{\otimes (\ell-m)},\;\cK_{\ell+m}(g)\theta^{\otimes
(\ell+m)}\Big\rangle.
\end{equation}
It remains to check that $R_\ell(\theta) = S_\ell(\theta)$.

To show this, let us compute explicitly the expectation over
$\eps \sim \N(0,\Id)$ in (\ref{eq:rl-def}).
Consider the identity matrix as an element of $(\R^d)^{\otimes 2}$,
\[\Id=\sum_{i=1}^d e_i \otimes e_i,\]
where $e_i$ is the $i^\text{th}$ standard basis vector in $\R^d$.
For any pairing $\pi$ of $[2m]$, denote
$\bigotimes_{S \in \pi} \Id \in (\R^d)^{\otimes 2m}$
as the tensor product of $m$ copies of $\Id$ that
associates the two coordinates of each copy of $\Id$ with a pair $S \in \pi$.
Using that the $2k^\text{th}$ moment of a standard Gaussian variable is the number of
pairings of $[2k]$, we have for any basis vector $e_{i_1} \otimes
\ldots \otimes e_{i_{2m}} \in (\R^d)^{2m}$ that
\begin{align*}
\langle \E_\eps[\eps^{\otimes {2m}}], e_{i_1} \otimes
\ldots e_{i_{2m}} \rangle=\E_\eps\left[\prod_{j=1}^{2m} \eps_{i_j}\right]
&=\sum_{\text{pairings } \pi \text{ of } [2m]}\;\;
\prod_{(j_1,j_2) \in \pi} \1\{i_{j_1}=i_{j_2}\}\\
&=\left\langle \sum_{\text{pairings } \pi \text{ of } [2m]}\;\;
\left(\bigotimes_{S \in \pi} \Id\right),\;
e_{i_1} \otimes \ldots \otimes e_{i_{2m}}\right\rangle.
\end{align*}
Hence we see that
\[
\E_\eps[\eps^{\otimes 2m}]=\sum_{\text{pairings } \pi \text{ of }[2m]}\;\;
\bigotimes_{S \in \pi} \Id.
\]
Applying (\ref{eq:cumtensorrelation}) and
the permutation invariance of $\kappa_{\ell+m}$ in its arguments, we get
\begin{equation}
\Big\langle \E_\eps[\eps^{\otimes 2m}] \otimes \theta_*^{\otimes (\ell-m)},\;
\cK_{\ell+m}(g)\theta^{\otimes (\ell+m)}\Big\rangle
=(2m-1)!! \cdot \Big\langle \Id^{\otimes m} \otimes \theta_*^{\otimes (\ell-m)},
\;\cK_{\ell+m}(g)\theta^{\otimes (\ell+m)}\Big\rangle,
\label{eq:sumpairing}
\end{equation}
since there are $(2m-1)!!$ total pairings, and by permutation invariance,
the term corresponding to each pairing contributes equally to this
inner product. (The right side of \eqref{eq:sumpairing} corresponds to the
consecutive pairing of $[2m]$.) Applying \eqref{eq:sumpairing} and $(2m-1)!!/(2m)!=1/(2^mm!)$
to (\ref{eq:rl-def}) yields
\begin{equation}\label{eq:Sltmp2}
R_\ell(\theta)=\1\{\ell=1\}\frac{\|\theta\|^2}{2}- \frac{1}{\ell!}\sum_{m=0}^\ell
\frac{1}{2^m } \binom{\ell}{m} \Big\langle \Id^{\otimes m}
\otimes \theta_*^{\otimes (\ell-m)},\;
\cK_{\ell+m}(g)\theta^{\otimes (\ell+m)}\Big\rangle.
\end{equation}
Now we use (\ref{eq:tensormomentcumulant}) to write
\begin{align}
&\Big\langle \Id^{\otimes m} \otimes \theta_*^{\otimes (\ell-m)},
\;\cK_{\ell+m}(g)\theta^{\otimes (\ell+m)}\Big\rangle\nonumber\\
&=\sum_{\text{partitions } \pi \text{ of } [\ell+m]}\;\;
(|\pi|-1)!(-1)^{|\pi|-1} \Big\langle
\Id^{\otimes m} \otimes \theta_*^{\otimes (\ell-m)},\;
\Big(\bigotimes_{S \in \pi} \cT_S(g)
\Big)\theta^{\otimes (\ell+m)}\Big\rangle\nonumber\\
&=\sum_{\text{partitions } \pi \text{ of } [\ell+m]}\;\;
(|\pi|-1)!(-1)^{|\pi|-1} M_{\ell,m}(\pi)
\label{eq:Mpiexpression}
\end{align}
where we set
\[M_{\ell,m}(\pi)
\equiv \Big\langle \Id^{\otimes m} \otimes \theta_*^{\otimes (\ell-m)},
\; \bigotimes_{S \in \pi} \E_g\big[(g\theta)^{\otimes S}\big]\Big\rangle.\]
We may move the expectations over $g$ out of the
inner product by writing this as an expectation over $|\pi|$ independent copies
of $g$, one for each $S \in \pi$, so that 
\[
M_{\ell,m}(\pi) = \E_{g_1,\ldots,g_{|\pi|}}\left[
\left\langle \Id^{\otimes m} \otimes \theta_*^{\otimes (\ell-m)},
\; \bigotimes_{i=1}^{\ell+m} (g_{\pi(i)} \theta) \right\rangle
\right],
\]
where for each $i\in[\ell+m]$, $\pi(i)$ denotes the index of the part in $\pi$ containing $i$.
Then using $\langle \Id,v \otimes w \rangle=\langle v,w \rangle$
and $\langle a \otimes b, c \otimes d \rangle=\langle a,c \rangle \langle b,d
\rangle$, we see that this is exactly the quantity
$M_{\ell,m}(\pi \mid \theta,\theta_*)$ defined previously in (\ref{eq:Ml}).

Finally, we combine (\ref{eq:Mpiexpression}) with (\ref{eq:Sltmp2}) and
describe a cancellation of terms that reduces the expression to $S_\ell(\theta)$:
First, note that $\langle \Id,(g \theta)^{\otimes 2} \rangle
=\langle g\theta,g\theta \rangle=\|\theta\|^2$, which does not depend on $g$.
If $m \geq 1$ and $\{1,2\}$ belong to the same part in $\pi$, then 
\begin{equation}
M_{\ell,m}(\pi)=\|\theta\|^2 M_{\ell-1,m-1}(\pi^-)
\label{eq:piminus}
\end{equation}
where $\pi^-$ is the partition of $\{3,\ldots,\ell+m\}$ obtained by removing
1 and 2. Suppose first that $\ell \geq 2$ and
$m \geq 1$. Fix any partition $\pi^-$ of $\{3,\ldots,\ell+m\}$. Let 
$\mathcal{S}$ be the collection of partitions of $[\ell+m]$ that 
do not separate $\{1,2\}$ and that reduce to $\pi^-$
upon removing 1 and 2. 
There are two types of such partitions $\pi$:
(a) $\pi$ includes $1,2$ into a part of $\pi^-$. Then $|\pi|=|\pi^-|$ and there are $|\pi^-|$ such partitions; 
(b) $\pi$ is the unique partition that adds $\{1,2\}$ as a new part to $\pi^-$ so that $|\pi|=|\pi^-|+1$.
Summing over both types and using (\ref{eq:piminus}), we get
\begin{align*}
&\sum_{\pi \in \mathcal{S}} (|\pi|-1)!(-1)^{|\pi|-1} M_{\ell,m}(\pi)\\
&=\|\theta\|^2 M_{\ell-1,m-1}(\pi^-)
\left(|\pi^-| \cdot (|\pi^-|-1)!(-1)^{|\pi^-|-1}
+1 \cdot (|\pi^-|)!(-1)^{|\pi^-|}\right) = 0.
\end{align*}
Summing over all $\pi^-$, the total
contribution to (\ref{eq:Mpiexpression}) from partitions $\pi$ that put
$\{1,2\}$ in the same set is 0. Similarly, the total contribution to
(\ref{eq:Mpiexpression}) from partitions $\pi$ that put $\{3,4\}$ in the
same set, but that \emph{do not} put $\{1,2\}$ in the same set, is also 0, and
so forth. Recalling the set of partitions $\cP(\ell,m)$ defined in (\ref{eq:goodpartition})
which separate each pair $\{1,2\},\ldots,\{2m-1,2m\}$,
we get in this case of $\ell \geq 2$ and $m \geq 1$ that only these partitions
contribute to (\ref{eq:Mpiexpression}), i.e.
\[\Big\langle \Id^{\otimes m} \otimes \theta_*^{\otimes (\ell-m)},\;
\cK_{\ell+m}(g)\theta^{\otimes (\ell+m)} \Big\rangle
=\sum_{\pi \in \cP(\ell,m)} (|\pi|-1)!(-1)^{|\pi|-1}M_{\ell,m}(\pi).\]
Using that $\cP(\ell,0)$ is simply the set of all partitions of $[\ell]$, and
applying this to (\ref{eq:Sltmp2}), we get that (\ref{eq:Sltmp2}) is the same as
$S_\ell(\theta)$ for $\ell \geq 2$.
For $\ell=1$, we have either $m=0$ or $m=1$. When $m=1$, the only partition
of $[\ell+m]=[2]$ not belonging to $\cP(1,1)$ is
$\{\{1,2\}\}$. Note that $M_{1,1}(\{\{1,2\}\})=\E_g[\langle g\theta,g\theta
\rangle]=\|\theta\|^2$, which cancels  the leading term
$\|\theta\|^2/2$ for $\ell=1$ in (\ref{eq:Sltmp2}). Thus 
(\ref{eq:Sltmp2}) also coincides with $S_\ell(\theta)$ for $\ell=1$, concluding the proof.
\end{proof}

\begin{proof}[Proof of Lemma \ref{lemma:seriesexpansion}]
  We will apply a truncation argument to handle the expansion of Lemma \ref{lem:formal-expansion}
  analytically.  Within the rest of the proof, all summations will be standard (non-formal)
  summations. For notational convenience, set
\[z=\sigma^{-1}, \qquad s(z)=r(z^{-1})=r(\sigma),
\qquad q(z)=\log(z^{-1})=\log \sigma.\]
The given conditions are $s(z) \to \infty$ and $zs(z)q(z) \to 0$ as $z \to 0$.

Consider the event $\|\eps\| \leq q(z)$ and define the truncation
\[R_\trunc(\theta)=\frac{\|\theta\|^2}{2}z^2-\sum_{k=1}^\infty \frac{1}{k!}
\E_\eps\left[\kappa_k\Big(
\langle z^2\theta_*+z\eps,\;g\theta \rangle\Big)\1\{\|\eps\| \leq q(z)\}\right].\]
For $\|\theta\|<s(z)$ and on this event $\|\eps\| \leq q(z)$,
observe that $\max_{g \in G} |f(g)| \leq (z^2\|\theta_*\|+zq(z))s(z)$.
By the given condition $zs(z)q(z) \to 0$ as $z \to 0$ (which also implies
$z^2s(z) \to 0$), and by Lemma
\ref{lem:cumulantbounds}(c), we observe that this series defining $R_\trunc(\theta)$
is absolutely convergent whenever $z<z_0$, for a small enough constant $z_0>0$.
Then, writing (\ref{eq:R}) as
\[R(\theta)=\frac{\|\theta\|^2}{2}z^2
-\E_\eps\Big[\1\{\|\eps\| \leq q(z)\} \cdot \log \E_g[e^{f(g)}]\Big]
-\E_\eps\Big[\1\{\|\eps\|>q(z)\} \cdot \log \E_g[e^{f(g)}]\Big]\]
and applying (\ref{eq:cumulantexpansionfg}) and
Fubini's theorem to exchange $\E_\eps$ and $\sum_k$ in the second term, we arrive
at
\begin{align}
R(\theta)&=R_\trunc(\theta)-\E_\eps\left[\1\{\|\eps\|>q(z)\} \cdot
\log \E_g\left[e^{\langle z^2\theta_*+z\eps, g\theta \rangle}\right]\right].\label{eq:RQ}
\end{align}

Note that $\E_\eps[\eps^{\otimes j}\1\{\|\eps\| \leq q(z)\}]=0$ if $j$ is odd
by sign symmetry of the law of $\eps$ conditional on $\|\eps\| \leq q(z)$.
Therefore, by the same argument as for (\ref{eq:Q}), we obtain
\begin{equation}\label{eq:Q2}
R_\trunc(\theta)=\sum_{\ell=1}^\infty z^{2\ell}R_{\trunc, \ell}(\theta)
\end{equation}
where
\[R_{\trunc, \ell}(\theta)=\1\{\ell=1\}\frac{\|\theta\|^2}{2}-\sum_{m=0}^\ell
\frac{1}{(2m)!(\ell-m)!} \Big\langle \E\big[\eps^{\otimes 2m}
\1\{\|\eps\| \leq q(z)\}\big] 
\otimes \theta_*^{\otimes (\ell-m)},\;\cK_{\ell+m}(g)\theta^{\otimes
(\ell+m)}\Big\rangle.\]
Applying the cumulant bound of Lemma \ref{lem:cumulantbounds} together with
(\ref{eq:cumtensorrelation}) and $k! \geq k^k/e^k$,
for $\ell \geq 2$,
\begin{align}
|R_{\trunc, \ell}(\theta)| &\leq \sum_{m=0}^\ell
\frac{1}{(2m)!(\ell-m)!} \E_\eps\left[\,\left|\Big\langle \eps^{\otimes 2m}
\otimes \theta_*^{\otimes (\ell-m)},\;\cK_{\ell+m}(g)\theta^{\otimes
(\ell+m)}\Big\rangle\right|\,\1\{\|\eps\| \leq q(z)\}\right]\nonumber\\
&\leq \sum_{m=0}^\ell \frac{1}{(\ell+m)!}
\binom{\ell+m}{2m}(\ell+m)^{\ell+m}q(z)^{2m}\|\theta_*\|^{\ell-m}
\|\theta\|^{\ell+m}\nonumber\\
&\leq e^{2\ell}\|\theta\|^{2\ell}\sum_{m=0}^\ell
\binom{\ell+m}{2m}q(z)^{2m}\|\theta_*\|^{\ell-m}
\leq e^{2\ell}(q(z)+\|\theta_*\|)^{2\ell}\|\theta\|^{2\ell}.\label{eq:Qbound}
\end{align}
Then for $\|\theta\|<s(z)$ and $z<z_0$, 
the series in (\ref{eq:Q2}) is absolutely convergent.
Differentiating each $R_{\trunc, \ell}(\theta)$ in $\theta$ using the product rule, a
similar argument shows that for $\ell \geq 2$,
\begin{align}
\|\nabla R_{\trunc, \ell}(\theta)\| &\leq 2\ell e^{2\ell}
(q(z)+\|\theta_*\|)^{2\ell}\|\theta\|^{2\ell-1},\label{eq:gradQbound}\\
\|\nabla^2 R_{\trunc, \ell}(\theta)\| &\leq 2\ell(2\ell-1) e^{2\ell}
(q(z)+\|\theta_*\|)^{2\ell}\|\theta\|^{2\ell-2}.\label{eq:hessQbound}
\end{align}
Then both $\sum_\ell z^{2\ell} \nabla R_{\trunc, \ell}(\theta)$ and $\sum_\ell
z^{2\ell} \nabla^2 R_{\trunc, \ell}(\theta)$ are also absolutely and
uniformly convergent over $\|\theta\|<s(z)$, so
\[\nabla R_\trunc(\theta)=\sum_{\ell=1}^\infty z^{2\ell} \nabla R_{\trunc, \ell}(\theta),
\qquad \nabla^2 R_\trunc(\theta)=\sum_{\ell=1}^\infty z^{2\ell}
\nabla^2 R_{\trunc, \ell}(\theta).\]

We now fix an integer $k \geq 1$ and remove the truncation event
$\|\eps\| \leq q(z)$. Note first that by Cauchy-Schwarz and a
chi-squared tail bound, for all $z<z_0$ and some constants $C,c,z_0>0$,
the second term in \eqref{eq:RQ} is at most
\begin{align*}
\left|\E_\eps\left[\1\{\|\eps\|>q(z)\} \cdot
\log \E_g\Big[e^{\langle z^2\theta_*+z\eps,g\theta \rangle}\Big]\right]\right|
&\leq \E_\eps\Big[\1\{\|\eps\|>q(z)\} \cdot \|z^2 \theta_*+z\eps\|
\cdot \|\theta\|\Big]\\
&\leq \|\theta\| \cdot \P\big[\|\eps\|>q(z)\big]^{1/2}
\E_\eps\big[\|z^2\theta_*+z\eps\|^2\big]^{1/2}\\
    &\leq s(z) \cdot e^{-cq(z)^2} \cdot Cz.
\end{align*}
Recalling $zs(z) \to 0$ and $q(z)=\log (1/z)$,
there exists $z_0$ (depending
on $k$) such that $zs(z)e^{-cq(z)^2} \leq z^{2k+2}$ for all $z<z_0$. Applying
this to (\ref{eq:RQ}), and also using (\ref{eq:Qbound}) to bound the sum over
$\ell \geq k+1$ in (\ref{eq:Q2}), we obtain
\begin{equation}\label{eq:Rboundtmp}
\left|R(\theta)-\sum_{\ell=1}^k z^{2\ell}
R_{\trunc, \ell}(\theta)\right| \leq [Czq(z)(\|\theta\| \vee 1)]^{2k+2}
\end{equation}
for $z<z_0$ and $C,z_0$ depending on $k$.
For the gradient and Hessian, recall (\ref{eq:gcondlaw}) and note that
\begin{align*}
\left\|\nabla_\theta
\log \E_g\Big[e^{\langle z^2\theta_*+z\eps,g\theta \rangle}\Big]\right\|
&=\Big\|\E_g[g^\top(z^2\theta_*+z\eps) \mid \eps,\theta]\Big\| \leq \|z^2\theta_*+z\eps\|,\\
\left\|\nabla_\theta^2
\log \E_g\Big[e^{\langle z^2\theta_*+z\eps,g\theta \rangle}\Big]\right\|
&=\Big\|\Cov_g[g^\top(z^2\theta_*+z\eps) \mid \eps,\theta]\Big\| \leq \|z^2\theta_*+z\eps\|^2.
\end{align*}
Then applying a similar Cauchy-Schwarz argument together with
(\ref{eq:gradQbound}) and (\ref{eq:hessQbound}), we get
\begin{align}
\left\|\nabla R(\theta)-\sum_{\ell=1}^k z^{2\ell}
\nabla R_{\trunc, \ell}(\theta) \right\| &\leq [Czq(z)]^{2k+2}(\|\theta\| \vee
1)^{2k+1},\label{eq:gradRboundtmp}\\
\left\|\nabla^2 R(\theta)-\sum_{\ell=1}^k z^{2\ell}
\nabla^2 R_{\trunc, \ell}(\theta) \right\| &\leq [Czq(z)]^{2k+2}(\|\theta\| \vee
1)^{2k}.\label{eq:hessRboundtmp}
\end{align}

Next, for all $\ell \leq k$ and some $C,z_0>0$ 
depending on $k$, the same Cauchy-Schwarz argument yields for $z<z_0$ that
\begin{align*}
|R_\ell(\theta)-R_{\trunc, \ell}(\theta)| &\leq \sum_{m=0}^\ell
\frac{1}{(2m)!(\ell-m)!}\P\big[\|\eps\|>q(z)\big]^{1/2}
\E_\eps\Big[\Big\langle \eps^{\otimes 2m} \otimes \theta_*^{\otimes (\ell-m)}
,\; \cK_{\ell+m}(g)\theta^{\otimes (\ell+m)}\Big \rangle^2\Big]^{1/2}\\
&\leq C \cdot \P[\|\eps\|>q(z)]^{1/2} \cdot \|\theta\|^{2\ell}
\leq Ce^{-cq(z)^2}\|\theta\|^{2\ell} \leq [Cz(\|\theta\| \vee 1)]^{2k+2}.
\end{align*}
Applying this to each term $\ell=1,\ldots,k$ in (\ref{eq:Rboundtmp}), we get
\[\left|R(\theta)-\sum_{\ell=1}^k z^{2\ell} R_\ell(\theta)\right|
\leq [Czq(z)(\|\theta\| \vee 1)]^{2k+2}.\]
The differences $\|\nabla R_\ell(\theta)-\nabla R_{\trunc, \ell}(\theta)\|$
and $\|\nabla^2 R_\ell(\theta)-\nabla^2 R_{\trunc, \ell}(\theta)\|$ may be
bounded similarly, and combined with (\ref{eq:gradRboundtmp})
and (\ref{eq:hessRboundtmp}) to show
\[\left\|\nabla R(\theta)-\sum_{\ell=1}^k z^{2\ell} \nabla
 R_\ell(\theta) \right\| \leq [Czq(z)]^{2k+2}(\|\theta\| \vee 1)^{2k+1},\]
\[\left\|\nabla^2 R(\theta)-\sum_{\ell=1}^k z^{2\ell} \nabla^2
 R_\ell(\theta) \right\| \leq [Czq(z)]^{2k+2}(\|\theta\| \vee 1)^{2k}.\]
 Recalling that $z=1/\sigma$ and $q(z)=\log \sigma$ and noting that
 $R_\ell(\theta) = S_\ell(\theta)$ by Lemma \ref{lem:formal-expansion}
 concludes the proof.
\end{proof}

To provide sharper finite-sample concentration bounds, we now establish
an analogous expansion for the empirical risk $R_n(\theta)$. Note that whereas
in the population expansion (\ref{eq:formalseries})
the term for $\sigma^{-\ell}$ was a polynomial of $\theta$ belonging to
$\cR_{\leq \ell/2}^G$ (for even $\ell$), here the term for $\sigma^{-\ell}$ in
this expansion of the empirical risk is a polynomial of $\theta$ only
guaranteed to belong to $\cR_{\leq \ell}^G$.

\begin{lemma}\label{lemma:Rnexpansion}
Fix any function $r:(0,\infty) \to [1,\infty)$ such that $r(\sigma) \cdot (\log
\sigma)/\sigma \to 0$ as $\sigma \to \infty$.
There are polynomials $P_\ell(\eps,\theta,\theta_*)$ such that for any
$k \geq 1$, some $(\theta_*,d,G,k)$-dependent constants $C,c,c_0,\sigma_0>0$,
any $\sigma>\sigma_0$, and any $t>e^{-c_0(\log \sigma)^2}$,
with probability at least
\[1-Ce^{-c(\log n)^2}-(C\sigma (\log n)/t)^d
\Big(e^{-cnt\sigma/(\log n)^{k \vee 2}}+
e^{-cnt^2(\sigma/\log \sigma)^{2k+2}}\Big),\]
we have
\begin{align}
\bigg|R_n(\theta)-\sum_{\ell=1}^k \sigma^{-\ell} \cdot
\frac{1}{n}\sum_{i=1}^n P_\ell(\eps_i,\theta,\theta_*)\bigg|
&\leq \bigg[t+C\bigg(\frac{\log \sigma}{\sigma}\bigg)^{k+1}\bigg]
(\|\theta\| \vee 1)^{k+1},\label{eq:Rnexpansion0}\\
\bigg\|\nabla R_n(\theta)-\sum_{\ell=1}^k \sigma^{-\ell} \cdot
\frac{1}{n}\sum_{i=1}^n \nabla P_\ell(\eps_i,\theta,\theta_*)\bigg\|
&\leq \bigg[t+C\bigg(\frac{\log \sigma}{\sigma}\bigg)^{k+1}\bigg]
(\|\theta\| \vee 1)^{k},\label{eq:Rnexpansion1}\\
\bigg\|\nabla^2 R_n(\theta)-\sum_{\ell=1}^k \sigma^{-\ell} \cdot
\frac{1}{n}\sum_{i=1}^n \nabla^2 P_\ell(\eps_i,\theta,\theta_*) \bigg\|
&\leq \bigg[t+C\bigg(\frac{\log \sigma}{\sigma}\bigg)^{k+1}\bigg]
(\|\theta\| \vee 1)^{k-1} \label{eq:Rnexpansion2}
\end{align}
simultaneously for all $\theta \in \R^d$ satisfying $\|\theta\|<r(\sigma)$.

Here each term takes the form
\begin{equation}
P_\ell(\eps,\theta,\theta_*)=\sum_{m=1}^{M_\ell}
A_{\ell,m}(\eps,\theta_*)P_{\ell,m}(\theta)
\label{eq:Pell-form}
\end{equation}
for some $M_\ell \geq 1$, where
\begin{itemize}
\item Each $A_{\ell,m}$ is a polynomial in $\eps$ and $\theta_*$
of total degree at most $\ell$,
\item Each $P_{\ell,m} \in \cR_{\leq \ell}^G$ is a $G$-invariant polynomial in
$\theta$ of degree at most $\ell$,
\item $\E_\eps[P_\ell(\eps,\theta,\theta_*)]$ equals
$S_{\ell/2}(\theta)$ if $\ell$ is even and equals 0 if $\ell$ is odd,
where $S_{\ell/2}$ is as defined in (\ref{eq:Sl}) for the series expansion
of $R(\theta)$, and
\item $P_\ell$ and its derivatives satisfy, for some universal constant $C_0>0$,
\begin{align}
|P_\ell(\eps,\theta,\theta_*)| \leq  & ~
(\|\eps\|+\|\theta_*\|+C_0)^\ell (\|\theta\|\vee 1)^{\ell} 	\label{eq:Pell0} \\
\|\nabla_\theta P_\ell(\eps,\theta,\theta_*)\| \leq  & ~ C_0^{\ell}
(\|\eps\|+\|\theta_*\|+C_0)^\ell (\|\theta\|\vee 1)^{\ell-1} 	\label{eq:Pell1} \\
\|\nabla^2_\theta P_\ell(\eps,\theta,\theta_*)\| \leq & ~ C_0^{\ell}
(\|\eps\|+\|\theta_*\|+C_0)^\ell (\|\theta\|\vee 1)^{(\ell-2)
\vee 0}.\label{eq:Pell2}
\end{align}
\end{itemize}
\end{lemma}

\begin{proof}
As in the preceding proof, let $z=\sigma^{-1}$, $s(z)=r(z^{-1})=r(\sigma)$,
and $q(z)=\log(z^{-1})=\log \sigma$.
We write as shorthand $\E_n[f(\eps_i)]=n^{-1}\sum_{i=1}^n f(\eps_i)$. Then
analogous to (\ref{eq:RQ}), we have
\[R_n(\theta)=R_{\trunc, n}(\theta)-\E_n\Big[\1\{\|\eps_i\|>q(z)\}
\cdot \log \E_g\Big[e^{\langle z^2\theta_*+z\eps_i,g\theta \rangle}\Big]\Big]\]
where
\begin{align*}
R_{\trunc, n}(\theta)&=\frac{\|\theta\|^2}{2}z^2-\sum_{k=1}^\infty \frac{1}{k!}
\E_n\Big[\kappa_k\Big(\langle z^2\theta_*+z \eps_i,\,g\theta\rangle\Big)
\1\{\|\eps_i\| \leq q(z)\}\Big]\\
&=\frac{\|\theta\|^2}{2}z^2-\sum_{k=1}^\infty \sum_{j=0}^k
\frac{z^{2k-j}}{j!(k-j)!}\Big\langle \E_n[\eps_i^{\otimes j}\1\{\|\eps_i\|
\leq q(z)\}] \otimes \theta_*^{k-j},\,\cK_k(g)\theta^{\otimes k}\Big\rangle.
\end{align*}
Both sums are absolutely convergent, and the second line follows from the same
multi-linear expansion of the cumulant $\kappa_k$ as in
(\ref{eq:multilinearexpansion}).
Rearranging this sum according to powers of $z$ and applying the cumulant tensor identity (\ref{eq:cumtensorrelation}), we obtain
\[R_{\trunc, n}(\theta)=\sum_{\ell=1}^\infty z^\ell \cdot
\E_n\bigg[\1\{\|\eps_i\| \leq q(z)\} \cdot P_\ell(\eps_i,\theta,\theta_*)
\bigg]\]
where 
\begin{align}
P_\ell(\eps,\theta,\theta_*)
= & ~ \frac{\|\theta\|^2}{2}\1\{\ell=2\} - \sum_{\ell/2 \leq k \leq \ell} \frac{1}{k!} \binom{k}{\ell-k}
\kappa_k(\underbrace{\langle \eps,g\theta\rangle, \cdots, \langle \eps,g\theta\rangle}_{\text{$2k-\ell$ times}}, 
\underbrace{\langle \theta_*,g\theta\rangle, \cdots, \langle
\theta_*,g\theta\rangle}_{\text{$\ell-k$ times}}).
\label{eq:Pell-explicit}
\end{align}
These polynomials $P_\ell(\eps,\theta,\theta_*)$ satisfy
the conditions of the lemma. 
In particular, by the moment-cumulant relationship, in the form \eqref{eq:Pell-form} each $P_{\ell,m}$ can be taken to be an entry of the moment tensor $T_\ell(\theta) = \E_g[(g\theta)^{\otimes \ell}]$ (which is a degree-$\ell$ $G$-invariant polynomial) and hence $M_\ell \leq d^\ell$. 
The bounds \eqref{eq:Pell0}--\eqref{eq:Pell2} on $P_\ell$ and its derivatives follow from the same arguments as those that led to \eqref{eq:Qbound}--\eqref{eq:hessQbound}, in particular, the cumulant bound in Lemma \ref{lem:cumulantbounds}.

Thus, we arrive at
\begin{equation}\label{eq:Rndecomp}
R_n(\theta)=\sum_{\ell=1}^k z^\ell \cdot \frac{1}{n}\sum_{i=1}^n
P_\ell(\eps_i,\theta,\theta_*)+\mathrm{I}(\theta)+\mathrm{II}(\theta)
+\mathrm{III}(\theta)
\end{equation}
where
\begin{align*}
\mathrm{I}(\theta)&=\sum_{\ell=k+1}^\infty z^\ell \cdot \frac{1}{n}\sum_{i=1}^n
\1\{\|\eps_i\| \leq q(z)\} \cdot P_\ell(\eps_i,\theta,\theta_*)\\
\mathrm{II}(\theta)&=-\sum_{\ell=1}^k z^\ell \cdot \frac{1}{n}\sum_{i=1}^n 
\1\{\|\eps_i\|>q(z)\} \cdot P_\ell(\eps_i,\theta,\theta_*)\\
\mathrm{III}(\theta)&=-\frac{1}{n}\sum_{i=1}^n \1\{\|\eps_i\|>q(z)\} \cdot
\log \E_g\Big[e^{\langle z^2\theta_*+z\eps_i,g\theta \rangle}\Big].
\end{align*}

We conclude the proof by bounding these three remainder terms and their
derivatives. Throughout, $C,C',c,c'>0$ denote $(\theta_*,d,G,k)$-dependent
constants changing from instance to instance.
Beginning with $\mathrm{II}(\theta)$ and $\mathrm{III}(\theta)$,
we define the event $\cE$ where
$\|\eps_i\| \leq \log n$ for all $i=1,\ldots,n$. Then
\[\P[\cE^c]=\P\Big[\max_{i=1}^n \|\eps_i\|>\log n\Big] \leq ne^{-c(\log n)^2}
\leq Ce^{-c'(\log n)^2}.\]
Let $f(\eps)$ be either $z^\ell \cdot \1\{\|\eps\|>q(z)\} \cdot
P_\ell(\eps,\theta,\theta_*)$ for some $\ell \in \{1,\ldots,k\}$ in the case of
$\mathrm{II}(\theta)$, or $\1\{\|\eps\|>q(z)\} \cdot \log \E_g[e^{\langle
z^2\theta_*+z\eps_i,g\theta \rangle}]$ in the case of $\mathrm{III}(\theta)$.
Applying the bounds
\[\E_g[\langle z^2\theta_*+z\eps_i,g\theta \rangle]
\leq \log \E_g[e^{\langle z^2\theta_*+z\eps_i,g\theta \rangle}]
\leq \max_g \langle z^2\theta_*+z\eps_i,g\theta \rangle,
\]
in both cases and on the event $\cE$ for small $z=\sigma^{-1}$, we have
$|f(\eps_i)| \leq Cz(\|\theta\| \vee 1)^k(\log n)^k$.
Introducing the bounded summand
\[\check{f}(\eps)=\min\Big(\max\Big(f(\eps),-Cz(\|\theta\| \vee 1)^k
(\log n)^k\Big),Cz(\|\theta\| \vee 1)^k(\log n)^k\Big),\]
Bernstein's inequality yields
\[\P\bigg[\bigg|\frac{1}{n}\sum_{i=1}^n \check{f}(\eps_i)
-\E_\eps[\check{f}(\eps)]\bigg|>\tau \bigg] \leq
2\exp\left(-\frac{\frac{1}{2}n\tau^2}{\Var_\eps[\check{f}(\eps)]
+\frac{1}{3}\tau \cdot Cz(\|\theta\| \vee 1)^k(\log n)^k}\right).\]
We apply this with $\tau=\eta t \cdot (\|\theta\| \vee 1)^k$
and some small constant $\eta>0$. By the definition of $f(\eps)$ and
Cauchy-Schwarz,
\begin{equation}\label{eq:fepsmean}
\E_\eps[|\check{f}(\eps)|] \leq \E_\eps[|f(\eps)|]
\leq \P_\eps\big[\|\eps\|>q(z)\big]^{1/2} \cdot
\Big(Cz^2(\|\theta\| \vee 1)^{2k}\Big)^{1/2} \leq 
Ce^{-cq(z)^2} \cdot z(\|\theta\| \vee 1)^k \leq \tau,
\end{equation}
the last inequality holding when the constant $c_0$ for which
$t>e^{-c_0(\log \sigma)^2}$ is sufficiently small. Similarly,
\begin{equation}\label{eq:fepsvar}
\Var_\eps[\check{f}(\eps)]
\leq \E_\eps[f(\eps)^2]
\leq Ce^{-cq(z)^2} \cdot z^2(\|\theta\| \vee 1)^{2k}
\leq \tau \cdot z(\|\theta\| \vee 1)^k.
\end{equation}
Applying this to Bernstein's inequality above,
\[\P\Bigg[\bigg|\frac{1}{n}\sum_{i=1}^n \check{f}(\eps_i)\bigg|
>2\eta t \cdot (\|\theta\| \vee 1)^k\Bigg] \leq
2\exp\Big(-\frac{c\eta tn}{z(\log n)^k}\Big).\]
Note that $\check{f}(\eps_i)=f(\eps_i)$ for all $i$ on the event $\cE$. 
Then for some constants $C,C'>0$,
\[\P\Big[|\mathrm{II}(\theta)|>C\eta t \cdot (\|\theta\| \vee 1)^k
\text{ and } \cE\Big],\;
\P\Big[|\mathrm{III}(\theta)|>C\eta t \cdot (\|\theta\| \vee 1)^k
\text{ and } \cE\Big] \leq C'\exp\Big(-\frac{c\eta tn}
{z(\log n)^k}\Big).\]
To obtain a uniform guarantee over the ball $\|\theta\|<s(z)$, observe that
on the event $\cE$, both $\mathrm{II}(\theta)$ and $\mathrm{III}(\theta)$ are
$C(\log n)^k \cdot s(z)^k$-Lipschitz in $\theta$ over this ball.
Let us take a $\delta$-net of this ball with
$\delta=c\eta t/[s(z)^k(\log n)^k]$ and a sufficiently small constant $c>0$,
where the net has cardinality $(Cs(z)/\delta)^d$.
Applying the Lipschitz continuity and a union bound over $\theta$ in this net,
we then obtain
\begin{align*}
&\P\bigg[\sup_{\theta:\|\theta\|<s(z)}
|\mathrm{II}(\theta)|>C\eta t \cdot (\|\theta\| \vee 1)^k\bigg],\;
\P\bigg[\sup_{\theta:\|\theta\|<s(z)}
|\mathrm{III}(\theta)|>C\eta t \cdot (\|\theta\| \vee 1)^k\bigg]\\
&\leq (Cs(z)/\delta)^d e^{-\frac{c\eta tn}{z(\log n)^k}}+e^{-c(\log n)^2}
\leq (C'\sigma(\log n)/\eta t)^d e^{-\frac{c\eta t \cdot n\sigma}{(\log n)^k}}
+Ce^{-c(\log n)^2}.
\end{align*}

We may bound $\nabla \mathrm{II}(\theta)$, $\nabla^2 \mathrm{II}(\theta)$,
$\nabla \mathrm{III}(\theta)$, and $\nabla^2 \mathrm{III}(\theta)$ similarly:
Defining $f(\eps_i)$ as the summand corresponding to any entry of one of these
quantities, and recalling the forms of the derivatives of $\log \E_g[e^{\langle
z^2\theta_*+z\eps_i,g\theta \rangle}]$ from Lemma \ref{lem:empiricalriskform},
on $\cE$ we have $|f(\eps_i)| \leq Cz(\|\theta\| \vee 1)^{k-1}(\log n)^k$ in the case of $\nabla \mathrm{II}(\theta)$ or $\nabla \mathrm{III}(\theta)$, and 
$|f(\eps_i)| \leq Cz(\|\theta\| \vee 1)^{{(k-2) \vee 0}} (\log n)^{k \vee 2}$
in the case of $\nabla^2 \mathrm{II}(\theta)$ or
$\nabla^2 \mathrm{III}(\theta)$.
The inequalities (\ref{eq:fepsmean}) and (\ref{eq:fepsvar}) continue to hold,
and $\nabla \mathrm{II}(\theta)$, $\nabla^2 \mathrm{II}(\theta)$,
$\nabla \mathrm{III}(\theta)$, and $\nabla^2 \mathrm{III}(\theta)$ all remain
$C(\log n)^{k \vee 3} \cdot s(z)^k$-Lipschitz over the ball $\|\theta\|<s(z)$.
Then applying the same arguments as above, we obtain
for $i=0,1,2$,
\begin{equation}\label{eq:IIIIIbound}
\|\nabla^{(i)} \mathrm{II}(\theta)\|, \|\nabla^{(i)} \mathrm{III}(\theta)\|
\leq C\eta t \cdot (\|\theta\| \vee 1)^{{(k-i) \vee 0}}
\text{ for all } \|\theta\|<s(z)
\end{equation}
with probability at least
$1-(C'\sigma(\log n)/\eta t)^d e^{-\frac{c\eta t \cdot n\sigma}{(\log n)^{k \vee
2}}}-Ce^{-c(\log n)^2}$.

Turning to $\mathrm{I}(\theta)$, write the summand $f(\eps_i)=
z^\ell \cdot \1\{\|\eps_i\| \leq q(z)\} \cdot P_\ell(\eps_i,\theta,\theta_*)$.
Using \eqref{eq:Pell0}, we have $|f(\eps_i)| \leq
(C_0z q(z))^\ell (\|\theta\| \vee 1)^\ell$ where $C_0$ is a universal
constant independent of $\ell$. Then Hoeffding's inequality yields
\[\P\Bigg[\bigg|\frac{1}{n}\sum_{i=1}^n f(\eps_i)-\E_\eps[f(\eps)] \bigg|>
\tau\Bigg]
\leq 2e^{-\frac{2n\tau^2}{C_0^{2\ell}z^{2\ell}q(z)^{2\ell}(\|\theta\|
\vee 1)^{2\ell}}}.\]
We apply this with $\tau=\eta t \cdot (\|\theta\| \vee 1)^{k+1}/\ell^2$, and
we apply also
\[\sum_{\ell=k+1}^\infty z^\ell \cdot \E_\eps\Big[\1\{\|\eps\| \leq q(z)\} \cdot
P_\ell(\eps,\theta,\theta_*)\Big] \leq \sum_{\ell=k+1}^\infty
(C_0 z)^\ell q(z)^\ell(\|\theta\| \vee 1)^\ell
\leq C'\Big[zq(z)(\|\theta\| \vee 1)\Big]^{k+1}\]
for any small enough $z$, by the given condition $zq(z)s(z)=\frac{r(\sigma)\log\sigma}{\sigma} \to 0$.
Then taking a union bound over all $\ell \geq k+1$ and recalling the definition
of $\mathrm{I}(\theta)$,
\[\P\Big[|\mathrm{I}(\theta)|>
C\big[zq(z)(\|\theta\| \vee 1)\big]^{k+1}
+C\eta t \cdot (\|\theta\| \vee 1)^{k+1}\Big] \leq \sum_{\ell=k+1}^\infty
{2\exp\left(-\frac{2\eta^2t^2 (\|\theta\| \vee 1)^{2k+2}n}{C_0^{2\ell}
z^{2\ell}q(z)^{2\ell}(\|\theta\| \vee 1)^{2\ell}\ell^4}\right)}.\]
Applying again $zs(z)q(z) \to 0$ as $z \to 0$, for any constant $B>0$ and
sufficiently small $z$ we have
$1/(C_0^{2\ell}z^{2\ell}q(z)^{2\ell}(\|\theta\| \vee
1)^{2\ell}\ell^4) \geq B^\ell$. Then the summands of this probability bound
decay at least geometrically fast, so that the sum is at most
\[C'\exp\left(-\frac{c'\eta^2 t^2(\|\theta\| \vee 1)^{2k+2}n}
{z^{2k+2}q(z)^{2k+2}(\|\theta\| \vee 1)^{2k+2}} \right)
\leq C'e^{-c'\eta^2 t^2n(\sigma/\log\sigma)^{2k+2}}.\]
The same argument applies to bound $\nabla \mathrm{I}(\theta)$ and $\nabla^2
\mathrm{II}(\theta)$ entrywise, except that we use \eqref{eq:Pell1} and 
\eqref{eq:Pell2} in lieu of \eqref{eq:Pell0}  to get $\1\{\|\eps_i\| \leq q(z)\} \cdot \|\nabla P_\ell(\eps_i,\theta,\theta_*)\| \leq
(C_0q(z))^\ell (\|\theta\| \vee 1)^{\ell-1}$ and 
$\1\{\|\eps_i\| \leq q(z)\} \cdot \|\nabla^2 P_\ell(\eps_i,\theta,\theta_*)\| \leq
(C_0q(z))^\ell (\|\theta\| \vee 1)^{\ell-2}$ respectively, {for all
$\ell \geq 2$}. From the previous Lipschitz bounds for $P_\ell(\eps,\theta,\theta_*)$,
$\mathrm{II}(\theta)$, $\mathrm{III}(\theta)$ and their derivatives, and from
those for $R_n(\theta)$ and its derivatives from Lemma
\ref{lem:local_lipschitz}, we see that $\mathrm{I}(\theta)$, $\nabla
\mathrm{I}(\theta)$, and $\nabla^2 \mathrm{I}(\theta)$ are also
$C(\log n)^{k \vee 3} \cdot s(z)^k$-Lipschitz in $\theta$ over the ball
$\|\theta\|<s(z)$. Then applying a union bound over a $\delta$-net as before, we
obtain
for $i=0,1,2$,
\begin{equation}\label{eq:Ibound}
\|\nabla^{(i)} \mathrm{I}(\theta)\| \leq 
C\Big(\frac{\log \sigma}{\sigma}\Big)^{k+1}(\|\theta\| \vee 1)^{k+1-i}
+C\eta t \cdot (\|\theta\| \vee 1)^{k+1} \text{ for all } \|\theta\|<s(z)
\end{equation}
with probability at least
$1-(C\sigma(\log n)/\eta t)^de^{-c\eta^2t^2 n(\sigma/\log \sigma)^{2k+2}}$.
Applying (\ref{eq:IIIIIbound}) and (\ref{eq:Ibound}) to (\ref{eq:Rndecomp}) and
now taking $\eta$ to be a sufficiently small constant, we obtain the lemma.
\end{proof}

\subsection{Descent directions and pseudo-local-minimizers}\label{sec:localreparam}

We now relate the series expansion result of Lemma \ref{lemma:seriesexpansion}
to the landscape
of $R(\theta)$ around a fixed point $\wtheta \in \R^d$, for large $\sigma$. The
constants in this section may depend on this point $\wtheta$.

The following lemma establishes a condition for $\wtheta \in \R^d$ under which
we will be able to show that $R(\theta)$ has either a first-order or second-order 
descent direction in a neighborhood $\wtheta$.

\begin{lemma}\label{lemma:largesigmadescent}
Fix $\wtheta \in \R^d$, let $\varphi$ be a local reparametrization in an
open neighborhood $U$ of $\wtheta$, and let $\wvarphi=\varphi(\wtheta)$.
Suppose there exists $\ell \geq 1$ and a partition of $\varphi$
into subvectors $\varphi=(\varphi^1,\ldots,\varphi^\ell)$ such that
$S_1(\varphi),\ldots,S_{\ell-1}(\varphi)$ are functions depending only on
$\varphi^1,\ldots,\varphi^{\ell-1}$ and not on $\varphi^\ell$, and
\[\text{ either } \quad \nabla_{\varphi^\ell} S_\ell(\wvarphi) \neq 0 \quad
\text{ or } \quad \lambda_{\min}(\nabla_{\varphi^\ell}^2 S_\ell(\wvarphi))<0.\]
Then there exist constants $c,\sigma_0>0$ and an open neighborhood $U_0$ of
$\wtheta$ (all depending on $\wtheta,\theta_*,d,G$ but not on $\sigma$) such
that for all $\sigma>\sigma_0$ and for every $\theta \in U_0$,
\[\text{ either } \quad \|\nabla_\theta R(\theta)\| \geq c\sigma^{-2\ell} \quad
\text{ or } \quad \lambda_{\min}(\nabla_\theta^2 R(\theta)) \leq
-c\sigma^{-2\ell}.\]
\end{lemma}

\begin{proof}
First suppose that $\nabla_{\varphi^\ell} S_\ell(\wvarphi) \neq 0$.
Denote $c=\|\nabla_{\varphi^\ell} S_\ell(\wvarphi)\|$, and
note that this constant $c$ depends only on $\wtheta,\theta_*,d,G$ and
not on $\sigma$. By continuity of $\nabla_{\varphi^\ell} S_\ell$, this
implies $\|\nabla_{\varphi^\ell} S_\ell(\varphi)\|>c/2$ for all $\varphi$ in a
neighborhood $V_0$ of $\wvarphi$. Since $S_1,\ldots,S_{\ell-1}$
do not depend on $\varphi^\ell$, we have
$\nabla_{\varphi^\ell} S_1=\ldots=\nabla_{\varphi^\ell} S_{\ell-1}=0$.
Then, recalling (\ref{eq:Rk}), we get
$\|\nabla_{\varphi} R^\ell(\varphi)\| \geq
\|\nabla_{\varphi^\ell} R^\ell(\varphi)\|>(c/2)\sigma^{-2\ell}$ for all
$\varphi \in V_0$. Applying (\ref{eq:gradreparam}) and continuity and invertibility
of $\der_\theta \varphi$ near $\wtheta$,
this implies that $\|\nabla_\theta R^\ell(\theta)\|>c'\sigma^{-2\ell}$
for a constant $c'>0$ and all $\theta$ in a small enough neighborhood $U_0$ of
$\wtheta$. Then applying Lemma \ref{lemma:seriesexpansion}, for all
$\sigma>\sigma_0$, large enough $\sigma>0$, and all $\theta \in U_0$,
\[\|\nabla_\theta R(\theta)\| \geq (c'/2)\sigma^{-2\ell}.\]

Now suppose that $\lambda_{\min}(\nabla_{\varphi^\ell}^2 S_\ell(\wvarphi))<0$.
The argument is similar: Denote $-c=\lambda_{\min}(\nabla_{\varphi^\ell}^2
S_\ell(\wvarphi))$. Then $\lambda_{\min}(\nabla_{\varphi^\ell}^2
S_\ell(\varphi))<-c/2$ for all $\varphi$ in a neighborhood
$V_0$ of $\wvarphi$ by continuity, so $\lambda_{\min}(\nabla_\varphi^2
R^\ell(\varphi)) \leq \lambda_{\min}(\nabla_{\varphi^\ell}^2
R^\ell(\varphi))<-(c/2)\sigma^{-2\ell}$. Applying (\ref{eq:hessreparam}),
\[\nabla_\theta^2 R^\ell(\theta)=
(\der_\theta \varphi)^\top \cdot \nabla_\varphi^2 R^\ell(\varphi) \cdot \der_\theta \varphi
+\sum_{i=1}^d \partial_{\varphi_i}R^\ell(\varphi) \cdot \nabla_\theta^2 \varphi_i.\]
Then by continuity and invertibility of $\der_\theta \varphi$ near $\wtheta$,
for the first term we have
\[\lambda_{\min}((\der_\theta \varphi)^\top \cdot \nabla_\varphi^2
R^\ell(\varphi) \cdot \der_\theta \varphi)<-c'\sigma^{-2\ell}\]
for a constant $c'>0$
and all $\theta$ in a neighborhood $U_0$ of $\wtheta$. Then either
$\lambda_{\min}(\nabla_\theta^2 R^\ell(\theta))<-(c'/2)\sigma^{-2\ell}$, or we must
have for the second term and some $i \in \{1,\ldots,d\}$ that
$\|\nabla_\varphi R^\ell(\varphi)\| \geq |\partial_{\varphi_i}
R^\ell(\varphi)|>c'' \sigma^{-2\ell}$.
Here, we may take $c''=c'/(2d\max \|\nabla_\theta^2 \varphi_j(\theta)\|)$,
where this maximum is taken over all $j \in \{1,\ldots,d\}$ and $\theta \in U_0$.
Applying again Lemma \ref{lemma:seriesexpansion}, for all $\sigma>\sigma_0$ and
large enough $\sigma_0>0$, this implies that for
every $\theta \in U_0$, either
$\lambda_{\min}(\nabla_\theta^2 R^\ell(\theta)) \leq -(c'/4)\sigma^{-2\ell}$ or
$\|\nabla_\theta R^\ell(\theta)\| \geq (c''/2)\sigma^{-2\ell}$.
\end{proof}

Conversely, the following is a condition for $\wtheta \in \R^d$ under which we
will show that $R(\theta)$ has a local minimizer in any fixed neighborhood of
$\wtheta$, for all sufficiently large $\sigma$. We call these points
pseudo-local-minimizers, and these will be in correspondence with the
true local minimizers of $R(\theta)$ for large $\sigma$.
Note that pseudo-local-minimizers are determined by $\theta_*,d,G$
and do not depend on $\sigma$, but true local minimizers of $R(\theta)$ not
belonging to $\O_{\theta_*}$ may in general depend on $\sigma$. We discuss
an example of this phenomenon in Remark \ref{rem:min-example}.

\begin{definition}\label{eq:pseudolocalminimizer}
A point $\wtheta \in \R^d$ is a {\bf pseudo-local-minimizer} in
a local reparametrization $\varphi=(\varphi^1,\ldots,\varphi^L)$ around
$\wtheta$ if each function $S_\ell(\varphi)$ for $\ell=1,\ldots,L$ depends
only on $\varphi^1,\ldots,\varphi^\ell$ and not on
$\varphi^{\ell+1},\ldots,\varphi^L$, and
for each $\ell \in \{1,\ldots,L\}$ where $\varphi^\ell$ has non-zero dimension,
\[\nabla_{\varphi^\ell} S_\ell(\wvarphi)=0 \qquad \text{ and }
\qquad \lambda_{\min}(\nabla_{\varphi^\ell}^2 S_\ell(\wvarphi))>0.\]
\end{definition}

For each pseudo-local-minimizer $\wtheta$, we will also show that the risk
$R(\varphi)$ is strongly convex in a $\sigma$-independent neighborhood of
$\wvarphi=\varphi(\wtheta)$, and its Hessian
$\nabla_\varphi^2 R(\varphi)$ has the following graded block structure.

\begin{definition}\label{def:gradedblocks}
Consider a partition of coordinates $(\varphi^1,\ldots,\varphi^L)$
for $\R^d$. Let $H \equiv H(\sigma) \in \R^{d \times d}$ be a symmetric matrix,
and write its $L \times L$ block decomposition with respect to this partition as
\[H=\begin{pmatrix} H_{11} & \cdots & H_{1L} \\
\vdots & \ddots & \vdots \\
H_{L1} & \cdots & H_{LL} \end{pmatrix}.\]
The matrix $H(\sigma)$ has a {\bf graded block structure} with respect to
this partition if there are constants
$C,c,\sigma_0>0$ such that for all $\sigma>\sigma_0$ and all $k,\ell \in
\{1,\ldots,L\}$
where $\varphi^k$ and $\varphi^\ell$ have non-zero dimension,
\[C\sigma^{-2\ell} \geq \lambda_{\max}(H_{\ell\ell}) \geq
\lambda_{\min}(H_{\ell\ell}) \geq c\sigma^{-2\ell}
\qquad \text{ and } \qquad
\|H_{k\ell}\| \leq C\sigma^{-2\max(k,\ell)}.\]
\end{definition}

Thus the upper-left block of $H(\sigma)$ has magnitude $\sigma^{-2}$, the three
blocks adjacent to this have magnitude $\sigma^{-4}$, and so forth. We allow
$\varphi^\ell$ to have dimension 0, in which case the blocks $H_{k\ell}$ and
$H_{\ell k}$ for $k=1,\ldots,L$ are empty.

\begin{lemma}\label{lemma:largesigmapseudomin}
Let $\wtheta \in \R^d$ be a pseudo-local-minimizer in the
reparametrization $\varphi=(\varphi^1,\ldots,\varphi^L)$.
Denote $\wvarphi=\varphi(\wtheta)$. Then for any sufficiently small
open neighborhood $V_0$ of $\varphi$, there exist constants $c,\sigma_0>0$
depending on $\wtheta,V_0$ and $\theta_*,d,G$ but not on $\sigma$,
such that for all $\sigma>\sigma_0$ and $\varphi \in V_0$,
\begin{enumerate}[(a)]
\item $\nabla_\varphi^2 R(\varphi)$ has a graded block structure with respect to
the partition $\varphi=(\varphi^1,\ldots,\varphi^L)$,
\item $\lambda_{\min}(\nabla_\varphi^2 R(\varphi)) \geq c\sigma^{-2L}$, and
\item There is a unique critical point of $R(\varphi)$ in $V_0$,
which is a local minimizer of $R(\varphi)$.
\end{enumerate}
\end{lemma}

\begin{proof}[Proof of Lemma \ref{lemma:largesigmapseudomin}]
For part (a), observe that the Hessian $\nabla_\varphi^2 S_\ell(\varphi)$ is
non-zero only in the upper-left $\ell \times \ell$ blocks of the decomposition
corresponding to $(\varphi^1,\ldots,\varphi^L)$. Since
$\nabla_{\varphi^\ell}^2 S_\ell(\wvarphi)$ is positive-definite by assumption,
by continuity there is a neighborhood $V_0$ of $\wvarphi$ and constants $C,c>0$
for which
\begin{equation}\label{eq:Sellbounds}
\lambda_{\min}(\nabla_{\varphi^\ell}^2 S_\ell(\varphi)) \geq c
\qquad \text{ and } \qquad
\|\nabla_{\varphi}^2 S_\ell(\varphi)\| \leq C
\end{equation}
for all $\varphi \in V_0$. Applying this for each $\ell=1,\ldots,L$ and
recalling (\ref{eq:Rk}), we see that $\nabla_\varphi^2 R_L(\varphi)$ has a graded block
structure. Then $\nabla_\varphi^2 R(\varphi)$ also has a graded block structure, by Lemma
\ref{lemma:seriesexpansion}. This shows (a). Part (b) will follow from (a)
and Lemma \ref{lemma:gradedblocks} which we prove in the next section.

To show part (c), let us assume for expositional simplicity that each $\varphi^\ell$
has positive dimension---the same argument applies with minor modification to the
setting where some of the vectors $\varphi^\ell$ have dimension 0. Let
$\check{\varphi}=(\check{\varphi}^1,\ldots,\check{\varphi}^L)$
be a point which minimizes $R(\varphi)$ over the compact set $\overline{V}_0$.
Observe that the given condition implies $\wvarphi$ is a local and global 
minimizer of $S_1$ over $V_0$, and that
\[S_1(\check{\varphi})-S_1(\wvarphi) \geq c\|\check{\varphi}^1-\wvarphi^1\|^2.\]
Then by Lemma \ref{lemma:seriesexpansion}, for all $\sigma>\sigma_0$ and large
enough $\sigma_0>0$,
\[R(\check{\varphi})-R(\wvarphi) \geq
c\sigma^{-2}\|\check{\varphi}^1-\wvarphi^1\|^2-C\left(\frac{\log
\sigma}{\sigma}\right)^4.\]
The left side is non-positive because $\check{\varphi}$ minimizes $R(\varphi)$,
so we get $\|\check{\varphi}^1-\wvarphi^1\| \leq \sigma^{-\tau}$ for, say, $\tau=0.9$. Now consider the functions $f(\varphi^2)=S_2(\wvarphi^1,\varphi^2)$
and $\check{f}(\varphi^2)=S_2(\check{\varphi}^1,\varphi^2)$. The given condition
implies that $f$ is strongly convex and has a local and global minimizer in $V_0$ given by
$\wvarphi^2$. Applying the bound $\|\check{\varphi}^1-\wvarphi^1\| \leq
\sigma^{-\tau}$, we get that $\|f-\check{f}\| \leq C\sigma^{-\tau}$ and
$\|\nabla^2 f-\nabla^2 \check{f}\| \leq C\sigma^{-\tau}$, for some constant
$C>0$ and any sufficiently small neighborhood $V_0$ of $\wvarphi$.
Then applying Lemma
\ref{lem:localmincompare}, $\check{f}$ is also strongly convex on $V_0$, with a
local and global minimizer in $V_0$ given by some point $\bar{\varphi}^2$ for which
$\|\bar{\varphi}^2-\wvarphi^2\| \leq C'\sigma^{-\tau}$. This implies
\[S_2(\check{\varphi}^1,\check{\varphi}^2)-S_2(\check{\varphi}^1,\bar{\varphi}^2)
\geq c\|\check{\varphi}^2-\bar{\varphi}^2\|^2.\]
Since $S_1$ depends only on $\varphi^1$ and not on $\varphi^2$, we have by Lemma
\ref{lemma:seriesexpansion} that
\[R(\check{\varphi})-R((\check{\varphi}^1,\bar{\varphi}^2,\wvarphi^3,\ldots,\wvarphi^L))
\geq c\sigma^{-4}\|\check{\varphi}^2-\bar{\varphi}^2\|^2
-C\left(\frac{\log \sigma}{\sigma}\right)^6.\]
Then, since this is again non-positive, we obtain
$\|\check{\varphi}^2-\bar{\varphi}^2\| \leq \sigma^{-\tau}$, and hence also
$\|\check{\varphi}^2-\wvarphi^2\| \leq C\sigma^{-\tau}$. Now applying this
argument to $f(\varphi^3)=S_3(\wvarphi^1,\wvarphi^2,\varphi^3)$ and
$\check{f}(\check{\varphi}^1,\check{\varphi}^2,\varphi^3)$, we obtain similarly
$\|\check{\varphi}^3-\wvarphi^3\| \leq C\sigma^{-\tau}$. Iterating this argument
yields $\|\check{\varphi}-\wvarphi\| \leq C\sigma^{-\tau}$ for a constant $C>0$.
For any neighborhood $V_0$, large enough $\sigma_0>0$ (depending on $V_0$),
and all $\sigma>\sigma_0$, this implies 
that this minimizer $\check{\varphi}$ belongs to the interior of $V_0$,
and hence must be a critical point of $R(\varphi)$. Then the strong convexity
in part (b) implies that this is the unique critical point in $V_0$,
which shows (c).
\end{proof}

\subsection{Local landscape and Fisher information}\label{sec:locallandscape}

We apply Lemma \ref{lemma:largesigmapseudomin} to analyze the 
Fisher information $I(\theta_*)=\nabla_\theta^2 R(\theta_*)$
and the local landscape of $R(\theta)$ near $\theta_*$. By rotational symmetry
of $R(\theta)$, the same statements hold locally around each point
in the orbit $\O_{\theta_*}$.

Recall the transcendence basis $\varphi$ in Lemma \ref{lemma:phiconstruction},
and the decompositions $d=d_1+\ldots+d_L$ and
$\varphi=(\varphi^1,\ldots,\varphi^L)$ according to the sequence of
subspaces $\cR_{\leq \ell}^G$ and their transcendence degrees. Lemma
\ref{lemma:polynomialreparam} establishes that $\varphi$ is a local
reparametrization around generic points $\theta_*$, and we will analyze the
landscape in this reparametrization.

\begin{theorem}\label{thm:locallargenoise}
Fix a choice of transcendence basis $\varphi=(\varphi^1,\ldots,\varphi^L)$
satisfying Lemma \ref{lemma:phiconstruction}, and let $\theta_* \in \R^d$ be 
a point with $\der_\theta \varphi(\theta_*)$ non-singular (which holds for
generic $\theta_*$). For some constants $C,c,\sigma_0>0$ and some
neighborhood $U$ of $\theta_*$, and for all $\sigma\geq \sigma_0$,
\begin{enumerate}[(a)]
\item In the reparametrization by $\varphi$, $R(\varphi)$ is strongly convex on
$\varphi(U)$ with $\lambda_{\min}(\nabla_\varphi^2 R(\varphi)) \geq c\sigma^{-2L}$.
\item The Fisher information matrix $I(\theta_*)$ has
$d_\ell$ eigenvalues belonging to $[c\sigma^{-2\ell},C\sigma^{-2\ell}]$
for each $\ell=1,\ldots,L$, where $d_\ell=\trdeg(\cR_{\leq \ell}^G)
-\trdeg(\cR_{\leq \ell-1}^G)$.
\item For any polynomial $\psi \in \cR^G_{\leq \ell}$, there is a constant $C>0$
(depending also on $\psi$) such that
\[\nabla_\theta \psi(\theta_*)^\top I(\theta_*)^{-1}
\nabla_\theta \psi(\theta_*) \leq C\sigma^{2\ell}.\]
\end{enumerate}
\end{theorem}

Note that part (c) describes the limiting variance in (\ref{eq:deltamethod}) for
estimating $\psi(\theta_*)$ by the plug-in maximum likelihood estimate
$\psi(\hat{\theta})$.

The proof of Theorem \ref{thm:locallargenoise} relies on the following
linear-algebraic result for any $\sigma$-dependent matrix with the graded block
structure of Definition \ref{def:gradedblocks}, and large enough $\sigma$. 

\begin{lemma}\label{lemma:gradedblocks}
Suppose $H \equiv H(\sigma) \in \R^{d \times d}$ has a graded block structure with respect
$(\varphi^1,\ldots,\varphi^L)$. Let $d_\ell$ be the dimension of each subvector
$\varphi^\ell$. Let $H_{:\ell,:\ell}$ and
$(H^{-1})_{:\ell,:\ell}$ denote the submatrices consisting of the upper-left
$\ell \times \ell$ blocks in the $L \times L$ block decompositions of $H$ and
$H^{-1}$. Then for some constants $C,c,\sigma_0>0$ and all $\sigma>\sigma_0$:
\begin{enumerate}[(a)]
\item $H$ has $d_\ell$ eigenvalues belonging to
$[c\sigma^{-2\ell},C\sigma^{-2\ell}]$ for each $\ell=1,\ldots,L$.
In particular, $\lambda_{\min}(H) \geq c\sigma^{-2L}$.
\item For each $\ell$ where $d_1+\ldots+d_\ell>0$,
$\lambda_{\min}(H_{:\ell,:\ell}) \geq c\sigma^{-2\ell}$.
\item For each $\ell$ where $d_1+\ldots+d_\ell>0$,
$\lambda_{\max}((H^{-1})_{:\ell,:\ell}) \leq C\sigma^{2\ell}$.
\end{enumerate}
\end{lemma}

\begin{proof}
We first show part (b). This holds for the smallest $\ell$ where
$d_1+\ldots+d_\ell>0$, by the definition of the graded block structure.
Assume inductively that it holds for $\ell \leq L-1$, and consider $\ell+1$
where $d_{\ell+1}>0$. For
any unit vector $v=(v_{:\ell},v_{\ell+1})$ where $v_{:\ell} \in \R^{d_1+\ldots+d_\ell}$ and
$v_{\ell+1} \in \R^{d_{\ell+1}}$, we have by the induction hypothesis and
Cauchy-Schwarz
\begin{align*}
v^\top H_{:(\ell+1),:(\ell+1)} v &=v_{:\ell}^\top H_{:\ell,:\ell} v_{:\ell}
+v_{\ell+1}^\top H_{\ell+1,\ell+1}v_{\ell+1}
+2v_{:\ell}^\top H_{:\ell,\ell+1}v_{\ell+1}\\
&\geq c\sigma^{-2\ell}\|v_{:\ell}\|^2+c\sigma^{-2(\ell+1)}\|v_{\ell+1}\|^2
-2C\sigma^{-2(\ell+1)}\|v_{:\ell}\|\|v_{\ell+1}\|\\
&\geq \Big(c\sigma^{-2\ell}-(2C/c)\sigma^{-2(\ell+1)}\Big)
\|v_{:\ell}\|^2+(c/2)\sigma^{-2(\ell+1)}\|v_{\ell+1}\|^2.
\end{align*}
For large $\sigma$, we get
$v^\top H_{:(\ell+1),:(\ell+1)} v \geq c'\sigma^{-2(\ell+1)}$ and some $c'>0$.
Hence (b) holds by induction for each $\ell=1,\ldots,L$.

Next, we show part (a). That $\lambda_{\min}(H) \geq c\sigma^{-2L}$ follows from (b).
For the first statement, for any $\ell$ where $d_\ell>0$, write
$H=H^{(\ell-1)}+R^{(\ell-1)}$ where $H^{(\ell-1)}$ equals
$H_{:(\ell-1),:(\ell-1)}$ on the upper-left $(\ell-1) \times (\ell-1)$
blocks and is 0 elsewhere, and $R^{(\ell-1)}$ is the
remainder. Part (a) implies that $H^{(\ell-1)}$ has $d_1+\ldots+d_{\ell-1}$
eigenvalues at least $c\sigma^{-2(\ell-1)}$, and remaining eigenvalues 0.
The graded block structure condition
implies $\|R^{(\ell-1)}\| \leq C\sigma^{-2\ell}$ for a constant $C>0$.
Then for a constant $c'>0$ and all large $\sigma$,
Weyl's inequality implies that $H$ has
$d_1 + \ldots + d_{\ell-1}$ eigenvalues at least $c'\sigma^{-2(\ell-1)}$,
and remaining eigenvalues at most $C\sigma^{-2\ell}$. Since this result holds for
every $\ell=1,\ldots,L$, this implies part (a).

Finally, for part (c), denote $G^{(\ell)}=[(H^{-1})_{:\ell,:\ell}]^{-1}$.
We claim that for all $\ell$ where $d_1+\ldots+d_\ell>0$, this matrix
$G^{(\ell)}$ has a graded block structure with respect to
$(\varphi^1,\ldots,\varphi^\ell)$.
That is to say, there are constants $C,c>0$ such that for all large $\sigma$
and all $1 \leq j,k \leq \ell$,
\begin{equation}\label{eq:inductioninverse}
C\sigma^{-2j} \geq \lambda_{\max}(G^{(\ell)}_{jj})
\geq \lambda_{\min}(G^{(\ell)}_{jj}) \geq c\sigma^{-2j} \quad
\text{ and } \quad \|G_{jk}^{(\ell)}\| \leq C\sigma^{-2\max(j,k)}.
\end{equation}
For $\ell=L$, we have $G^{(\ell)}=H$, so this holds by assumption.
Assume inductively that it holds for $\ell+1$, and consider $\ell$ where
$d_{\ell+1}>0$. Applying the
definition of $G^{(\ell)}$ and the Schur complement identity,
\[[G^{(\ell)}]^{-1}=([G^{(\ell+1)}]^{-1})_{:\ell,:\ell}
=\Big(G^{(\ell+1)}_{:\ell,:\ell}-G^{(\ell+1)}_{:\ell,\ell+1}[G^{(\ell+1)}_{\ell+1,\ell+1}]^{-1}
G^{(\ell+1)}_{\ell+1,:\ell}\Big)^{-1}.\]
Then
\[G^{(\ell)}=G^{(\ell+1)}_{:\ell,:\ell}-G^{(\ell+1)}_{:\ell,\ell+1}[G^{(\ell+1)}_{\ell+1,\ell+1}]^{-1}
G^{(\ell+1)}_{\ell+1,:\ell}.\]
We have $\|G^{(\ell+1)}_{:\ell,\ell+1}\| \leq C'\sigma^{-2(\ell+1)}$ and
$\|[G^{(\ell+1)}_{\ell+1,\ell+1}]^{-1}\| \leq C'\sigma^{2(\ell+1)}$ for some
$C'>0$, by the induction hypothesis.
For large enough $\sigma$, applying the induction hypothesis also to each block
of $G^{(\ell+1)}_{:\ell,:\ell}$, we get that (\ref{eq:inductioninverse}) holds
for $\ell$ (and some constants $C,c>0$ different from those for $\ell+1$).
Hence (\ref{eq:inductioninverse}) holds by
induction for each $\ell=1,\ldots,L$. Then, applying part (b) to this matrix
$G^{(\ell)}$ in place of $H$, we get that $\lambda_{\min}(G^{(\ell)}) \geq c\sigma^{-2\ell}$,
which implies $\lambda_{\max}((H^{-1})_{:\ell,:\ell}) \leq C\sigma^{2\ell}$.
This establishes (c).
\end{proof}

\begin{proof}[Proof of Theorem \ref{thm:locallargenoise}]
Since $\der_\theta \varphi(\theta_*)$ is non-singular, $\varphi=(\varphi^1,\ldots,\varphi^L)$
forms a local reparametrization on an open neighborhood of $\theta_*$.
We first show that $\theta_*$ is a pseudo-local-minimizer with respect to
this reparametrization. For this, we apply the form
\[S_\ell(\varphi)=\frac{1}{2(\ell!)}\|T_\ell(\varphi)-T_\ell(\varphi_*)\|_\HS^2
+Q_\ell(\varphi)\]
provided in Lemma \ref{lemma:Selldegell}, where $T_\ell(\varphi)$ and
$Q_\ell(\varphi)$ are shorthand for $T(\theta(\varphi))$ and
$Q(\theta(\varphi))$, and $\varphi_*=\varphi(\theta_*)$. Differentiating in $\varphi$,
\begin{align*}
\nabla_\varphi S_\ell(\varphi)&=\frac{1}{\ell!}\der_\varphi T_\ell(\varphi)^\top
(T_\ell(\varphi)-T_\ell(\varphi_*))+\nabla_\varphi Q_\ell(\varphi),\\
\nabla_\varphi^2 S_\ell(\varphi)&=\frac{1}{\ell!}\der_\varphi T_\ell(\varphi)^\top
\der_\varphi T_\ell(\varphi)
+\frac{1}{\ell!}\sum_i (T_\ell(\varphi)_i-T_\ell(\varphi_*)_i)
\nabla_\varphi^2 T_\ell(\varphi)_i+\nabla_\varphi^2 Q_\ell(\varphi).
\end{align*}
Here, $T_\ell(\varphi)_i$ is the $i^\text{th}$ entry of $T_\ell(\varphi)$
and the summation is over all multi-indices $i$. Note that $Q_\ell$ is in the algebra
generated by $\cR_{\leq \ell-1}^G$, so Lemma \ref{lemma:polynomialreparam}(c) ensures
that $Q_\ell$ depends only on $(\varphi^1,\ldots,\varphi^{\ell-1})$ on a neighborhood
of $\theta_*$. Thus,
evaluating the above at $\varphi=\varphi_*$ and restricting to the 
coordinates $\varphi^\ell$ yields
\[\nabla_{\varphi^\ell} S_\ell(\varphi_*)=0,
\qquad \nabla_{\varphi^\ell}^2
S_\ell(\varphi_*)=\frac{1}{\ell!}\der_{\varphi^\ell}
T_\ell(\varphi_*)^\top \der_{\varphi^\ell} T_\ell(\varphi_*).\]
In particular, $\nabla_{\varphi^\ell}^2 S_\ell(\varphi_*) \succeq 0$. To see
that $\nabla_{\varphi^\ell}^2 S_\ell(\varphi_*)$ has full rank $d_\ell$, observe
that every degree-$\ell$ polynomial of $\theta \in \R^d$ is a linear combination of
entries of the tensors $\theta^{\otimes 1},\ldots,\theta^{\otimes \ell}$. Thus,
symmetrizing by $G$, every polynomial in $\cR_{\leq \ell}^G$ is
a linear combination of entries of $T_1,\ldots,T_\ell$ (monomials). This means that
$\varphi^\ell=f(T_1(\varphi),\ldots,T_\ell(\varphi))$ for some linear function
$f:\R^{d+d^2+\ldots+d^\ell} \to \R^{d_\ell}$. Differentiating both sides
in $\varphi^\ell$ and observing that $T_1,\ldots,T_{\ell-1}$ do not depend on
$\varphi^\ell$ by Lemma \ref{lemma:polynomialreparam}(c), we obtain
\[\Id=(\der_{T_\ell} f)(\der_{\varphi^\ell} T_\ell),\]
where the left side is the $d_\ell \times d_\ell$ identity. Thus
$\der_{\varphi^\ell} T_\ell$ has full rank $d_\ell$, so
$\nabla_{\varphi^\ell}^2 S_\ell(\varphi_*) \succ 0$ and $\theta_*$ is a
pseudo-local-minimizer.

Then part (a) of the theorem
follows immediately from Lemma \ref{lemma:largesigmapseudomin}(b).
For (b) and (c), note that since $\nabla_\theta R(\theta_*)=0$, we have
from (\ref{eq:hessreparam}) that
\[I(\theta_*) \equiv \nabla_\theta^2 R(\theta_*)
=\Big(\der \varphi(\theta_*)^\top \cdot
\nabla_\varphi^2 R(\varphi_*) \cdot \der \varphi(\theta_*)\Big)\]
where $\varphi_*=\varphi(\theta_*)$. Then setting
$\tilde{V}=\der \varphi(\theta_*)^{-1}$,
Lemma \ref{lemma:largesigmapseudomin} shows that
$\nabla_\varphi^2 R(\varphi_*)=\tilde{V}^\top I(\theta_*)\tilde{V}$ has a graded
block structure. For any polynomial $\psi \in
\cR^G_{\leq \ell}$, Lemma \ref{lemma:polynomialreparam} shows that $\psi$
is an analytic function of $\varphi^1,\ldots,\varphi^\ell$, and hence that
$\nabla_\varphi \psi=\tilde{V}^\top \nabla_\theta \psi$ is non-zero only in
its first $\ell$ blocks. Writing
\[\nabla_\theta \psi(\theta_*)^\top I(\theta_*)^{-1}
\nabla_\theta \psi(\theta_*)=\Big(\tilde{V}^\top \nabla_\theta
\psi(\theta_*)\Big)^\top \Big(\tilde{V}^\top I(\theta_*)\tilde{V}\Big)^{-1}
\Big(\tilde{V}^\top \nabla_\theta \psi(\theta_*)\Big),\]
part (c) then follows from Lemma \ref{lemma:gradedblocks}(c). Also, by the QR 
decomposition, there is a non-singular lower-triangular matrix $W$ for which
$V=\tilde{V}W$ is orthogonal. It may be verified from Definition
\ref{def:gradedblocks} that the matrix
$V^\top I(\theta_*)V=W^\top (\tilde{V}^\top I(\theta_*)\tilde{V})W$ also has a
graded block structure,
for modified constants $C,c,\sigma_0$.
As the eigenvalues of $V^\top I(\theta_*)V$ are the same as those of
$I(\theta_*)$, this and Lemma \ref{lemma:gradedblocks}(a) show part (b).
\end{proof}

The following then shows that with high probability for $n \gg \sigma^{2L}$,
the empirical risk $R_n(\varphi)$ is also strongly convex with
a local minimizer in $\varphi(U)$. This requirement for $n$ matches the
requirement for list-recovery of generic signals in
\cite{bandeira2017estimation}.

\begin{corollary}\label{cor:locallargenoise}
In the setting of Theorem \ref{thm:locallargenoise}, for
$(\theta_*,d,G)$-dependent constants $C,c,c',\sigma_0>0$, with probability
$1-e^{c(\log n)^2}-Ce^{-cn^{1/(2L)}\sigma^{-1}}$,
we have $\lambda_{\min}(\nabla_\varphi^2 R_n(\varphi)) \geq
c'\sigma^{-2L}$ for all $\varphi \in \varphi(U)$,
and $R_n(\theta)$ has a critical point and unique local minimizer
in $U$.
\end{corollary}
\begin{proof}
Note that
applying directly Lemma \ref{lem:localmincompare} and the general
concentration result of Lemma \ref{lem:concentration} with $t \asymp
\sigma^{-2L}$ small enough, we may obtain that this corollary holds with
probability at least
$1-e^{-c\sigma^{-4L+2}n+C\log \sigma}-e^{-cn^{2/3}}$.

To strengthen this probability guarantee, we apply a
concentration argument tailored to large $\sigma$ based on Lemma
\ref{lemma:Rnexpansion}: We assume throughout that $n \geq \sigma^{2L}$, 
as otherwise the desired probability guarantee is vacuous.
Let $A_{\ell,m}$ and $P_{\ell,m}$ be the polynomials of
Lemma \ref{lemma:Rnexpansion}. Since $A_{\ell,m}$ has degree at most
$\ell$ in $\eps$, by Gaussian hypercontractivity, for any $t>0$ we have
\begin{equation}
\P\Bigg[\bigg|\frac{1}{n}\sum_{i=1}^n A_{\ell,m}(\eps_i,\theta_*)
-\E_\eps[A_{\ell,m}(\eps,\theta_*)]\bigg|>t\Bigg]
\leq 2e^{-c(nt^2)^{1/\ell}}.
\label{eq:gaussianpoly}
\end{equation}
(This follows also from Theorem \ref{thm:AW15} applied with
$\eps=(\eps_1,\ldots,\eps_n)$, $f(\eps)=\sum_i
A_{\ell,m}(\eps_i,\theta_*)$, and the bounds $\|\nabla^\ell f(x)\|_\cJ
\leq \|\nabla^\ell f(x)\|_\HS \leq C\sqrt{n}$ and $\|\E[\nabla^j f(\eps)]\|_\cJ
\leq \|\E[\nabla^j f(\eps)]\|_\HS \leq C\sqrt{n}$ for any partition $\cJ$ and
any $j \leq \ell-1$.) For a sufficiently small constant $c_0>0$ to be chosen
later, let $\cE$ be the event where the conclusion of Lemma
\ref{lemma:Rnexpansion} holds with $k=2L$, $t=\sigma^{-2L}/\log \sigma$,
and $r(\sigma)$ a constant larger than $\|\theta_*\|$, and also where
\begin{equation}\label{eq:Abound}
\bigg|\frac{1}{n}\sum_{i=1}^n
A_{\ell,m}(\eps_i,\theta_*)-\E_\eps[A_{\ell,m}(\eps,\theta_*)]\bigg|
\leq c_0\sigma^{-(\ell \wedge L)}
\end{equation}
for every $\ell=1,\ldots,2L$ and $m=1,\ldots,M_\ell$. Note that
\eqref{eq:gaussianpoly} implies (\ref{eq:Abound}) holds with
probability at least $1-2e^{-cn^{1/\ell}\sigma^{-2}} \geq
1-2e^{-cn^{1/L}\sigma^{-2}}$ for $\ell \leq L$, and at least
$1-2e^{-c(n\sigma^{-2L})^{1/\ell}} \geq 1-2e^{-cn^{1/(2L)}\sigma^{-1}}$
for $2L \geq \ell \geq L$ and $n \geq \sigma^{2L}$. Thus, combining with the
probability guarantee in Lemma \ref{lemma:Rnexpansion} and taking a union bound,
\[\P[\cE] \geq 1-e^{c(\log n)^2}
-e^{-\frac{cn}{\sigma^{2L-1}\log \sigma \cdot (\log n)^{2L}}
+C\log \sigma \cdot \log \log n}
-Ce^{-\frac{cn^{1/(2L)}}{\sigma}}
\geq 1-e^{c'(\log n)^2}-C'e^{-c'n^{1/(2L)}\sigma^{-1}}.\]

We now restrict to this event $\cE$ and parametrize by $\varphi$.
There is a $\sigma$-independent neighborhood $U$ of $\theta_*$
on which $\nabla_\varphi^2 R^L(\varphi)$ has a graded block structure:
for each $\ell=1,\ldots,L$ and every $\varphi \in \varphi(U)$,
\[C\sigma^{-2\ell} \geq \lambda_{\max}(\nabla_{\varphi^\ell}^2 R^L(\varphi))
\geq \lambda_{\min}(\nabla_{\varphi^\ell}^2 R^L(\varphi)) \geq
c\sigma^{-2\ell}, \quad
\|\nabla_{\varphi^\ell,\varphi^{\ell'}}^2 R^L(\varphi)\| \leq
C\sigma^{-2\max(\ell,\ell')}.\]
Let us denote the first $2L$ terms of the expansion in Lemma
\ref{lemma:Rnexpansion} by
\[R_n^L(\varphi)=\sum_{\ell=1}^{2L} \sum_{m=1}^{M_\ell}
\sigma^{-\ell} \frac{1}{n}\sum_{i=1}^n
A_{\ell,m}(\eps_i,\theta_*)P_{\ell,m}(\varphi),\]
observe that $\E[R_n^L(\varphi)]=R^L(\varphi)$, and write
\begin{equation}\label{eq:Rdiff}
\nabla^2 R_n(\varphi)-\nabla^2 R^L(\varphi)
=\Big(\nabla^2 R_n(\varphi)-\nabla^2 R_n^L(\varphi)\Big)
+\Big(\nabla^2 R_n^L(\varphi)-\nabla^2 R^L(\varphi)\Big).
\end{equation}
On the event $\cE$, the first term is controlled by Lemma
\ref{lemma:Rnexpansion}, and for our choice of $t=\sigma^{-2L}/\log \sigma$
we have
\begin{equation}\label{eq:Rnapprox}
\|\nabla^2 R_n(\theta)-\nabla^2 R_n^L(\theta)\Big\| \leq \frac{C}{\sigma^{2L}
\log \sigma}.
\end{equation}
For the second term, observe that $P_{\ell,m} \in \cR_{\leq \ell}^G$ and hence
$P_{\ell,m}$ depends only on the coordinates $\varphi^1,\ldots,\varphi^\ell$.
Thus for any $\ell,\ell' \leq L$,
\[\nabla_{\varphi^\ell,\varphi^{\ell'}}^2
(R_n^L(\varphi)-R^L(\varphi))
=\sum_{k=\max(\ell,\ell')}^{2L} \sum_{m=1}^{M_k} \sigma^{-k}
\cdot \left(\frac{1}{n}\sum_{i=1}^n A_{k,m}(\eps_i,\theta_*)
-\E_\eps[A_{k,m}(\eps,\theta_*)]\right) \nabla_{\varphi^\ell,\varphi^{\ell'}}^2
P_{k,m}(\varphi).\]
Then on the event $\cE$, applying (\ref{eq:Abound}) and $k+(k \wedge L)
\geq 2k \wedge 2L$, also
\[\|\nabla_{\varphi^\ell,\varphi^{\ell'}}^2
R_n^L(\varphi)-\nabla_{\varphi^\ell,\varphi^{\ell'}}^2 R^L(\varphi)\|
\leq C \cdot c_0\sigma^{-2\max(\ell,\ell')}.\]
Choosing $c_0$ sufficiently small and combining with (\ref{eq:Rdiff})
and (\ref{eq:Rnapprox}), we obtain that the Hessian of the empirical risk
$\nabla_\varphi^2 R_n(\varphi)$ also has a graded block structure over
$\varphi \in U$. The lower bound
$\lambda_{\min}(\nabla_\varphi^2 R_n(\varphi)) \geq c'\sigma^{-2L}$ then follows
immediately from Lemma \ref{lemma:gradedblocks}(b).

To show that $R_n(\varphi)$ has a critical point (and hence unique local
minimizer) in $\varphi(U)$, we apply an argument similar to that of Lemma
\ref{lemma:largesigmapseudomin}(c). Let $\hat{\varphi} \in
\overline{\varphi(U)}$ be any minimizer of $R_n(\varphi)$ on the
closure of $\varphi(U)$.
We aim to show that $\hat{\varphi}$ cannot occur on the boundary of $\varphi(U)$. Recall from the proof of Theorem
\ref{thm:locallargenoise} that $\varphi_*=\varphi(\theta_*)$ is a pseudo-local-minimizer for the
sequence $S_1,\ldots,S_L$ with respect to $(\varphi^1,\ldots,\varphi^L)$.
In particular, $\varphi_*$ minimizes $S_1$ over any sufficiently small
neighborhood $U$, so we
have $S_1(\hat{\varphi})-S_1(\varphi_*) \geq c\|\hat{\varphi}^1-\varphi_*^1\|^2$
and
\[R^L(\hat{\varphi})-R^L(\varphi_*) \geq c\sigma^{-2}
\|\hat{\varphi}^1-\varphi_*^1\|^2-C\left(\frac{\log \sigma}{\sigma}\right)^4.\]
This implies also for the empirical risk $R_n$ that, on the event $\cE$,
\begin{align*}
R_n(\hat{\varphi})-R_n(\varphi_*)
&\geq R_n(\hat{\varphi})-R^L(\hat{\varphi})-R_n(\varphi_*)+R^L(\varphi_*)
+c\sigma^{-2}
\|\hat{\varphi}^1-\varphi_*^1\|^2-C\left(\frac{\log \sigma}{\sigma}\right)^4\\
&\geq \sum_{k=1}^{2L} \sum_{m=1}^{M_k} \sigma^{-k}
\left(\frac{1}{n}\sum_{i=1}^n A_{k,m}(\eps_i,\theta_*)
-\E_\eps[A_{k,m}(\eps,\theta_*)]\right)\Big(P_{k,m}(\hat{\varphi})
-P_{k,m}(\varphi_*)\Big)\\
&\hspace{1in}-\frac{C}{\sigma^{2L}\log \sigma}+c\sigma^{-2}
\|\hat{\varphi}^1-\varphi_*^1\|^2-C\left(\frac{\log \sigma}{\sigma}\right)^4\\
&\geq c\sigma^{-2}\|\hat{\varphi}^1-\varphi_*^1\|^2
-C' \cdot c_0\sigma^{-2}\|\hat{\varphi}^1-\varphi_*^1\|-C''\left(\frac{\log
\sigma}{\sigma}\right)^4,
\end{align*}
where the last line applies (\ref{eq:Abound}) and the fact 
that the polynomials $P_{1,m}$ depend only on $\varphi^1$ and must be
Lipschitz over $\overline{\varphi(U)}$.
Since $\hat{\varphi}$ minimizes $R_n$, we have $0 \geq
R_n(\hat{\varphi})-R_n(\varphi_*)$. Then this implies for some constant $C>0$
and all $\sigma>\sigma_0$ that $\|\hat{\varphi}^1-\varphi_*^1\|
\leq C \cdot c_0$.

Now defining $f(\varphi^2)=S_2(\varphi_*^1,\varphi^2)$ and
$\hat{f}(\varphi^2)=S_2(\hat{\varphi}^1,\varphi^2)$, these functions
must be Lipschitz over $\overline{\varphi^2(U)}$, so this yields
$|f-\hat{f}|,\|\nabla^2 f-\nabla^2 \hat{f}\| \leq C' \cdot c_0$ on
$\overline{\varphi^2(U)}$. Since $\varphi_*$ is a pseudo-local-minimizer, 
taking the neighborhood $U$ sufficiently small ensures that the
function $f$ is convex over $\varphi^2(U)$, and is minimized at $\varphi_*^2$.
Then for $c_0$ sufficiently small, Lemma
\ref{lem:localmincompare} guarantees that $\hat{f}$
also has a critical point and local minimizer $\bar{\varphi}^2$, which satisfies
\begin{equation}\label{eq:barvarphibound}
\|\bar{\varphi}^2-\varphi_*^2\| \leq C \cdot c_0
\end{equation}
and $S_2(\hat{\varphi}^1,\hat{\varphi}^2)-S_2(\hat{\varphi}^1,\bar{\varphi}^2)
\geq c\|\hat{\varphi}^2-\bar{\varphi}^2\|^2$. Then
\[R^L(\hat{\varphi})-R^L((\hat{\varphi}^1,\bar{\varphi}^2,\varphi_*^3,\ldots,\varphi_*^L)) \geq c\sigma^{-4}
\|\hat{\varphi}^2-\bar{\varphi}^2\|^2
-C\left(\frac{\log \sigma}{\sigma}\right)^6,\]
where there is no $\sigma^{-2}$ term because $S_1$ depends only on $\varphi^1$,
which coincides in these two arguments of $R^L$.
Applying a similar argument as above, this implies on $\cE$ that
\begin{align*}
0 &\geq
R_n(\hat{\varphi})-R_n((\hat{\varphi}^1,\bar{\varphi}^2,\varphi_*^3,\ldots,\varphi_*^L))\\
&\geq c\sigma^{-4}
\|\hat{\varphi}^2-\bar{\varphi}^2\|^2
-C' \cdot c_0\sigma^{-4}\|\hat{\varphi}^2-\bar{\varphi}^2\|
-C''\left(\frac{\log \sigma}{\sigma}\right)^6.
\end{align*}
Then for some constant $C>0$, $\|\hat{\varphi}^2-\bar{\varphi}^2\| \leq C \cdot
c_0$. Combining with the preceding bound (\ref{eq:barvarphibound}),
we get $\|\hat{\varphi}^2-\varphi_*^2\| \leq C' \cdot c_0$.

Now defining $f(\varphi^3)=S_3(\varphi_*^1,\varphi_*^2,\varphi^3)$ and
$\hat{f}(\varphi^3)=S_2(\hat{\varphi}^1,\hat{\varphi}^2,\varphi^3)$, we may
repeat this argument to obtain
$\|\hat{\varphi}^3-\varphi_*^3\| \leq C \cdot c_0$, and so forth. This
establishes $\|\hat{\varphi}^\ell-\varphi_*^\ell\| \leq C \cdot c_0$ for some
constant $C>0$ and each $\ell=1,\ldots,L$. Finally, for $c_0$ sufficiently
small, this implies that the minimizer $\hat{\varphi}$ must belong to the interior of
$\varphi(U)$. Since $R_n(\varphi)$ is differentiable and convex over
$\varphi(U)$, this implies that $\hat{\varphi}$ must be a unique critical
point and local minimizer of $R_n(\varphi)$ over $\varphi(U)$.
\end{proof}

\subsection{Globally benign landscapes at high noise}\label{sec:globalbenign}

In the following three subsections, we apply the tools of
Section \ref{sec:localreparam} to analyze three
examples in which the landscapes of $R(\theta)$ and $R_n(\theta)$ are
\emph{globally} benign in this high-noise regime
$\sigma>\sigma_0(\theta_*,d,G)$, for generic $\theta_* \in \R^d$.

In each example, for each fixed point $\wtheta \in \R^d$, we study the landscape of
$R(\theta)$ near $\wtheta$ using a local reparametrization
$\varphi=(\varphi^1,\ldots,\varphi^L)$ around $\wtheta$. Note that, in
general, we cannot use the same reparametrization $\varphi$ at all points
$\wtheta \in \R^d$, as we must handle non-generic points where
$\der_\theta \varphi(\wtheta)$ is singular for any particular map $\varphi$, even
if the true parameter $\theta_*$ is generic.

We will combine these local statements over a large enough ball
$\{\theta \in \R^d:\|\theta\| \leq M\}$ using a compactness
argument. The following result strengthens Lemmas \ref{lemma:thetabound}
and \ref{lemma:localization} to
provide a lower bound for $\|\nabla R(\theta)\|$ and $\|\nabla
R_n(\theta)\|$ outside this ball.

\begin{lemma}\label{lemma:localizationstrong}
For some $(\theta_*,d,G)$-dependent constants $M,\rho,c_0,\sigma_0>0$ and all
$\sigma>\sigma_0$, if $\|\theta\|>M$, then $\|\nabla R(\theta)\|>c_0\sigma^{-4}$.
If, in addition, $\|\E_g[g\theta]-\E_g[g\theta_*]\|>\rho$, then $\|\nabla
R(\theta)\|>c_0\sigma^{-2}$. For some $(\theta_*,d,G)$-dependent constants
$C,c>0$, with probability at least
$1-e^{-c(\log n)^2}-Ce^{-cn^{1/4}\sigma^{-1}}$,
the same bounds hold for $\nabla R_n(\theta)$ and all $\|\theta\|>M$.
\end{lemma}

\begin{proof}
For the empirical risk $R_n(\theta)$,
we may assume $n \geq \sigma^4$, as otherwise the desired probability
guarantee is vacuous. By Lemma \ref{lemma:thetabound}, for all
$\|\theta\| \geq C\sigma$ and a large enough constant $C>0$, we have
$\|\nabla R(\theta)\| \geq c\sigma^{-1}$ and
$\|\nabla R_n(\theta)\| \geq c\sigma^{-1}$ with probability $1-C'e^{-cn}$. By
Lemma \ref{lemma:localization}, if $C'\sigma^{2/3} \leq \|\theta\| \leq
C\sigma$, then $\|\nabla R(\theta)\|\geq c\sigma^{-2}$.
Applying the concentration bound \eqref{eq:concgradRn} in Lemma \ref{lem:concentration} with $r=C\sigma$ and $t=c\sigma^{-2}/2$, we get 
$\|\nabla R_n(\theta)\| \geq c\sigma^{-2}/2$ for all such $\theta$,
with probability $1-e^{-c n \sigma^{-2}+C\log \sigma} \geq
1-C'e^{-c'n^{1/4}\sigma^{-1}}$ when $n \geq \sigma^4$.

It remains to consider $M<\|\theta\| \leq C\sigma^{2/3}$. We consider
two cases:

{\bf Case I.} Suppose that $\E_g[g]=0$. Then Lemma \ref{lemma:Smeanzero} implies
$S_1(\theta)=0$ and
\[\nabla S_2(\theta)=\E_g[g\theta \theta^\top g^\top]\theta
-\E_g[g\theta_*\theta_*^\top g^\top]\theta.\]
For every unit vector $v \in \R^d$, we have
$v^\top \E_g[gvv^\top g^\top]v=\E_g[\|v^\top gv\|^2] \geq 1/K$ where $K=|G|$,
because $g=\Id$ with probability $1/K$. Thus $\|\E_g[gvv^\top g^\top]v\| \geq
1/K$, so $\|\E_g[g\theta \theta^\top g^\top]\theta\| \geq \|\theta\|^3/K$.
For sufficiently large $M$, this shows $\|\nabla S_2(\theta)\| \geq
c\|\theta\|^3$. Then by Lemma \ref{lemma:seriesexpansion},
for a constant $c_0>0$ (independent of $M$), we have
\begin{equation}\label{eq:populationgradlowerbound}
\|\nabla R(\theta)\| \geq c_0\sigma^{-4}\|\theta\|^3 \geq c'\sigma^{-4}.
\end{equation}

For the empirical risk $R_n$, write
	\begin{equation}
	\nabla R_n(\theta)-\nabla R(\theta) = \underbrace{\nabla R_n(\theta) - \nabla
R_n^2(\theta)}_{\cE_1} + \underbrace{\nabla R_n^2(\theta) - \nabla
R^2(\theta)}_{\cE_2} + \underbrace{\nabla R^2(\theta) - \nabla
R(\theta)}_{\cE_3}
	\label{eq:nablaRn-decomp}
	\end{equation}
	where 
	$R_n^2(\theta) = \sum_{\ell=2}^{4} \sigma^{-\ell} \cdot
\frac{1}{n}\sum_{i=1}^n P_\ell(\eps_i,\theta,\theta_*)$
is the degree-$4$ approximation to $R_n(\theta)$ in Lemma \ref{lemma:Rnexpansion} and 
$R^2(\theta)= \E[R_n^2(\theta)]$.
Note that the $\ell=1$ term is absent in $R_n^2$, because by
\eqref{eq:Pell-explicit} we have $P_1(\eps,\theta,\theta_*) = \kappa_1(\langle
\eps,g\theta\rangle) = \eps^\top \E[g] \theta = 0$. By Lemma
\ref{lemma:Smeanzero}, we have $R^2(\theta)= \sigma^{-4} S_2(\theta)$ since $S_1(\theta)=0$.

	The first term in \eqref{eq:nablaRn-decomp} is controlled by Lemma
\ref{lemma:Rnexpansion}: Applying \eqref{eq:Rnexpansion1} with $k=4$,
$r(\sigma)=C\sigma^{2/3}$ and $t=(\sigma^{-1}\log \sigma)^5$, we get 
$\|\cE_1\| \leq C(\sigma^{-1}\log \sigma)^5 \|\theta\|^4$
with probability at least $1-Ce^{-c(\log n)^2}-e^{-\frac{c n (\log
\sigma)^5}{\sigma^4 \log n}+C\log \sigma \cdot \log \log n} \geq 1-Ce^{-c(\log
n)^2}-C'e^{-c'n^{1/4}\sigma^{-1}}$.
For the third term in \eqref{eq:nablaRn-decomp}, Lemma
\ref{lemma:seriesexpansion} yields
$\|\cE_3\| \leq C (\sigma^{-1}\log \sigma)^{6} \|\theta\|^5$.
	For the second term in \eqref{eq:nablaRn-decomp}, using \eqref{eq:Pell-form} we have
	\begin{equation}
	\cE_2 = \sum_{\ell=2}^4 \sigma^{-\ell} \sum_{m=1}^{M_\ell}
\frac{1}{n}\sum_{i=1}^n
\Big(A_{\ell,m}(\eps_i,\theta_*)-\E_\eps[A_{\ell,m}(\eps,\theta_*)]\Big) \nabla P_{\ell,m}(\theta).
	\label{eq:cE2}
	\end{equation}
	Applying the polynomial concentration \eqref{eq:gaussianpoly}, we have
with probability $1-Ce^{- c n^{1/4} \sigma^{-1}}$
	\begin{equation}
\bigg|\frac{1}{n}\sum_{i=1}^n
A_{\ell,m}(\eps_i,\theta_*)-\E_\eps[A_{\ell,m}(\eps,\theta_*)]\bigg|
\leq C \sigma^{-2}, \quad m=1,\ldots,~M_\ell,\;\;\ell=2,3,4.	
\end{equation}
As stated in Lemma \ref{lemma:Rnexpansion}, each $\nabla P_{\ell,m}(\theta)$ is a polynomial in $\theta$ of degree at most $\ell-1$. Thus $\|\nabla P_{\ell,m}(\theta)\| \leq  C(\|\theta\|\vee 1)^{\ell-1}$ for all $m=1,\ldots,M_\ell$. 	
Combining this with the previous two displays and using 
$\|\theta\|\ll \sigma$ yields
$\|\cE_2\| \leq C \sigma^{-2} \sum_{\ell=2}^4 \sigma^{-\ell}\|\theta\|^{\ell-1}
\leq C'\sigma^{-4}\|\theta\|$. For sufficiently large $M$, this
implies $\|\cE_1\|+\|\cE_2\|+\|\cE_3\| \leq (c_0/2)\sigma^{-4}\|\theta\|^3$
where $c_0$ is the constant in (\ref{eq:populationgradlowerbound}). Thus also
\[\|R_n(\theta)\| \geq (c_0/2)\sigma^{-4}\|\theta\|^3 \geq c'\sigma^{-4}.\]
This concludes the proof in the case $\E_g[g]=0$.

{\bf Case II.} Suppose $\E_g[g] \neq 0$. We apply
Lemmas \ref{lemma:productgroup} and \ref{lemma:kerneldecomp} to write
$R(\theta)=R^{\Id}(\theta_1)+R^{G_2}(\theta_2)$ and
$R_n(\theta)=R_n^{\Id}(\theta_1)+R_n^{G_2}(\theta_2)$, where $\theta_1 \in
\R^{d_1}$ and $\theta_2 \in \R^{d_2}$ are the components
of $\theta$ orthogonal to and belonging to the kernel of $\E_g[g]$.
Then $\|\theta-\theta_*\|^2=\|\theta_1-\theta_{1,*}\|^2+
\|\theta_2-\theta_{2,*}\|^2$, $\|\nabla R(\theta)\|^2=
\|\nabla R^{\Id}(\theta_1)\|^2+\|\nabla R^{G_2}(\theta_2)\|^2$,
and $\|\nabla R_n(\theta)\|^2
=\|\nabla R_n^{\Id}(\theta_1)\|^2+\|\nabla R_n^{G_2}(\theta_2)\|^2$.
Recall from the proof
of Lemma \ref{lemma:kerneldecomp} that $\E_g[g]$ is the projection
orthogonal to its kernel, so 
$\|\E_g[g\theta]-\E_g[g\theta_*]\|=\|\theta_1-\theta_{1,*}\|$.
Since $R^{\Id}(\theta_1)$ and $R_n^{\Id}(\theta_1)$ correspond
to the single Gaussian model
$\N(\theta_{1,*},\sigma^2\Id_{d_1 \times d_1})$, we may verify that
\[\nabla R^{\Id}(\theta_1)=\frac{\theta_1-\theta_{1,*}}{\sigma^2},
\qquad \nabla R_n^{\Id}(\theta_1)=\frac{\theta_1-\theta_{1,*}}{\sigma^2}
-\frac{\bar{\eps}}{\sigma},
\qquad \bar{\eps} \sim \N(0,\Id_{d_1 \times d_1}/n).\]
Thus, if $\|\E_g[g\theta]-\E_g[g\theta_*]\|=\|\theta_1-\theta_{1,*}\|>\rho$, then 
$\|\nabla R(\theta)\| \geq \|\nabla R^{\Id}(\theta_1)\| \geq \rho\sigma^{-2}$
and $\|\nabla R_n(\theta)\| \geq \|\nabla R_n^{\Id}(\theta_1)\| \geq
(\rho/2)\sigma^{-2}$ with probability at least $1-Ce^{-cn\sigma^{-2}}$.
Otherwise, for small enough $\rho>0$, we have
$\|\theta_2-\theta_{2,*}\|>M/2$. Applying the argument of Case I
for the mean-zero group $G_2$ shows $\|\nabla R(\theta)\| \geq 
\|\nabla R^{G_2}(\theta_2)\| \geq c\sigma^{-4}$
and $\|\nabla R_n(\theta)\| \geq 
\|\nabla R_n^{G_2}(\theta_2)\| \geq c\sigma^{-4}$ as desired.
\end{proof}

\subsubsection{Discrete rotations in $\R^2$}\label{sec:rotations}

We consider first the group of $K$-fold discrete rotations on $\R^2$:
For a fixed integer $K$, we have
\begin{equation}\label{eq:GrotationsR2}
G=\{\Id,h,h^2,\ldots,h^{K-1}\} \cong \ZZ/K\ZZ
\end{equation}
where
\begin{equation}\label{eq:rotationgenerator}
h=\begin{pmatrix} \cos 2\pi/K & -\sin 2\pi/K \\
\sin 2\pi/K & \cos 2\pi/K \end{pmatrix}
\end{equation}
is the counterclockwise rotation in the plane by the angle $2\pi/K$.
For fixed $\theta_* \neq 0$ and for any $\theta \neq 0$, denote
\[t(\theta)=\arccos \frac{\langle \theta,\theta_* \rangle}
{\|\theta\|\|\theta_*\|}\]
as the angle formed by $\theta$ and $\theta_*$.

The special case of $K=2$ and $G=\{+\Id,-\Id\}$ is subsumed by results
of \cite[Corollary 3]{xu2016global},
which imply that the global landscape of $R(\theta)$ is benign for all
$\sigma>0$. Thus, we consider here the setting where $K \geq 3$.

\begin{theorem} \label{thm:rotations}
Let $G$ be the group of rotations (\ref{eq:GrotationsR2}) on $\R^2$, with $K
\geq 3$. Consider $\theta_* \neq 0$.
There exists a $(\theta_*,K)$-dependent constant $\sigma_0$ such that for all
$\sigma>\sigma_0$, the landscape of $R(\theta)$ is globally benign. More
quantitatively, for small enough $\rho > 0$, there are $(\theta_*,K)$-dependent
constants $c, \sigma_0>0$ such that when $\sigma>\sigma_0$,
\begin{enumerate}[(a)]
\item For each $\wtheta \in \O_{\theta_*}$, reparametrizing by
$\varphi=(\|\theta\|,t(\theta))$ on $B_\rho(\wtheta)$,
we have the strong convexity
$\lambda_{\min}(\nabla_\varphi^2 R(\varphi)) \geq c\sigma^{-2K}$ for all $\varphi \in
\varphi(B_\rho(\wtheta))$.
\item For each $\theta \in \R^d$ satisfying $\|\theta\|-\|\theta_*\| \in
(-\rho,\rho)$ and $\theta \notin \bigcup_{\wtheta \in \O_{\theta_*}}
B_\rho(\wtheta)$, either $\|\nabla R(\theta)\| \geq c\sigma^{-2K}$
or $\lambda_{\min}(\nabla^2 R(\theta)) \leq -c\sigma^{-2K}$.
\item For each $\theta \in \R^d$ satisfying $\|\theta\|-\|\theta_*\| \notin
(-\rho,\rho)$, either $\|\nabla_\theta R(\theta)\| \geq c\sigma^{-4}$ or
$\lambda_{\min}(\nabla_\theta^2 R(\theta)) \leq -c\sigma^{-4}$.
\end{enumerate}
\end{theorem}

The proof rests on the following lemma, which characterizes the functions
$S_\ell(\theta)$ in \eqref{eq:Sl} for this discrete rotation group.

\begin{lemma}\label{lemma:Slrotations}
Let $G$ be the group of rotations (\ref{eq:GrotationsR2}) on $\R^2$, with $K
\geq 3$. Then
\begin{enumerate}[(a)]
\item $S_1(\theta) = 0$ and
$S_2(\theta)=\|\theta\|^4/8-\|\theta\|^2\|\theta_*\|^2/4$.
\item For each $\ell \in \{3,\ldots,K-1\}$,
$S_\ell(\theta)=p_\ell(\|\theta\|^2)$ for some univariate polynomial
$p_\ell:\R \to \R$ (with coefficients depending on $\theta_*$).
\item For $\ell=K$ and
some polynomial $p_K:\R \to \R$ (with coefficients depending on $\theta_*$),
\[S_K(\theta)=-\frac{1}{2^{K-1}K!} \|\theta\|^K \|\theta_*\|^K
\cos(K \cdot t(\theta)) + p_K(\|\theta\|^2).\]
\end{enumerate}
\end{lemma}
\begin{proof}
Let $z = (\theta_*)_1 + \ii (\theta_*)_2$ and $w = \theta_1 + \ii \theta_2$
as elements of $\C$. Let $\zeta = e^{2\pi \ii/K}$, and denote the set of
$K^\text{th}$ roots of unity by $X_K=\{1, \zeta, \ldots, \zeta^{K-1}\}$. Then
$\zeta^k z=(h^k\theta_*)_1+\ii (h^k\theta_*)_2$ where $h$ is the generator
(\ref{eq:rotationgenerator}), and similarly for $w$ and
$\theta$. Notice that for $a = a_1 + \ii a_2$ and $b = b_1 + \ii b_2$ we have
\[\langle (a_1,a_2),(b_1,b_2) \rangle
=a_1 b_1 + a_2 b_2
= \frac{1}{4}\Big((a + \bar{a})(b + \bar{b}) - (a - \bar{a})(b - \bar{b})\Big)
= \frac{1}{2}\Big(a \bar{b} + \bar{a} b\Big).
\]
Then
\[\E_g[\langle \theta_*,g\theta \rangle^2]
=\frac{1}{4K}\sum_{\zeta \in X_K}
\left(\zeta^{-1}z\bar{w}+\zeta\bar{z}w\right)^2=\frac{|z|^2|w|^2}{2}
=\frac{\|\theta\|^2\|\theta_*\|^2}{2},\]
where we have used $K \geq 3$ and
\begin{equation}\label{eq:sumrootsofunity}
\sum_{\zeta \in X_K} \zeta^a = \begin{cases} K & \text{ if } a \equiv 0 \mod K
\\ 0 & \text{ if } a \not\equiv 0 \mod K\end{cases}
\end{equation}
for the second equality. Similarly $\E_g[\langle \theta,g\theta \rangle^2]
=\|\theta\|^4/2$, and (a) follows from Lemma \ref{lemma:Smeanzero}.

Applying this argument for a general term
$M_{\ell,m}(\pi \mid \theta,\theta_*)$, we have
\begin{align*}
&M_{\ell,m}(\pi \mid \theta,\theta_*)\\
&=\E_{g_1,\ldots,g_{|\pi|}}
\left[\prod_{j=1}^m \Big\langle g_{\pi(2j-1)}\theta,\;g_{\pi(2j)}\theta
\Big\rangle \cdot \prod_{j=2m+1}^{\ell+m} \Big\langle \theta_*,
\;g_{\pi(j)}\theta \Big\rangle\right]\\
&= \frac{1}{2^\ell K^{|\pi|}} \sum_{i_1, \ldots, i_{|\pi|} = 0}^{K-1}
\left[\prod_{j = 1}^m
\Big((\zeta^{i_{\pi(2j - 1)} - i_{\pi(2j)}} + \zeta^{i_{\pi(2j)} - i_{\pi(2j -
1)}})|w|^2\Big)
\prod_{j = 2m + 1}^{\ell + m} (\zeta^{-i_{\pi(j)}} z \bar{w} + \zeta^{i_{\pi(j)}} \bar{z} w) \right]\\
&= \frac{|w|^{2m}}{2^\ell K^{|\pi|}} \sum_{\zeta_1, \ldots, \zeta_{|\pi|} \in
X_K} \left[\prod_{j = 1}^m
(\zeta_{\pi(2j - 1)} / \zeta_{\pi(2j)} + \zeta_{\pi(2j)} / \zeta_{\pi(2j - 1)})
\prod_{j = 2m + 1}^{\ell + m} (\zeta_{\pi(j)}^{-1} z \bar{w} + \zeta_{\pi(j)} \bar{z} w) \right].
\end{align*}
Expanding into polynomials of $z, \bar{z}, w, \bar{w}$, this expression is a
linear combination with constant coefficients of terms of the form
\[
\sum_{\zeta_1, \ldots, \zeta_{|\pi|} \in X_K} |w|^{2m + 2a} |z|^{2a} w^b \bar{z}^b
\zeta_1^{c_1} \cdots \zeta_{|\pi|}^{c_{|\pi|}}.
\]
Here, the exponents satisfy
$a \geq 0$, $2a+|b|=(\ell+m)-2m=\ell-m$, $\sum_i c_i = b$,
$\sum_i |c_i| \leq \ell+m$, and $|c_i| \leq m+(\ell+m-2m)=\ell$ for each $i$.
By (\ref{eq:sumrootsofunity}),
these terms vanish unless each $c_i$ is a multiple of $K$. In particular, for
$\ell<K$, the condition $|c_i| \leq \ell$ implies that
the only non-zero terms must have $c_1=\ldots=c_{|\pi|}=b=0$. Then
$M_{\ell,m}(\pi \mid \theta,\theta_*)$ is a polynomial in
$|w|^2=\|\theta\|^2$. Since $S_\ell(\theta)$ is a linear combination of such
terms $M_{\ell,m}(\pi \mid \theta,\theta_*)$, this shows (b).

For (c), if $\ell=K$, the only non-zero terms which are not a polynomial of 
$\|\theta\|^2$ must have $b \neq 0$, so that the condition $2a+|b|=K-m$ requires
$m<K$. Then since $\sum_i |c_i|
\leq K+m<2K$, there is some $i^*$ with $c_{i^*} \in \{-K,K\}$
and $c_j = 0$ for all $j \neq i^*$. Such terms can only appear in
$M_{K,m}(\pi \mid \theta,\theta_*)$ when $m=0$ and
$\pi=\{\{1,\ldots,K\}\}$, for which we have
\[M_{K,0}(\{\{1,\ldots,K\}\} \mid \theta, \theta_*) =
\frac{1}{2^K}(z^K\bar{w}^K+\bar{z}^K w^K).\]
Writing $w=\|\theta\|e^{ir}$ and $z=\|\theta_*\|e^{ir_*}$, this is
\[M_{K,0}(\{\{1,\ldots,K\}\} \mid \theta, \theta_*) =
\frac{\|\theta\|^K\|\theta_*\|^K}{2^K}(e^{i(r-r_*)K}+e^{i(r_*-r)K})
=\frac{\|\theta\|^K\|\theta_*\|^K}{2^{K-1}}\cos(Kt(\theta)).\]
Substituting into (\ref{eq:Sl}) and recalling that the remaining terms are
polynomial in $\|\theta\|^2$ shows (c).
\end{proof}

\begin{proof}[Proof of Theorem \ref{thm:rotations}]
For each point $\wtheta \in \R^d$, we consider a local
reparametrization by $\varphi$ in a neighborhood $U_{\wtheta}$ of $\wtheta$.
At $\wtheta=0$, we take the reparametrization to be $\varphi=\theta$. At each
$\wtheta \neq 0$, we take it to be $\varphi=(\|\theta\|,t(\theta))$. 
We then apply Lemmas \ref{lemma:largesigmadescent} and
\ref{lemma:largesigmapseudomin} on $U_{\wtheta}$.

For $\wtheta=0$, observe that
$\nabla_\theta^2 S_2(\wtheta)=-\frac{1}{2}\|\theta_*\|^2\Id \prec 0$. For
$\wtheta \neq 0$ where $\|\wtheta\| \neq \|\theta_*\|$, set
$\wvarphi=\varphi(\wtheta)$.
Observe that $S_2(\varphi)=\varphi_1^4/8-\varphi_1^2\varphi_{1,*}^2/4$,
so $\nabla_{\varphi_1} S_2(\wvarphi)=\frac{1}{2}\wvarphi_1(\wvarphi_1^2-\wvarphi_{1,*}^2) \neq 0$. In both cases, Lemma \ref{lemma:largesigmadescent} implies that
$\|\nabla_\theta R(\theta)\| \geq c\sigma^{-4}$ or
$\lambda_{\min}(\nabla_\theta^2 R(\theta)) \leq
-c\sigma^{-4}$, for some $c,\sigma_0>0$ and all $\sigma>\sigma_0$ and
$\theta \in U_{\wtheta}$.

For $\wtheta \notin \O_{\theta_*}$ where $\|\wtheta\|=\|\theta_*\|$,
observe that $S_1,\ldots,S_{K-1}$ depend only on $\varphi_1$. For $S_K$,
applying $\wvarphi_1=\varphi_{1,*}$, we have
\begin{equation}\label{eq:gradSKrotations}
\nabla_{\varphi_2}
S_K(\wvarphi)=\frac{1}{2^{K-1}(K-1)!}\varphi_{1,*}^{2K}\sin(K\wvarphi_2),
\qquad \nabla_{\varphi_2}^2
S_K(\wvarphi)=\frac{K}{2^{K-1}(K-1)!}\varphi_{1,*}^{2K}\cos(K\wvarphi_2).
\end{equation}
Then either
$\nabla_{\varphi_2} S_K(\wvarphi) \neq 0$ (when $\wvarphi_2 \notin
\{j\pi/K:j=0,1,\ldots,2K-1\}$), or
$\lambda_{\min}(\nabla_{\varphi_2}^2 S_K(\wvarphi))<0$ (when $\wvarphi_2 \in
\{j\pi/K:j=1,3,5,\ldots,2K-1\}$). So
Lemma \ref{lemma:largesigmadescent} implies that
$\|\nabla_\theta R(\theta)\| \geq c\sigma^{-2K}$ or
$\lambda_{\min}(\nabla_\theta^2 R(\theta)) \leq
-c\sigma^{-2K}$ for all $\sigma>\sigma_0$ and $\theta \in U_{\wtheta}$.

Finally, for $\wtheta \in \O_{\theta_*}$, (\ref{eq:gradSKrotations})
verifies that $\wvarphi=\varphi(\wtheta)$ is a pseudo-local-minimizer in the
parametrization by $\varphi$. Then
Lemma \ref{lemma:largesigmapseudomin} implies
$R(\theta)$ has a unique local minimizer in $U_{\wtheta}$ and
$\lambda_{\min}(\nabla_\varphi^2 R(\varphi))
\geq c\sigma^{-2K}$ for all $\varphi \in \varphi(U_{\wtheta})$ and
$\sigma>\sigma_0$. This unique local minimizer must be $\wtheta$ itself, since
$\wtheta$ is a global minimizer of $R(\theta)$.

The constants $c,\sigma_0>0$ above depend on $\wtheta$. By compactness, for
any $M>0$, there is a finite collection of points
$\wtheta$ where the neighborhoods $U_{\wtheta}$ cover
$\{\theta \in \R^2:\|\theta\| \leq M\}$, and the above statements then
hold for uniform choices of $c,\sigma_0>0$ in their union. For a sufficiently
small constant $\rho>0$, this establishes all
claims of the theorem for points $\theta \in \R^d$ where $\|\theta\| \leq M$,
and the result for $\|\theta\|>M$ follows from Lemma
\ref{lemma:localizationstrong}.
\end{proof}

The following then shows that with high probability for
$n \gg \sigma^{2K}$,
the empirical risk $R_n(\theta)$ is also globally
benign and satisfies the same properties. This is a special case of
the guarantee for $R_n(\theta)$ in Theorem \ref{thm:globalbenign} to follow.

\begin{corollary}\label{cor:rotations}
For some $(\theta_*,K)$-dependent constants $C,c>0$,
the statements of Theorem \ref{thm:rotations} hold also for $R_n(\theta)$,
with probability at least
$1-e^{-c(\log n)^2}-Ce^{-cn^{1/(2K)}\sigma^{-1}}$.
\end{corollary}

\subsubsection{All permutations in $\R^d$}\label{sec:permutations}
Consider any dimension $d \geq 1$, and let $G \cong S_d$ be the symmetric group
of all permutations of coordinates in $\R^d$. Here, the size of the group
is $K=d!$. Define the symmetric power sums in $\theta$ by
\[p_k(\theta)=\frac{1}{d} \sum_{j=1}^d \theta_j^k,\]
and (for fixed $\theta_* \in \R^d$) the Vandermonde varieties by
\begin{equation}\label{eq:vandermondevariety}
\cV_k=\Big\{\theta \in \R^d:
p_\ell(\theta)=p_\ell(\theta_*) \text{ for all } \ell=1,\ldots,k\Big\}.
\end{equation}
Note that the map $\theta \mapsto (p_1(\theta),\ldots,p_d(\theta))$ is injective
on $\{\theta \in \R^d:\theta_1 \leq \ldots \leq \theta_d\}$
(see \cite[Corollary 1.2]{kostov1989geometric}), so $\cV_d=\O_{\theta_*}$.

\begin{theorem} \label{thm:permutations}
Let $G \cong S_d$ be the symmetric group acting on $\R^d$ by permutation of
coordinates. For generic $\theta_* \in \R^d$,
there exists a $(\theta_*,d)$-dependent
constant $\sigma_0>0$ such that the global landscape of $R(\theta)$ is benign
for all $\sigma>\sigma_0$. More quantitatively, for small enough $\rho > 0$
there are $(\theta_*,d)$-dependent constants $c, \sigma_0>0$ such that when $\sigma>\sigma_0$,
\begin{enumerate}[(a)]
\item For each $\wtheta \in \O_{\theta_*}$, reparametrizing by the symmetric power sums
$\varphi=(p_1,\ldots,p_d)$ in $B_\rho(\wtheta)$, we have the strong convexity
$\lambda_{\min}(\nabla_\varphi^2 R(\varphi)) \geq c\sigma^{-2d}$ for all
$\varphi \in \varphi(B_\rho(\wtheta))$.
\item Denote $\cV_\ell^\rho=\{\theta \in
\R^d:\dist(\theta,\cV_\ell)<\rho\}$, where $\cV_0^\rho=\R^d$.
Then for each $\ell=1,\ldots,d$ and each
$\theta \in \cV_{\ell-1}^\rho \setminus \cV_\ell^\rho$,
either $\|\nabla_\theta R(\theta)\| \geq c\sigma^{-2\ell}$ or
$\lambda_{\min}(\nabla_\theta^2 R(\theta)) \leq -c\sigma^{-2\ell}$.
\end{enumerate}
\end{theorem}

The proof rests on the following lemma, which characterizes the functions
$S_\ell(\theta)$ in this example.

\begin{lemma} \label{lem:perm-risk}
Let $G \cong S_d$ be the symmetric group acting on $\R^d$ by permutation of
coordinates. For each $\ell=1,\ldots,d$,
some constant $a_\ell>0$, and some polynomials $q_\ell,r_\ell:\R^{\ell-1} \to \R$
with coefficients depending on $\theta_*$ and such that
\[q_\ell\Big(p_1(\theta_*),\ldots,p_{\ell-1}(\theta_*)\Big)=0,\]
we have
\begin{align}
S_\ell(\theta) &= a_\ell\big(p_\ell(\theta)^2 - p_{\ell}(\theta_*)\big)^2
+q_\ell\big(p_1(\theta), \ldots, p_{\ell - 1}(\theta)\big)
\cdot p_\ell(\theta)
+r_\ell\big(p_1(\theta), \ldots, p_{\ell - 1}(\theta)\big).
\label{eq:Sellpermutation}
\end{align}
\end{lemma}
\begin{proof}
We apply Lemma \ref{lemma:Selldegell} and the fact that the symmetric power
sums $p_1(\theta),p_2(\theta),\ldots$ generate $\cR^G$ as an algebra over $\R$
(see \cite[Eq.\ (2.12)]{macdonald2015symmetric}).
Thus, any polynomial $\varphi \in \cR^G_{\leq \ell}$ may be written as
\[\varphi(\theta)=c_\varphi p_\ell(\theta)+q_\varphi(p_1(\theta),\ldots,p_{\ell-1}(\theta))\]
for some $c_\varphi \in \R$ and some polynomial $q_\varphi$ with real
coefficients. In particular, applying this to each entry of the moment
tensor $T_\ell(\theta)$ in Lemma \ref{lemma:Selldegell}, we
obtain the form (\ref{eq:Sellpermutation}) where
\[a_\ell=\sum_{\varphi} \frac{c_\varphi^2}{2(\ell!)}\]
and
\[q_\ell\big(p_1(\theta),\ldots,p_{\ell-1}(\theta)\big)=
\sum_{\varphi} \frac{c_\varphi}{\ell!} \cdot
\Big(q_\varphi\big(p_1(\theta),\ldots,p_{\ell-1}(\theta)\big)
-q_\varphi\big(p_1(\theta_*),\ldots,p_{\ell-1}(\theta_*)\big)\Big),\]
with both summations taken over all entries of $T_\ell(\theta)$.
We have $a_\ell>0$ strictly because the diagonal entries of $T_\ell(\theta)$ are
given by
\[T_\ell(\theta)_{i,\ldots,i}=\frac{1}{d!}\sum_{\sigma \in S_d}
\theta_{\sigma(i)}^\ell=\frac{(d-1)!}{d!}\sum_{i=1}^d \theta_i^\ell=p_\ell(\theta),\]
so that $c_{\varphi}=1$ for these entries.
\end{proof}

The derivative of this map $\varphi=(p_1,\ldots,p_d)$ is
singular at points $\wtheta$ having repeated entries. To analyze the
landscape of $R(\theta)$ near such points,
we use the following known (and non-trivial) facts about the symmetric power sums
and Vandermonde varieties.

\begin{lemma}\label{lemma:permutationsaddle}
Let $\cV_k$ be the Vandermonde variety (\ref{eq:vandermondevariety}), with
$\cV_0=\R^d$. For each $k \in \{1,\ldots,d\}$ and any generic $\theta_* \in \R^d$,
\begin{enumerate}[(a)]
\item Each point $\theta \in \cV_k$ has at least $k$ distinct entries.
\item $\cV_{k-1}$ is a nonsingular algebraic variety, and
$p_k(\theta)$ is a Morse function on $\cV_{k-1}$.
\item The critical points of the restriction $p_k|_{\cV_{k-1}}$ are the
points $\theta \in \cV_{k-1}$ having exactly $k-1$ distinct entries.
\item If $\theta$ is a local minimizer or local maximizer of
$p_k|_{\cV_{k-1}}$, then it is also a global minimizer or global
maximizer of $p_k|_{\cV_{k-1}}$.
\end{enumerate}
\end{lemma}

\begin{proof}
For (a), fixing any integer multiplicities $d_1,\ldots,d_{k-1} \geq 0$ summing
to $d$, the image of the polynomial function $F:\R^{k-1} \to \R^k$ given by
\[F(x_1,\ldots,x_{k-1})=\left(\frac{1}{d}\sum_{j=1}^{k-1}
d_jx_j^\ell:\ell=1,\ldots,k\right)\]
is a constructible set in the Zariski topology on $\R^k$,
by Chevalley's theorem (see \cite[Theorem 3.16]{harris1992algebraic}).
By \cite[Theorem 11.12]{harris1992algebraic},
the Zariski closure of this image has dimension at most $k - 1$,
so its complement is generic.
Taking the intersection of these complements over the finitely many
choices of $d_1, \ldots, d_{k - 1}$, we find that the complement of the set
\[\Big\{(p_1(\theta),\ldots,p_k(\theta)):\theta \text{ has at most } k-1 \text{
distinct coordinates}\Big\}\]
is also generic in $\R^k$. We conclude that for generic $\theta_* \in \R^d$, the
point $(p_1(\theta_*),\ldots,p_k(\theta_*))$ does not belong to the above set,
meaning that each point $\theta \in \cV_k$ has at least $k$ distinct coordinates.

For (b), the gradient of $p_\ell$ is given by
\[\nabla p_\ell(\theta)=\frac{\ell}{d}(\theta_1^{\ell-1},\ldots,
\theta_d^{\ell-1}).\]
Thus, if $\nabla p_1,\ldots,\nabla p_k$ are linearly dependent, then there is a
non-zero polynomial $P$ of degree at most $k-1$ for which $P(\theta_i)=0$ for
every $i=1,\ldots,d$. Since $P$ has at most $k-1$ real roots, this implies that
$\theta$ has at most $k-1$ distinct coordinates. Applying (a), this shows that
for generic $\theta_*$, the vectors
$\nabla p_1,\ldots,\nabla p_k$ are linearly independent at every
$\theta \in \cV_k(\theta_*)$, so $\cV_k(\theta_*)$ is nonsingular. The remaining
two statements then follow from the results
of \cite[Theorems 5, 6, and 7]{arnold1986hyperbolic}; see
also \cite{kostov1989geometric}.
\end{proof}

\begin{proof}[Proof of Theorem \ref{thm:permutations}]
For each $\wtheta \in \R^d$, we consider a local reparametrization by
$\varphi$ in a neighborhood $U_{\wtheta}$. If $k$ is the number of distinct
entries of $\wtheta$, then we take the first $k$ functions in $\varphi$ to
be the symmetric power sums $p_1(\theta),\ldots,p_k(\theta)$. As shown in the
proof of Lemma \ref{lemma:permutationsaddle} above, the gradients $\nabla
p_1,\ldots,\nabla p_k$ must be linearly independent at $\wtheta$. We
arbitrarily pick $d-k$ remaining functions to 
complete $(p_1,\ldots,p_k)$ into the local reparametrization $\varphi$.
Denote $\wvarphi=\varphi(\wtheta)$ and $\varphi_*=\varphi(\theta_*)$. 

We apply Lemmas \ref{lemma:largesigmadescent} and
\ref{lemma:largesigmapseudomin} on each neighborhood $U_{\wtheta}$. Fix $\ell \in
\{1,\ldots,d\}$ and consider
$\wtheta \in \cV_{\ell-1} \setminus \cV_\ell$.
By Lemma \ref{lemma:permutationsaddle}(a), $\wtheta$ has at least $\ell-1$ distinct
coordinates, so the first $\ell-1$ coordinates of $\varphi$ are
$(\varphi_1,\ldots,\varphi_{\ell-1})=(p_1,\ldots,p_{\ell-1})$.
Denote $\varphi^\ell=(\varphi_\ell,\ldots,\varphi_d)$, 
and note that $S_1,\ldots,S_{\ell-1}$ are functions
only of $\varphi_1,\ldots,\varphi_{\ell-1}$. Furthermore, recalling
(\ref{eq:Sellpermutation}) and applying
\[q_\ell(\wvarphi_1,\ldots,\wvarphi_{\ell-1})=
q_\ell(\varphi_{1,*},\ldots,\varphi_{\ell-1,*})=0\]
and the chain rule,
\begin{align*}
\nabla_{\varphi^\ell} S_\ell(\wvarphi)&=2a_\ell\Big(p_\ell(\wvarphi)
-p_\ell(\varphi_*)\Big) \nabla_{\varphi^\ell} p_\ell(\wvarphi),\\
\nabla_{\varphi^\ell}^2 S_\ell(\wvarphi)&=2a_\ell\Big(\nabla_{\varphi^\ell}
p_\ell(\wvarphi)\nabla_{\varphi^\ell} p_\ell(\wvarphi)^\top
+(p_\ell(\wvarphi)-p_\ell(\varphi_*))\cdot \nabla_{\varphi^\ell}^2 p_\ell(\wvarphi)\Big).
\end{align*}
Since $\wtheta \notin \cV_\ell$, we have
$p_\ell(\wvarphi) \neq p_\ell(\varphi_*)$. Then either
$\nabla_{\varphi^\ell} S_\ell(\wvarphi) \neq 0$, or
\[\nabla_{\varphi^\ell} p_\ell(\wvarphi)=0 \qquad \text{ and } \qquad
\nabla_{\varphi^\ell}^2 S_\ell(\wvarphi)=2a_\ell
(p_\ell(\wvarphi)-p_\ell(\varphi_*)) \cdot \nabla_{\varphi^\ell}^2 p_\ell(\wvarphi).\]
In this latter case, note that $\varphi^\ell$ is a local chart for
$\cV_{\ell-1}$ around $\wvarphi$, so $\wvarphi$ is a critical point of
$p_\ell|_{\cV_{\ell-1}}$. The Morse condition of Lemma
\ref{lemma:permutationsaddle}(b) implies that all eigenvalues of
$\nabla_{\varphi^\ell}^2 p_\ell(\wvarphi)$ are non-zero. If
$\nabla_{\varphi^\ell}^2
p_\ell(\wvarphi)$ has both positive and negative eigenvalues, then this
guarantees that
$\lambda_{\min}(\nabla_{\varphi^\ell}^2 S_\ell(\wvarphi))<0$. Otherwise, $\wvarphi$ is
a local minimizer or local maximizer of $p_\ell|_{\cV_{\ell-1}}$.
If it is a local minimizer, then all eigenvalues of $\nabla_{\varphi^\ell}^2
p_\ell(\wvarphi)$ are positive. Lemma \ref{lemma:permutationsaddle}(d) also
implies that $p_\ell(\wvarphi)<p_\ell(\varphi_*)$, so all eigenvalues of
$\nabla_{\varphi^\ell}^2 S_\ell(\wvarphi)$ are negative. The case where $\wvarphi$ is
a local maximizer of $p_\ell|_{\cV_{\ell-1}}$ is similar.
Combining these observations and applying Lemma
\ref{lemma:largesigmadescent}, we get that either $\|\nabla_\theta
R(\theta)\| \geq c\sigma^{-2\ell}$ or $\lambda_{\min}(\nabla_\theta^2 R(\theta))
\leq -c\sigma^{-2\ell}$ for all $\theta \in U_{\wtheta}$ and $\sigma>\sigma_0$.

For $\wtheta \in \cV_d=\O_{\theta_*}$, we have
$\varphi=(p_1,\ldots,p_d)$, so that each $S_\ell$ depends only on
$(\varphi_1,\ldots,\varphi_\ell)$ and
\[\nabla_{\varphi_\ell}
S_\ell(\wvarphi)=2a_\ell(\wvarphi_\ell-\varphi_{\ell,*})=0,
\qquad \nabla_{\varphi_\ell}^2
S_\ell(\wvarphi)=2a_\ell>0.\]
Thus $\wtheta$ is a pseudo-local-minimizer in the reparametrization by
$\varphi$. Lemma \ref{lemma:largesigmapseudomin} implies that $\wtheta$
is the unique critical point of $R(\theta)$ in $U_{\wtheta}$, and that
$\lambda_{\min}(\nabla_\varphi^2 R(\varphi)) \geq c\sigma^{-2d}$ for all
$\varphi \in \varphi(U_{\wtheta})$ and $\sigma>\sigma_0$.

Fixing any $M>0$ and taking a finite collection of these sets $U_{\wtheta}$
which cover the compact set $\{\theta \in \R^d:\|\theta\| \leq M\}$, the
above results hold for uniform choices of constants $c,\sigma_0>0$ in their union.
Then for a sufficiently small constant $\rho>0$, the claims of the theorem hold 
for all $\theta \in \R^d$ with $\|\theta\| \leq M$, and the result for
$\|\theta\|>M$ follows again from Lemma \ref{lemma:localizationstrong}.
\end{proof}

The following then shows that with high probability for
any $d \geq 2$ and $n \gg \sigma^{2d}$,
the empirical risk $R_n(\theta)$ is also globally benign and satisfies the same
properties. Again, this is a special case of
the guarantee for $R_n(\theta)$ in Theorem \ref{thm:globalbenign} to follow.

\begin{corollary}\label{cor:permutations}
For some $(\theta_*,d)$-dependent constants $C,c>0$,
the statements of Theorem \ref{thm:permutations} hold also for $R_n(\theta)$,
with probability at least $1-e^{-c(\log n)^2}-Ce^{-cn^{1/(2d \vee 4)}
\sigma^{-1}}$.
\end{corollary}

\subsubsection{General groups}
We provide a general condition under which the landscape of $R(\theta)$ is
globally benign for high noise, which captures the structure of the previous two
examples.

Let $M_\ell:\R^d \to \R^{d+d^2+\ldots+d^\ell}$ be the combined
vectorized moment map
\[M_\ell(\theta)=\Big(T_1(\theta),\ldots,T_\ell(\theta)\Big).\]
For fixed $\theta_* \in \R^d$, recall
$P_\ell(\theta)=\|T_\ell(\theta)-T_\ell(\theta_*)\|_\HS^2$
from Lemma \ref{lemma:Selldegell}, and define the moment varieties
\[\cV_\ell=\Big\{\theta \in \R^d:M_\ell(\theta)=M_\ell(\theta_*)\Big\},
\qquad \cV_0=\R^d.\]
We denote by $P_\ell|_{\cV_{\ell-1}}$ the restriction of the function $P_\ell$ to
$\cV_{\ell-1}$. We will assume that each $\cV_\ell$ is nonsingular and has the
same dimension $\bar{d}_\ell$ at every point. We then denote by
$\nabla P_\ell|_{\cV_{\ell-1}} \in \R^{\bar{d}_\ell}$ and
$\nabla^2 P_\ell|_{\cV_{\ell-1}} \in \R^{\bar{d}_{\ell} \times \bar{d}_\ell}$ 
the gradient and Hessian of the restriction $P_\ell|_{\cV_{\ell-1}}$ with
respect to any choice of local chart on $\cV_{\ell-1}$. Note that the conditions
below do not depend on the specific choice of chart.

\begin{theorem}\label{thm:globalbenign}
Let $\theta_* \in \R^d$ be generic, and let $L$ be the constant in Lemma
\ref{lemma:phiconstruction}. Suppose that
\[\cV_L=\O_{\theta_*}\]
and that for every $\ell \geq 1$, $\der_\theta M_\ell$ has constant rank on $\cV_\ell$.
Suppose also, for each $\ell=1,\ldots,L$ and
each $\theta \in \cV_{\ell-1}$, that either (1)
$\nabla P_\ell|_{\cV_{\ell-1}}(\theta) \neq 0$, (2)
$\lambda_{\min}(\nabla^2 P_\ell|_{\cV_{\ell-1}}(\theta))<0$, or (3)
$\theta \in \cV_\ell$.
Then there exists a $(\theta_*,d,G)$-dependent constant $\sigma_0>0$ such that 
the landscape of $R(\theta)$ is globally benign for all $\sigma>\sigma_0$.

More quantitatively, for small enough $\rho > 0$, there exist $(\theta_*,d,G)$-dependent
constants $c, \sigma_0 >0$ such that when $\sigma>\sigma_0$,
\begin{enumerate}[(a)]
\item For each $\wtheta \in \O_{\theta_*}$, there is a
local reparametrization $\varphi:B_\rho(\wtheta) \to \R^d$ such that
$\lambda_{\min}(\nabla_\varphi^2 R(\varphi)) \geq c\sigma^{-2L}$ for all $\varphi
\in \varphi(B_\rho(\wtheta))$.
\item Denote $\cV_\ell^\rho=\{\theta \in \R^d:\dist(\theta,\cV_\ell)<\rho\}$,
where $\cV_0^\rho=\R^d$.
Then for each $\ell=1,\ldots,L$ and each $\theta \in \cV_{\ell-1}^\rho \setminus
\cV_\ell^\rho$, either $\|\nabla R(\theta)\| \geq c\sigma^{-2\ell}$ or
$\lambda_{\min}(\nabla^2 R(\theta)) \leq -c\sigma^{-2\ell}$.
\end{enumerate}
With probability at least 
$1-e^{-c(\log n)^2}-Ce^{-cn^{1/(2L \vee 4)}\sigma^{-1}}$,
the same statements hold for the empirical risk $R_n(\theta)$.
\end{theorem}

\begin{proof}
For generic $\theta_* \in \R^d$, Lemma \ref{lemma:polynomialreparam} implies that
$\der_\theta M_\ell$ has rank $d_1+\ldots+d_\ell$ at $\theta_*$. Then
by the given assumption that $\der_\theta M_\ell$ has
constant rank over $\cV_\ell$, this rank must be $d_1+\ldots+d_\ell$, and
$\cV_\ell$ is a manifold of dimension $\bar{d}_\ell=d-(d_1+\ldots+d_\ell)$.

Note that for any $\ell \geq 2$, $\cV_\ell \subseteq
\{\theta:\|\theta\|^2=\|\theta_*\|^2\}$. Hence for large enough
$M>0$ and small enough $\rho>0$ (depending on $\theta_*$),
if $\|\theta\|>M$, we have either $\theta \in \cV_0^\rho \setminus \cV_1^\rho$
or $\theta \in \cV_1^\rho \setminus \cV_2^\rho$. Lemma
\ref{lemma:localizationstrong} shows that with probability $1-e^{-c(\log
n)^2}-Ce^{-cn^{1/4}\sigma^{-1}}$, we have $\|\nabla R(\theta)\|,\|\nabla
R_n(\theta)\| \geq c\sigma^{-2}$ for all $\theta \in \cV_0^\rho 
\setminus \cV_1^\rho$ and $\|\nabla R(\theta)\|,\|\nabla R_n(\theta)\|
\geq c\sigma^{-4}$ for all $\theta \in \cV_1^\rho \setminus \cV_2^\rho$.

It remains to consider the points $\{\theta:\|\theta\| \leq M\}$. We will apply
a compactness argument to take a
finite cover by neighborhoods of points $\wtheta \in \R^d$. We consider two
cases for such a point $\wtheta$:

{\bf Case I.} Suppose $\wtheta \notin \O_{\theta_*}$.
Then there must exist $\ell \in \{1,\ldots,L\}$ where $\wtheta \in
\cV_0,\ldots,\cV_{\ell-1}$ and $\wtheta \notin \cV_\ell$. For each
$k=1,\ldots,\ell-1$, since $\der_\theta M_k$ has rank $d_1+\ldots+d_k$, we may
pick $d_k$ coordinates $\varphi^k$ of the moment tensor $T_k$ such that
$(\varphi^1,\ldots,\varphi^{\ell-1})$ have linearly independent
gradients at $\wtheta$. Let us complete the parametrization by
$d-(d_1+\ldots+d_{\ell-1})$ additional coordinates $\varphi^\ell$, so that
$\varphi=(\varphi^1,\ldots,\varphi^\ell)$ has non-singular
derivative at $\wtheta$. Then for some neighborhood $U_{\wtheta}$ of $\wtheta$,
$\varphi$ forms a local reparametrization on $U_{\wtheta}$, and Lemma
\ref{lemma:polynomialreparam}(c) ensures that each polynomial $\psi \in
\cR^G_{\leq \ell-1}$ is a function only of $(\varphi^1,\ldots,\varphi^{\ell-1})$
in this reparametrization. In particular, the manifold $\cV_{\ell-1}$ is defined
by
$\varphi^1(\theta)=\varphi^1(\theta_*),\ldots,\varphi^{\ell-1}(\theta)=\varphi^{\ell-1}(\theta_*)$
on $U_{\wtheta}$, so that the remaining coordinates $\varphi^\ell$ form a local
chart for $\cV_{\ell-1}$. By Lemma \ref{lemma:Selldegell}, $S_1,\ldots,S_{\ell-1}$
are functions only of $(\varphi^1,\ldots,\varphi^{\ell-1})$, and
\[\nabla_{\varphi^\ell} S_\ell(\varphi)=\frac{1}{2(\ell!)}
\nabla_{\varphi^\ell} P_\ell(\varphi),
\qquad \nabla_{\varphi^\ell}^2 S_\ell(\varphi)=\frac{1}{2(\ell!)}\nabla_{\varphi^\ell}^2 P_\ell(\varphi).\]
Since $\wtheta \notin \cV_\ell$, the given condition in the lemma implies
that either $\nabla_{\varphi^\ell} S_\ell(\varphi) \neq 0$
or $\lambda_{\min}(\nabla_{\varphi^\ell}^2 S_\ell(\varphi))<0$.
Then by Lemma \ref{lemma:largesigmadescent}, for $\sigma>\sigma_0$ and large
enough $\sigma_0$, there is a neighborhood $U_{\wtheta}$ of $\wtheta$ on which
either $\|\nabla R(\theta)\| \geq c\sigma^{-2\ell}$ or
$\lambda_{\min} (\nabla^2 R(\theta)) \leq -c\sigma^{-2\ell}$.

For the empirical risk $R_n$, the argument is similar to that
of Corollary \ref{cor:locallargenoise}: We may assume $n \geq \sigma^{2\ell}$,
as otherwise the desired probability guarantee is vacuous.
Observe that since $S_1,\ldots,S_{\ell-1}$ do not depend on $\varphi^\ell$, the
above and Lemma \ref{lemma:seriesexpansion} show for $\varphi \in
\varphi(U_{\wtheta})$ that
\begin{equation}\label{eq:populationdescent}
\|\nabla_{\varphi^\ell} R^\ell(\varphi)\| \geq c\sigma^{-2\ell} \quad \text{ or }
\quad
\lambda_{\min} (\nabla_{\varphi^\ell}^2 R^\ell(\varphi)) \leq
-c\sigma^{-2\ell}.\end{equation}
Let $\cE$ be the event where the guarantee of Lemma \ref{lemma:Rnexpansion}
holds with $k=2L$, $t=\sigma^{-2L}/\log \sigma$, and $r(\sigma)$ a large
enough constant, and also where
\[\bigg|\frac{1}{n}\sum_{i=1}^n
A_{k,m}(\eps_i,\theta_*)-\E_\eps[A_{k,m}(\eps,\theta_*)]\bigg|
\leq c_0\sigma^{-(k \wedge L)}\]
for each $k=1,\ldots,2L$, $m=1,\ldots,M_k$,
and a sufficiently small constant $c_0>0$. By
Lemma \ref{lemma:Rnexpansion} and (\ref{eq:gaussianpoly}), we have
$\P[\cE] \geq 1-e^{c(\log n)^2}-Ce^{-cn^{1/(2L)}\sigma^{-1}}$.
Since $P_{k,m}$ in Lemma \ref{lemma:Rnexpansion}
does not depend on $\varphi^\ell$ for all
$k \leq \ell-1$, on this event $\cE$, the bounds (\ref{eq:Rnexpansion1})
and (\ref{eq:Rnexpansion2}) together with (\ref{eq:populationdescent}) imply
that
\[\|\nabla_{\varphi^\ell} R_n(\varphi)\| \geq c\sigma^{-2\ell}
\quad \text{ or } \quad \lambda_{\min}(\nabla_{\varphi^\ell}^2 R_n(\varphi))
\leq -c\sigma^{-2\ell}.\]
Then, applying the same argument as in Lemma \ref{lemma:largesigmadescent},
this shows also for the gradient and Hessian in $\theta$ that for all
$\theta \in U_{\wtheta}$ a neighborhood small enough, we have
either $\|\nabla R_n(\theta)\| \geq c\sigma^{-2\ell}$ or
$\lambda_{\min} (\nabla^2 R_n(\theta)) \leq -c\sigma^{-2\ell}$.

{\bf Case II.} Suppose $\wtheta \in \O_{\theta_*}$.
Then Theorem \ref{thm:locallargenoise} and Corollary
\ref{cor:locallargenoise} show
that there is a neighborhood $U_{\wtheta}$ where, parametrizing by the full
transcendence basis $\varphi$ of Lemma \ref{lemma:phiconstruction}, we have
$\nabla_\varphi^2 R(\varphi) \geq c\sigma^{-2L}$ and $\nabla_\varphi^2
R_n(\varphi) \geq c\sigma^{-2L}$ on $\varphi(U_{\wtheta})$, with the desired
probability.

Taking a finite collection of these neighborhoods $U_{\wtheta}$ which cover the
compact set $\{\theta \in \R^d:\|\theta\| \leq M\}$, 
this establishes the claims of the theorem also for $\|\theta\| \leq M$
and some sufficiently small constant $\rho>0$.
\end{proof}

\subsection{Global landscape for cyclic permutations in $\R^d$}\label{sec:MRA}

For the group of cyclic permutations of coordinates in dimension $d$, the
orbit recovery problem is often called multi-reference alignment (MRA). We have
\begin{equation}\label{eq:GMRA}
G=\{\Id,h,h^2,\ldots,h^{d-1}\} \cong \ZZ/d\ZZ
\end{equation}
where the generator
\begin{equation}\label{eq:MRAgenerator}
h=\begin{pmatrix} 0 & 0 & \cdots & 0 & 1 \\
1 & 0 & \cdots & 0 & 0 \\
0 & 1 & \cdots & 0 & 0 \\
\vdots & \vdots & \ddots & \vdots & \vdots \\
0 & 0 & \cdots & 1 & 0 \end{pmatrix} \in \R^{d \times d}
\end{equation}
cyclically rotates coordinates by one position. Here, the size of the group is
$K=d$. Since this is the same as the group of all permutations when $d \in
\{1,2\}$, we consider $d \geq 3$. 

We change to the Fourier basis for $\theta$. Index $\R^d$ and $\C^d$
by $0,1,\ldots,d-1$, and define the $d^\text{th}$ root-of-unity
$\omega=e^{2\pi\ii/d}$. For all $k \in \ZZ$, let
\begin{equation}\label{eq:MRAvk}
v_k(\theta)=\frac{1}{\sqrt{d}}\sum_{j=0}^{d-1} \omega^{jk}\theta_j
\end{equation}
be the coordinates of the normalized Fourier transform of $\theta$.
Note that $v_0(\theta)$ is real, and $v_{d/2}(\theta)$ is also real for even
$d$.

Suppose now that $\theta_* \in \R^d$ is such that $v_{k,*} := v_k(\theta_*) \neq 0$ for all
$k \not\equiv 0 \bmod d$.  Denoting the unit circle by $\cS \cong [0, 2 \pi)$
  and writing $\Arg(z) \in \cS$ for the complex argument of $z \in \C$, we choose
  new coordinates $r_k(\theta)$ and $t_k(\theta) \in \cS$ on $\theta$ given by
\[
r_k(\theta)=|v_k(\theta)|, \qquad t_k(\theta)=\begin{cases}
\Arg\big(v_k(\theta)\big)-\Arg\big(v_{k,*}\big) & \text{ if } v_k(\theta) \neq
0\\
0 & \text{ otherwise}
\end{cases}
\]
The quantities $r_k(\theta)^2$ are known as the \emph{power spectrum} of $\theta$.
Finally, we denote $r_{k, *} := r_k(\theta_*)$. 

Because $\theta \in \R^d$ is real-valued, we have that 
\[v_k(\theta)=\overline{v_{-k}(\theta)}, \qquad r_k(\theta)=r_{-k}(\theta), \qquad
t_k(\theta)=-t_{-k}(\theta),
\]
which means that for 
\[
\cI=\{1,\ldots,\lfloor \tfrac{d-1}{2} \rfloor\},
\]
the quantities $\{t_i(\theta)\}_{i \in \cI}$, $\{r_i(\theta)\}_{i \in \cI}$, and $r_{d/2}(\theta)$
if $d$ is even uniquely specify $\theta$.

We now define two surrogate functions $F^+:\cS^{|\cI|} \to \R$ and $F^-:\cS^{|\cI|} \to \R$ in
these coordinates, making the identification $t_{-i}=-t_i$
for $i \in \cI$ and $t_i \in \cS$:
\begin{align} 
F^\pm(t_1,\ldots,t_{|\cI|})
&=-\Bigg(\frac{1}{6}\mathop{\sum_{i,j,k \in \cI \cup -\cI}}_{i+j+k \equiv 0 \bmod d}
r_{i,*}^2r_{j,*}^2r_{k,*}^2\cos(t_i+t_j+t_k)\nonumber\\
&\hspace{0.5in}\pm\1\{d \text{ is even}\} \cdot \frac{1}{2}
\mathop{\sum_{i,j \in \cI \cup -\cI}}_{i+j \equiv d/2
\bmod d} r_{i,*}^2r_{j,*}^2r_{d/2,*}^2\cos(t_i+t_j)\Bigg).\label{eq:MRAF}
\end{align}
We have $F^+=F^-$ when $d$ is odd, and in this case we will only refer to $F^+$.

For generic $\theta_* \in \R^d$ and $\sigma>\sigma_0 \equiv \sigma_0(\theta_*,d)$,
the following shows that local minimizers of $R(\theta)$ are in correspondence
with local minimizers of these surrogate functions on the manifold $\cS^{|\cI|}$.

\begin{theorem}\label{thm:MRA}
Let $G$ be the cyclic group (\ref{eq:GMRA}) acting on $\R^d$,
where $d \geq 3$. Suppose $\theta_* \in \R^d$ has $v_{k,*} \neq 0$ for all
$k \not\equiv 0 \bmod d$. For small enough $\rho > 0$, there exist
some $(\theta_*,d)$-dependent constants $c,\sigma_0>0$ and all $\sigma>\sigma_0$:
\begin{enumerate}[(a)]
\item For each local minimizer $\wt$ of $F^+(t)$ where
$\lambda_{\min}(\nabla^2 F^+(\wt))>0$, there is a unique local
minimizer of $R(\theta)$ in the ball $B_\rho(\wtheta)$, and a local
reparametrization $\varphi$ such that
$\lambda_{\min}(\nabla_\varphi^2 R(\varphi)) \geq c\sigma^{-6}$ for all
$\varphi \in \varphi(B_\rho(\wtheta))$.
Here, $\wtheta \in \R^d$ is the point where $r_k(\wtheta)=r_{k,*}$ for all
$k \in \ZZ$,
$v_0(\wtheta)=v_{0,*}$, $v_{d/2}(\wtheta)=v_{d/2,*}$ if $d$ is even, and
$\Arg(v_k(\wtheta))=\Arg(v_{k,*})+\wt_k$ for each $k \in \cI$.
\item If $d$ is even, then in addition, for each local minimizer
$\wt \in \cS^{|\cI|}$ of $F^-(t)$ where
$\lambda_{\min}(\nabla^2 F^-(\wt))>0$, the same statement of (a) holds over
$B_\rho(\wtheta)$ for $\wtheta \in \R^d$ defined by the same
conditions as in (a), except with
$v_{d/2}(\wtheta)=-v_{d/2,*}$ in place of
$v_{d/2}(\wtheta)=v_{d/2,*}$.
\item If $F^+(t)$ and $F^-(t)$ are Morse on
$\cS^{|\cI|}$, then (a) and (b) characterize all of the local minimizers
of $R(\theta)$. For each $\theta \in \R^d$ outside the union of the
balls $B_\rho(\wtheta)$ in (a) and (b),
either $\|\nabla_\theta R(\theta)\| \geq c\sigma^{-6}$
or $\lambda_{\min}(\nabla_\theta^2 R(\theta)) \leq -c\sigma^{-6}$.
\end{enumerate}
\end{theorem}

The following shows that the same statements then hold for the empirical risk $R_n(\theta)$,
with high probability when $n \gg \sigma^6$. The proof is the same as the
empirical risk analysis in Theorem \ref{thm:globalbenign},
and we omit this for brevity.

\begin{corollary}\label{cor:MRA}
For some $(\theta_*,d)$-dependent constants $C,c>0$,
the statements of Theorem \ref{thm:MRA} hold also for $R_n(\theta)$,
with probability at least $1-e^{-c(\log
n)^2}-Ce^{-cn^{1/6}\sigma^{-1}}$.
\end{corollary}

The following corollary will then follow from an analysis of the landscape
of the functions $F^{\pm}$.

\begin{corollary}\label{cor:MRAbad}
Let $G$ be the cyclic group (\ref{eq:GMRA}) acting on $\R^d$.
\begin{enumerate}[(a)]
\item For $d \leq 5$ and generic $\theta_* \in \R^d$,
there exists a $(\theta_*,d)$-dependent constant $\sigma_0>0$ such that the
landscape of $R(\theta)$ is globally benign for all $\sigma>\sigma_0$.
\item For even $d \geq 6$, there exists an open subset $U \subset \R^d$ and a constant
$\sigma_0>0$ such that for all $\theta_* \in U$ and $\sigma>\sigma_0$,
$R(\theta)$ has a local minimizer not belonging to $\O_{\theta_*}$.
\item For odd $d \geq 53$, there exists an open subset $U \subset \R^d$ and a constant
$\sigma_0>0$ such that for all $\theta_* \in U$ and $\sigma>\sigma_0$,
$R(\theta)$ has a local minimizer not belonging to $\O_{\theta_*}$.
\end{enumerate}
For $(\theta_*,d)$-dependent constants $C,c>0$,
the same statements hold for the empirical risk $R_n(\theta)$ with probability
at least $1-e^{-c(\log n)^2}-Ce^{-cn^{1/6}\sigma^{-1}}$.
\end{corollary}

\begin{figure}
\includegraphics[height=3in]{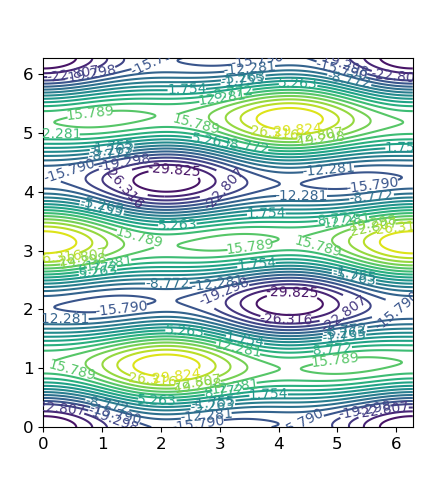}%
\includegraphics[height=3in]{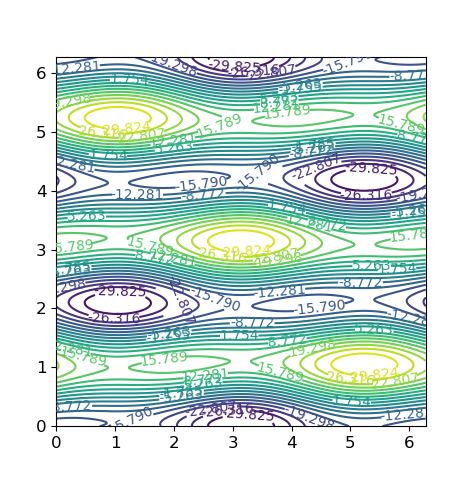}
\caption{Contours of the functions $F^+(t_1,t_2)$ (left) and $F^-(t_1,t_2)$
(right) corresponding to $\theta_*$ in (\ref{eq:MRAexample}), for the group of
cyclic permutations acting in dimension $d=6$.
Each function $F^+$ and $F^-$ is periodic over $t_1,t_2 \in \cS \cong
[0,2\pi)$ and has six local minimizers. Together, these twelve local minimizers
of $F^{\pm}(t_1,t_2)$
correspond to six global minimizers and six spurious local minimizers of
$R(\theta)$ under high noise.}\label{fig:MRAlandscape}
\end{figure}

\begin{figure}
  \begin{center}
    \begin{subfigure}{0.4\textwidth}
      \includegraphics[width=\textwidth]{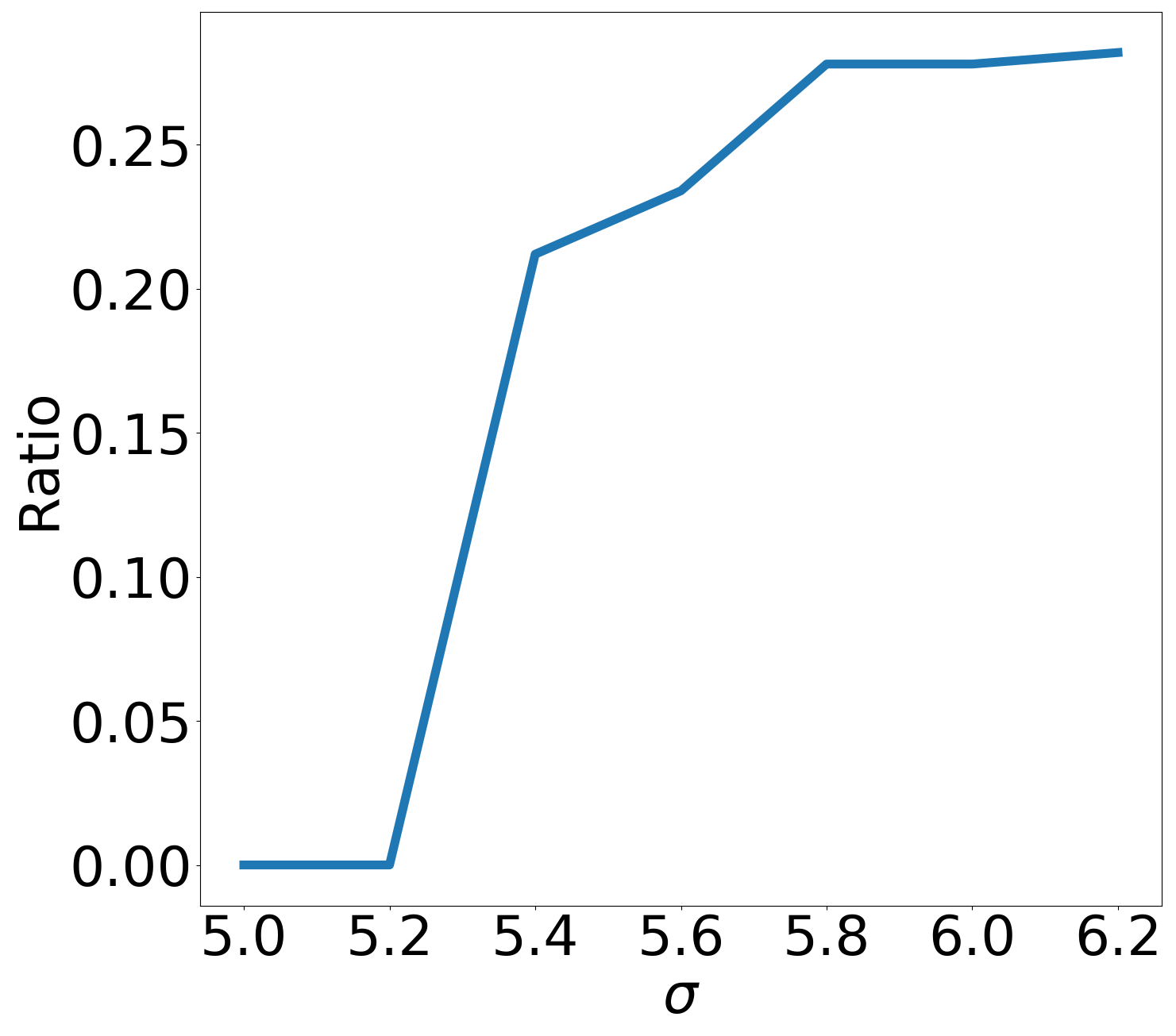}
      \caption{The fraction of AGD runs converging to the spurious
  local minimizers $\O_{\hat{\mu}}$ at different noise levels.}
    \end{subfigure}
\hspace{0.3in}
    \begin{subfigure}{0.52\textwidth}
      \includegraphics[width=\textwidth]{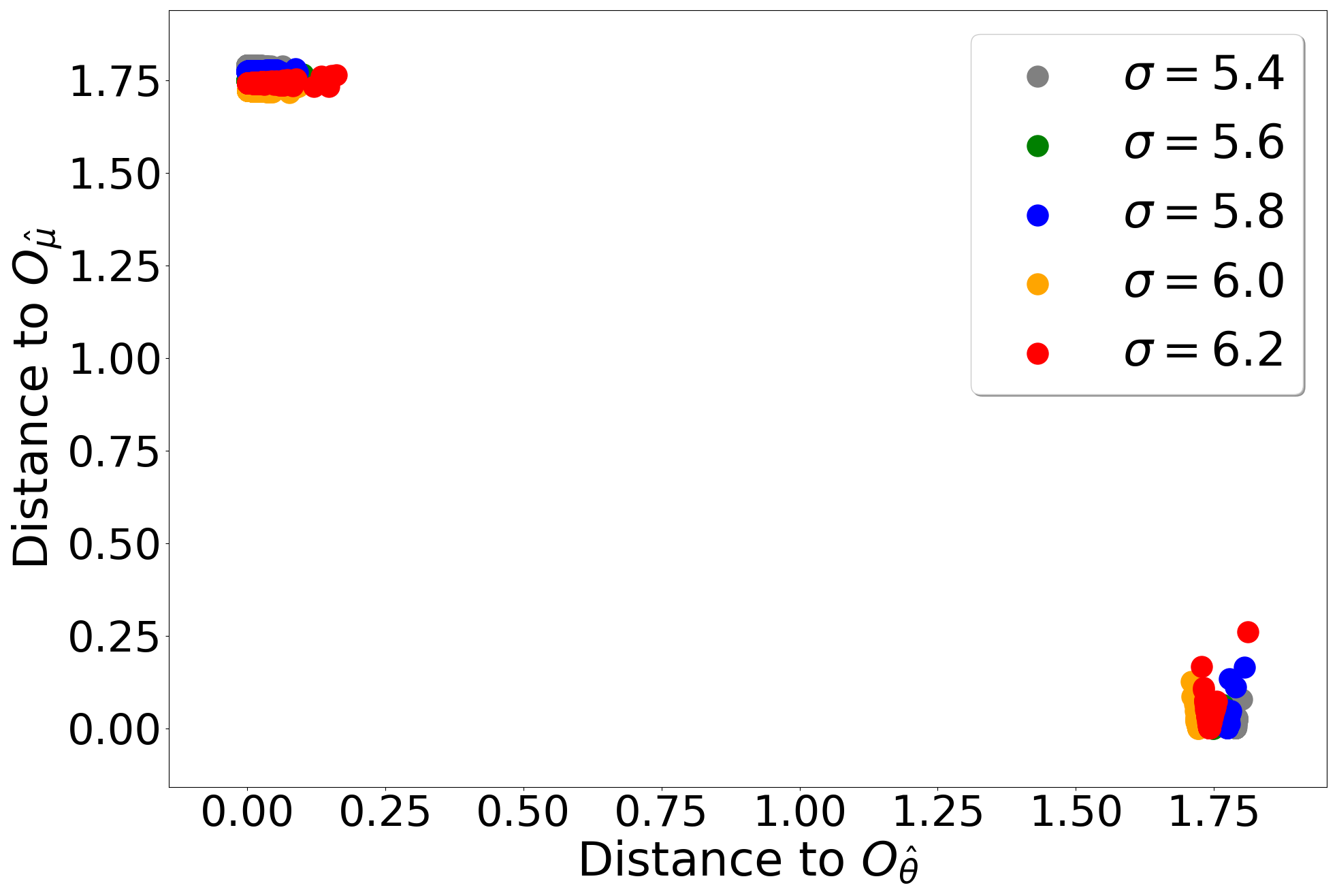}
      \caption{Distances from the $250^\text{th}$ AGD iterate to the
        orbits $\O_{\hat{\theta}}$ and $\O_{\hat{\mu}}$ for each run.}
    \end{subfigure}
\caption{Results of applying Nesterov-accelerated gradient descent (AGD)
  to minimize $R_n(\theta)$ for cyclic permutations in dimension $d=6$,
  with $n=1{,}000{,}000$ samples and $\theta_*$ as in (\ref{eq:MRAexample}).
  AGD was applied from 500 random initializations for noise levels $\sigma$
  between $5.0$ and $6.2$. AGD converges to a point near  $\O_{\hat{\theta}}$ or $\O_{\hat{\mu}}$
  in all cases.  For $\sigma=6.2$, we find
  $\hat{\theta} \approx (2.84,-0.82,-0.85,0.42,-0.79,-0.79)$ and
  $\hat{\mu} \approx (2.08,-0.03,-1.47,1.17,-1.53,-0.21)$, which are close
  to $(\theta_*,\mu_*)$ in (\ref{eq:MRAexample}) and (\ref{eq:MRAexamplelocal}).
}\label{fig:MRAsimulation}
\end{center}
\end{figure}

\begin{remark} \label{rem:min-example}
For $d=6$, setting $(r_{1,*},r_{2,*},r_{3,*})=(1,2,1)$ yields a concrete example
\begin{equation}\label{eq:MRAexample}
\theta_* \approx (2.86,-0.82,-0.82,0.41,-0.82,-0.82)
\end{equation}
belonging to the open set $U$, for which $R(\theta)$ has spurious local minimizers.
Contour maps of $F^+(t_1,t_2)$ and $F^-(t_1,t_2)$ for this point $\theta_*$
are displayed in Figure \ref{fig:MRAlandscape}. It may be verified that
$F^+$ and $F^-$ each has six local minimizers given by
\[(t_1,t_2)=(0,0),(\pi/3,2\pi/3),(2\pi/3,4\pi/3),(\pi,0),(4\pi/3,2\pi/3),
(5\pi/3,4\pi/3).\]
The corresponding twelve points $\wtheta$ constitute the orbits
$\O_{\theta_*}$ and $\O_{\mu_*}$ for a second point
\begin{equation}\label{eq:MRAexamplelocal}
\mu_* \approx (2.04,0.00,-1.63,1.22,-1.63,0.00).
\end{equation}
Theorem \ref{thm:MRA} implies that for large $\sigma$ and large $n$,
the empirical risk $R_n(\theta)$ has (with high probability) twelve
local minimizers, belonging to two orbits $\O_{\hat{\theta}}$ and
$\O_{\hat{\mu}}$ where $\hat{\theta} = \theta_*$ and $\hat{\mu}$
depends on $\sigma$ and lies in a small neighborhood of $\mu_*$.

Simulation results in Figure \ref{fig:MRAsimulation} verify this behavior:
We used the accelerated gradient descent (AGD) method described in
Section \ref{sec:optimization} to minimize $R_n(\theta)$,
with $n=1{,}000{,}000$ samples at various noise levels $\sigma$.
For each noise level, the underlying data $Y_1,\ldots,Y_n$ was fixed, and
simulations were performed with 500 random initializations $\theta^{(0)} \sim
\N(0,\Id)$.
At noise levels $\sigma \leq 5.2$, all simulations converged to the orbit of
a point $\hat{\theta}$ near $\theta_*$, suggesting a benign landscape
for $R_n(\theta)$. For $\sigma \geq 5.4$, a fraction of
simulations converged to the orbit of a second local minimizer $\hat{\mu}$ near
$\mu_*$, and this fraction stabilized to be roughly $28\%$.
This value 28\% may be understood as the ``size'' of the domain of attraction
for the spurious local
minimizers $\O_{\hat{\mu}}$ relative to that for the global minimizers $\O_{\hat{\theta}}$,
for the particular example of $\theta_*$ in (\ref{eq:MRAexample}) and our
simulation parameters.
\end{remark}

The proof of Theorem \ref{thm:MRA} rests on the following lemma, which
describes the first three terms $S_1,S_2,S_3$ of the expansion
(\ref{eq:formalseries}).

\begin{lemma}\label{lemma:MRAS}
Fix $\theta_* \in \R^d$ where $v_{k,*} \neq 0$ for all $k \not\equiv 0 \bmod d$.
Then for some polynomial $q:\R^{d-1} \to \R$ with coefficients depending on
$\theta_*$,
\begin{align}
S_1(\theta)&=-v_{0,*}v_0(\theta)+\frac{1}{2}v_0(\theta)^2\label{eq:MRAS1}\\
S_2(\theta)&=\sum_{i=1}^{d-1}\left(-\frac{1}{2}r_{i,*}^2r_i(\theta)^2
+\frac{1}{4}r_i(\theta)^4\right)\label{eq:MRAS2}\\
S_3(\theta)&=-\frac{1}{6}\mathop{\sum_{i,j,k=1}^{d-1}}_{i+j+k \equiv 0 \bmod d}
r_{i,*}r_{j,*}r_{k,*}r_i(\theta)r_j(\theta)r_k(\theta) \cos\Big(t_i(\theta)
+t_j(\theta)+t_k(\theta)\Big)+q\Big(r_1(\theta)^2,\ldots,r_{d-1}(\theta)^2\Big).
\label{eq:MRAS3}
\end{align}
\end{lemma}
\begin{proof}
Let $e=(1,\ldots,1)/\sqrt{d} \in \R^{d \times 1}$ and let $V \in
\R^{d \times (d-1)}$ complete the orthonormal basis. Then the columns of $V$
span the kernel of $\E_g[g]$, and Lemma \ref{lemma:kerneldecomp} applies with
$G_2=\{V^\top h^k V:k=0,\ldots,d-1\} \subset \rO(d-1)$
for the generator $h$ in (\ref{eq:MRAgenerator}).
Thus, noting that $e^\top \theta=v_0(\theta)$, we have
\[R(\theta)=R^{\Id}(v_0(\theta))+R^{G_2}(V^\top \theta).\]
Applying the series expansion (\ref{eq:formalseries}) to each of $R$, $R^{\Id}$,
and $R^{G_2}$, we have the analogous decomposition
\[S_\ell(\theta)=S^{\Id}_\ell(v_0(\theta))+S^{G_2}_\ell(V^\top \theta)\]
for every $\ell \geq 1$. Note that
\[R^{\Id}(v_0(\theta))=\frac{v_0(\theta)^2}{2\sigma^2}
-\frac{v_{0,*}v_0(\theta)}{\sigma^2}\]
by (\ref{eq:R}), so that
$S_1^{\Id}(v_0(\theta))=-v_{0,*}v_0(\theta)+v_0(\theta)^2/2$,
and $S_\ell^{\Id}(v_0(\theta))=0$ for all $\ell \geq 2$.

To compute the terms $S_\ell^{G_2}(V^\top \theta)$, we apply Lemma
\ref{lemma:Smeanzero}. For $\ell=1$, since $\E_{g\sim\Unif(G_2)}[g]=0$,
 we have
$S_1^{G_2}(V^\top \theta)=0$, so we get (\ref{eq:MRAS1}). For $\ell=2$,
\[S_2^{G_2}(V^\top \theta)=-\frac{1}{2}\E_g[\langle V^\top \theta_*,
(V^\top gV)V^\top \theta \rangle^2]
+\frac{1}{4}\E_g[\langle V^\top \theta,(V^\top gV)V^\top \theta \rangle^2].\]
Introduce $P=VV^\top=\Id-ee^\top$ and the partial Fourier matrix
$F \in \C^{(d-1) \times d}$ such that
$F\theta=(v_1(\theta),\ldots,v_{d-1}(\theta)) \in \C^{d-1}$.
Denote $v=F\theta$, and let $v_*=F\theta_*$. Then note that
\[Ph^kP=F^*D^kF\]
where $D=\diag(\omega,\omega^2,\ldots,\omega^{d-1})$, so
\[S_2^{G_2}(V^\top \theta)=-\frac{1}{2d}\sum_{k=0}^{d-1}
\langle v_*,D^k v\rangle^2+\frac{1}{4d}\sum_{k=0}^{d-1}
\langle v,D^k v\rangle^2.\]
We may write
\[\frac{1}{d}\sum_{k=0}^{d-1} \langle v_*,D^k v\rangle^2
=\frac{1}{d}\sum_{k=0}^{d-1} \left(\sum_{i=1}^{d-1}
\overline{v_{i,*}}\cdot \omega^{ki}v_i\right)^2
=\sum_{i,j=1}^{d-1} \overline{v_{i,*}v_{j,*}}v_iv_j
\left(\frac{1}{d}\sum_{k=0}^{d-1} \omega^{ki+kj}\right).\]
Applying
\begin{equation}\label{eq:avgrootofunity}
\sum_{k=0}^{d-1} w^{jk}=\begin{cases} d & \text{ if } j \equiv 0 \mod d \\
0 & \text{ if } j \not \equiv 0 \mod d,\end{cases}
\end{equation}
and also $v_i=\overline{v_{-i}}$ and $v_{i,*}=\overline{v_{-i,*}}$,
this yields
\[\frac{1}{d}\sum_{k=0}^{d-1} \langle v_*,D^k v\rangle^2=\sum_{i,j=1}^{d-1}
\overline{v_{i,*}v_{j,*}}v_iv_j \cdot \1\{i+j \equiv 0 \bmod d\}
=\sum_{i=1}^{d-1} |v_{i,*}|^2 \cdot |v_i|^2
=\sum_{i=1}^{d-1} r_{i,*}^2r_i(\theta)^2.\]
A similar computation shows $d^{-1}\sum_k \langle v, D^k v
\rangle^2=\sum_{i=1}^{d-1} r_i(\theta)^4$, which yields (\ref{eq:MRAS2}).

For $\ell=3$, applying Lemma \ref{lemma:Smeanzero} and similar arguments,
\begin{align*}
S_3^{G_2}(V^\top \theta)&=\frac{1}{d}\sum_{p=0}^{d-1}
\left(-\frac{\langle v_*,D^p v \rangle^3}{6}
+\frac{\langle v,D^p v \rangle^3}{12}\right)\\
&\hspace{0.5in}+\frac{1}{d^2}\sum_{p,q=0}^{d-1}
\left(\frac{\langle D^p v,D^q v \rangle \langle v_*,D^p v \rangle
\langle v_*,D^q v \rangle}{2}-\frac{\langle D^p v,D^q v \rangle
\langle v,D^p v \rangle\langle v,D^q v \rangle}{3}\right)\\
&=\sum_{i,j,k=1}^{d-1}\Bigg[\left(-\frac{\overline{v_{i,*}v_{j,*}v_{k,*}}
v_iv_jv_k}{6}+\frac{|v_i|^2|v_j|^2|v_k|^2}{12}\right)
\1\{i+j+k \equiv 0 \bmod d\}\\
&\hspace{0.5in}
+\left(\frac{|v_i|^2\overline{v_{j,*}v_{k,*}}v_jv_k}{2}
-\frac{|v_i|^2|v_j|^2|v_k|^2}{3}\right)
\1\{i+k \equiv 0 \bmod d,\,-i+j \equiv 0 \bmod d\}\Bigg].
\end{align*}
Observe that for the second term, we must have $k \equiv -i$ and $j \equiv i$,
in which case $\overline{v_{j,*}v_{k,*}}v_jv_k
=|v_i|^2|v_{i,*}|^2$. Then applying also
$\overline{v_{i,*}}v_i=r_{i,*}r_ie^{\ii t_i}$ (where we write $r_i,t_i$ for
$r_i(\theta),t_i(\theta)$), for some polynomial $q:\R^{d-1} \to \R$ we get
\[S_3^{G_2}(V^\top \theta)
=-\frac{1}{6}\sum_{i,j,k=1}^{d-1}
r_{i,*}r_{j,*}r_{k,*}r_ir_jr_k e^{\ii(t_i+t_j+t_k)}\1\{i+j+k \equiv 0 \bmod d\}
+q(r_1^2,\ldots,r_{d-1}^2).\]
Taking the real part on both sides yields (\ref{eq:MRAS3}).
\end{proof}

\begin{proof}[Proof of Theorem \ref{thm:MRA}]
For each $\wtheta \in \R^d$, we construct a local reparametrization
$\varphi=(\varphi^1,\varphi^2,\varphi^3)$ as follows: Let
$\varphi^1(\theta)=v_0(\theta)$. For each $k \in \cI$, if $v_k(\wtheta)
\neq 0$, then include $r_k(\theta)$ as a coordinate of $\varphi^2$.
If $v_k(\wtheta)=0$, then include $\Re v_k(\theta)$ and $\Im v_k(\theta)$ as two
coordinates of $\varphi^2$. If $d$ is even, include also
$v_{d/2}(\theta)$ as a coordinate of $\varphi^2$. Then for each $k \in \cI$
where $v_k(\wtheta) \neq 0$, include $t_k(\theta)$ as a coordinate of
$\varphi^3$. If there are $m$ coordinates $k \in \cI$ where $v_k(\wtheta) \neq
0$, then $\varphi^3 \in \R^m$ and
$\varphi^2 \in \R^{d-m-1}$. It is easily verified that this defines a 
local reparametrization in some neighborhood $U_{\wtheta}$ around
every $\wtheta \in \R^d$. Note that $S_1$ depends only on $\varphi^1$, and
$S_2$ on $\varphi^1$ and $\varphi^2$.

We now apply Lemmas \ref{lemma:largesigmadescent} and \ref{lemma:largesigmapseudomin}. Let
$\wvarphi=\varphi(\wtheta)$. For $\wtheta \in \R^d$
where $v_0(\wtheta) \neq v_{0,*}$, we have
$\nabla_{\varphi^1} S_1(\wvarphi) \neq 0$. For $\wtheta \in \R^d$ where
$v_k(\wtheta) \neq 0$ and $r_k(\wtheta) \neq r_{k,*}$ for some $k \in \cI$,
we similarly have $\nabla_{\varphi^2} S_2(\wvarphi) \neq 0$, because the
derivative of $S_2$ in the coordinate $r_k$ is non-zero. For $\wtheta \in \R^d$
where $v_k(\wtheta)=0$ for some $k \in \cI$, let us write $r_k(\theta)^2=(\Re
v_k(\theta))^2+(\Im v_k(\theta))^2$ in (\ref{eq:MRAS2}). Differentiating $S_2$
twice in these variables $\Re v_k(\theta)$ and $\Im v_k(\theta)$ and evaluating
at $\Re v_k(\wtheta)=\Im v_k(\wtheta)=0$, we get that the Hessian of $S_2$ in
these variables is $-r_{k,*}^2\Id$. Thus, $\lambda_{\min}(\nabla_{\varphi^2}^2
S_2(\wvarphi))<0$. Finally, for even $d$ and $\wtheta \in \R^d$ where
$v_{d/2}(\wtheta) \notin \{+v_{d/2,*},-v_{d/2,*}\}$, let us write
$r_{d/2}(\theta)^2=v_{d/2}(\wtheta)^2$ in (\ref{eq:MRAS2}). Then either
$v_{d/2}(\wtheta) \neq 0$ and $\nabla_{\varphi^2} S_2(\wvarphi) \neq 0$, or
$v_{d/2}(\wtheta)=0$ and $\lambda_{\min}(\nabla_{\varphi^2} S_2(\wvarphi))<0$.
In all of these cases, Lemma \ref{lemma:largesigmadescent} implies either
$\|\nabla_\theta R(\theta)\| \geq c\sigma^{-4}$ or
$\lambda_{\min}(\nabla_\theta^2 R(\theta)) \leq -c\sigma^{-4}$, for all $\theta
\in U_{\wtheta}$ and $\sigma>\sigma_0$.

It remains to consider those points $\wtheta \in \R^d$ where
$v_0(\wtheta)=v_{0,*}$ and $r_k(\wtheta)=r_{k,*} \neq 0$ for all $k \in \ZZ$.
For such $\wtheta$, we have $\varphi^3 \equiv (t_1,\ldots,t_{|\cI|})
\in \R^{|\cI|}$. When $d$ is odd, the summation defining (\ref{eq:MRAS3}) may be
written as that over $i,j,k \in \cI \cup -\cI$ with $i+j+k \equiv 0 \bmod d$,
and the restriction of $S_3(\varphi)$ to points $\varphi \in \R^d$ where
$r_k=r_{k,*}$ for all $k \in \ZZ$ coincides with $F^+(t)$. When $d$ is even,
we may isolate the terms of the summation in (\ref{eq:MRAS3}) where some
coordinate, say $k$, equals $d/2$. Then we must have $i+j \equiv d/2 \bmod d$,
and the constraint $i,j \not\equiv 0 \bmod d$ is equivalent to $i,j \in \cI \cup
-\cI$. When $v_{d/2}=v_{d/2,*}$, we have $t_k=0$ so
$\cos(t_i+t_j+t_k)=\cos(t_i+t_j)$. In this case, $S_3(\varphi)$ restricted to
$r_k=r_{k,*}$ is the function $F^+(t)$, where the factor
$1/2$ is produced from $1/6$ by considering the three symmetric
settings where $i$, $j$, or $k$ is $d/2$. When $v_{d/2}=-v_{d/2,*}$, we have $t_k=\pi$, so
$\cos(t_i+t_j+t_k)=-\cos(t_i+t_j)$. In this case, $S_3(\varphi)$ restricted to 
$r_k=r_{k,*}$ is the function $F^-(t)$.

Thus, if $\wt=\wvarphi^3$ is not a
critical point of $F^\pm(t)$, then $\nabla_{\varphi^3} S_3(\wvarphi)
\neq 0$. If $\wt$ is a critical point where $\lambda_{\min}(\nabla^2
F^\pm(t))<0$, then also $\lambda_{\min}(\nabla_{\varphi^3} S_3(\wvarphi))<0$.
In these cases, Lemma \ref{lemma:largesigmadescent} implies that either
$\|\nabla_\theta R(\theta)\| \geq c\sigma^{-6}$ or
$\lambda_{\min}(\nabla_\theta^2 R(\theta)) \leq -c\sigma^{-6}$, for all $\theta
\in U_{\wtheta}$ and $\sigma>\sigma_0$. If $\wt$ is a critical point of
$F^\pm(t)$ where $\lambda_{\min}(\nabla^2 F^\pm(t))>0$, then $\wvarphi$ is a
pseudo-local-minimizer of $R(\theta)$, and 
Lemma \ref{lemma:largesigmadescent} implies both that there is a unique
local minimizer of $R(\varphi)$ in $\varphi(U_{\wtheta})$ and that
$\nabla_\varphi^2 R(\varphi) \geq c\sigma^{-6}$ on $\varphi(U_{\wtheta})$.
Finally, if $F^\pm(t)$ is Morse, then this accounts for all possible points
$\wtheta$.

Taking a finite collection of these sets $U_{\wtheta}$ which cover
$\{\theta:\|\theta\| \leq M\}$, the above constants $c,\sigma_0>0$ may be chosen
to be uniform over this finite cover. Then for small enough $\rho>0$, the above
arguments establish the claims of the theorem for $\|\theta\| \leq M$.
The result for $\|\theta\|>M$ follows from Lemma
\ref{lemma:localizationstrong}.
\end{proof}

Finally, let us analyze the functions $F^\pm$ for $d \leq 5$, even $d \geq 6$,
and odd $d \geq 53$.

\begin{proof}[Proof of Corollary \ref{cor:MRAbad}]
{\it Part (a):} The result for $d=1$ or 2 follows from the analysis of
all permutations in Theorem \ref{thm:permutations}. 

For $d=3$ or 4, $\cI=\{1\}$, so $F^{\pm}(t)$ is a function of a single scalar
argument in $t_1 \in \cS$. For $d=3$,
\[\nabla F^+(t_1)=r_{1,*}^6\sin(3t_1),\quad
\nabla^2 F^+(t_1)=3r_{1,*}^6\cos(3t_1).\]
Then $F^+$ is Morse and there are six critical points, three of which are the
local minimizers $\{0,2\pi/3,4\pi/3\}$. These correspond to the three
points $\wtheta \in \O_{\theta_*}$.
For $d=4$,
\[\nabla F^\pm(t_1)=\pm 2r_{1,*}^6\sin(2t_1),\quad
\nabla^2 F^\pm(t_1)=\pm 4r_{1,*}^6\cos(2t_1).\]
Each function $F^+$ and $F^-$ is Morse with four critical points. For $F^+$,
there are two local minimizers $\{0,\pi\}$, and for $F^-$, there are two local
minimizers $\{\pi/2,3\pi/2\}$. These correspond to the four points $\wtheta \in
\O_{\theta_*}$.

For $d=5$, we have $\cI=\{1,2\}$. Let us abbreviate
\[s_i=r_{i,*}^2, \qquad u_1=2t_1-t_2, \qquad u_2=t_1+2t_2.\]
Then
\begin{align*}
\nabla F^+(t)&=\Big(2s_1^2s_2\sin u_1+s_1s_2^2\sin u_2,\;\;
-s_1^2s_2\sin u_1+2s_2s_1^2\sin u_2\Big),\\
\nabla^2 F^+(t)&=\begin{pmatrix}
4s_1^2s_2\cos u_1+s_1s_2^2\cos u_2 &
-2s_1^2s_2\cos u_1+2s_1s_2^2\cos u_2 \\
-2s_1^2s_2\cos u_1+2s_1s_2^2\cos u_2 &
s_1^2s_2\cos u_1+4s_1s_2^2\cos u_2\end{pmatrix}.
\end{align*}
From this, we may also compute
\begin{align*}
\det \nabla^2 F^+(t)&=25s_1^3s_2^3\cos u_1\cos u_2,\\
\Tr \nabla^2 F^+(t)&=5s_1^2s_2\cos u_1+5s_1s_2^2\cos u_2.
\end{align*}
For generic $\theta_*$ and hence generic $(s_1,s_2)$,
the condition $\nabla F^+(t)=0$ for a critical point requires
$\sin u_1=\sin u_2=0$. We have $\det \nabla^2 F^+(t) \neq 0$ at such
points, so $F^+$ is Morse. The condition $\nabla^2 F^+(t) \succ 0$ for a local
minimizer then requires
$\det H^+(t)>0$ and $\Tr H^+(t)>0$, so we must have
$\cos u_1=\cos u_2=1$, and hence
$t_1+2t_2 \equiv 2t_1-t_2 \equiv 0 \bmod 2\pi$. This implies that
$5t_1 \equiv 0 \bmod 2\pi$, and there are five local minimizers
$(t_1,t_2)=(0,0)$, $(2\pi/5,4\pi/5)$, $(4\pi/5,8\pi/5)$,
$(6\pi/5,2\pi/5)$, or $(8\pi/5,6\pi/5)$. These correspond to the five points
$\wtheta \in \O_{\theta_*}$. Together with Theorem \ref{thm:MRA} and Corollary
\ref{cor:MRA}, this shows part (a).

{\it Part (b):} Write $d = 2m + 2$ with $m \geq 2$ so that $\cI = \{1, 2, \ldots, m\}$.
Define the quantities $s_i = r_{i, *}^2$ so that $s_i = s_{-i}$.  We exhibit a family
of points where $F^\pm(t)$ have spurious local minima. For each $a \in \{0,
\ldots, d - 1\}$,
define $t^{a} = (t^{a}_1, \ldots, t^{a}_m)$ by
\begin{equation} \label{eq:t-def}
t^a_i = \frac{2\pi ai}{d}.
\end{equation}
For $i, j, k \in \cI \cup - \cI$ with $i + j + k \equiv 0 \bmod{d}$, we find that
\[
t^a_i + t^a_j + t^a_k \equiv 0 \bmod{2\pi}.
\]
Similarly, for $i, j \in \cI \cup - \cI$ with $i + j \equiv d/2 \bmod{d}$, we have
$t^a_i + t^a_j \equiv 0 \bmod{\pi}$.  Together, these imply that $\nabla
F^\pm(t^a) = 0$.

We now restrict to $s_1 = \cdots = s_m = 1$.  We claim that $t^a$ is a local
minimum of both $F^+$ and $F^-$ for sufficiently small values of $s_{m + 1}$. Define
\[
F_0(t) := - \frac{1}{6} \sum_{\substack{i, j, k \in \cI \cup -\cI \\ i + j + k
\equiv 0 \bmod{d}}}
\cos(t_i + t_j + t_k)
\]
so that for $d$ even we have
\begin{equation} \label{eq:fpm-eq}
F^\pm(t) = F_0(t) \mp \frac{1}{2} s_{m+1} \sum_{\substack{i, j \in \cI \cup -
\cI \\ i + j \equiv d/2 \bmod{d}}}
\cos(t_i + t_j).
\end{equation}
We show that $\nabla^2 F_0(t^a)$ is diagonally dominant: Denote by $\partial_p$
the partial derivative in $t_p$. For any $p \in \cI$,
\begin{align*}
\partial_p F_0(t)&=\frac{1}{6}\mathop{\sum_{i,j,k \in \cI \cup
-\cI}}_{i+j+k \equiv 0 \bmod d} \sin(t_i+t_j+t_k) \cdot \\
&\hspace{1in}\Big(\1\{i=p\}+\1\{j=p\}+\1\{k=p\}
-\1\{i=-p\}-\1\{j=-p\}-\1\{k=-p\}\Big)\\
&=\mathop{\sum_{j,k \in \cI \cup -\cI}}_{p+j+k \equiv 0 \bmod d}
\sin(t_p+t_j+t_k)
\end{align*}
where the second line applies symmetry with respect to permutations of $(i,j,k)$ and negation $(i,j,k) \mapsto (-i,-j,-k)$.
Then, for any $q \in \cI$,
\begin{align}
\partial_{pq} F_0(t)&=
\mathop{\sum_{j,k \in \cI \cup -\cI}}_{p+j+k \equiv 0 \bmod d}
\cos(t_p+t_j+t_k) \cdot \nonumber\\
&\hspace{1in}\Big(\1\{p=q\}+\1\{j=q\}+\1\{k=q\}-\1\{j=-q\}-\1\{k=-q\}\Big)\nonumber\\
&=\1\{p=q\}\mathop{\sum_{j,k \in \cI \cup -\cI}}_{p+j+k \equiv 0 \bmod d}
\cos(t_p+t_j+t_k)\nonumber\\
&\hspace{0.2in}+\1\{p+q \not\equiv d/2 \bmod d\} \cdot
2\cos(t_p+t_q+t_{-p-q})-\1\{p \neq q\} \cdot 2\cos(t_p+t_{-q}+t_{q-p}).
\label{eq:hessF0}
\end{align}
At any point $t^a$, we have $t_i^a+t_j^a+t_k^a \equiv 0 \bmod 2\pi$, so
$\cos(t_i^a + t_j^a + t_k^a)=1$ for all triples $(i,j,k)$ above. Then
\[\partial_{pp} F_0(t^a)=2m-2+2 \cdot \1\{p \equiv d/4 \bmod d\},\]
where the first term accounts for the sum over $j \in \cI \cup -\cI$
excluding $j=-p$ and $j=d/2-p$. We also have
\[\sum_{q:q \neq p} |\partial_{pq} F_0(t^a)|=
\sum_{q:q \neq p} 2 \cdot \1\{p+q \equiv d/2 \bmod d\}
=2 \cdot \1\{p \not\equiv d/4 \bmod d\}.\]
Thus, for $m \geq 2$,
\[\partial_{pp} F_0(t^a)  - \sum_{q:q \neq p} |\partial_{pq} F_0(t^a)|
=2m-2>0.\]
This implies that $\nabla F_0(t^a)$ is diagonally dominant and thus positive definite. Taking
$s_{m + 1}$ sufficiently small in (\ref{eq:fpm-eq}), we find that the Hessians $\nabla^2 F^\pm(t^a)$
are also positive definite, meaning that each $t^a$ for $a=0,\ldots,d-1$ 
is a local minimum of both $F^+(t)$ and $F^-(t)$. By continuity, this statement
also holds for $(s_1,\ldots,s_{m+1}) \in U_s$ and some open set $U_s
\subset \R^{m+1}$.

Now for each $\theta_* \in \R^d$ such that
$(r_{1,*}^2,\ldots,r_{{m+1},*}^2) \in U_s$, Theorem \ref{thm:MRA}(a--b)
implies that $R(\theta)$ has $2d$ local minima (for sufficiently
large $\sigma$), corresponding to these $2d$
local minima $t^a$ for $F^\pm(t)$.
Of these, $d$ local minima constitute the orbit $\cO_{\theta_*}$, and
the other $d$ local minima are spurious and lie on another orbit $\cO_{\mu_*}$
for some $\mu_* \in \R^d$. The set of such $\theta_*$ contains an open set $U
\subset \R^d$, and this establishes part (b).

{\it Part (c):} Write $d = 2m + 1$ so that $\cI = \{1, 2, \ldots, m\}$.
We will exhibit a family of points where
$F^+(t)$ has spurious local minima.  For each $a \in \{0, \ldots, d - 1\}$, define
$t^{a, \pm} = (t^{a, \pm}_1, \ldots, t^{a, \pm}_m)$ by
\[
t^{a, +}_i = \frac{2\pi a i}{d} \qquad \text{ and } \qquad t^{a, -} = t^{a, +} + (\pi, 0, \ldots, 0).
\]
For $i, j, k \in \cI \cup - \cI$ with $i + j + k \equiv 0 \bmod{d}$, we find
\begin{align*}
t^{a, +}_i + t^{a, +}_j + t^{a, +}_k &\equiv 0 \bmod{2\pi}\\
t^{a, -}_i + t^{a, -}_j + t^{a, -}_k &\equiv \pi \Big(\bI\{i = 1\} + \bI\{j =
1\} + \bI\{k = 1\}\Big)\bmod{2\pi}.
\end{align*}
So $\nabla F^+(t^{a, \pm}) = 0$. It may be checked that the $d$ points $t^{a,+}$
are minimizers of $F^+(t)$ and correspond to the $d$ points of the true
orbit $\cO_{\theta_*}$. Thus, we focus on the points $t^{a,-}$, which
correspond to a second orbit $\cO_{\mu_*}$ for some $\mu_* \in \R^d$.

Again set $s_i = r_{i, *}^2$.
We now construct $(s_1,\ldots,s_m)$ for which these points $t^{a,-}$ are local
minima of $F^+$. By a computation similar to (\ref{eq:hessF0}), we obtain
for all $p,q \in \cI$ that
\begin{align*}
\partial_{pq} F^+(t)&=
\1\{p=q\}\mathop{\sum_{j,k \in \cI \cup -\cI}}_{p+j+k \equiv 0 \bmod d}
s_ps_js_k \cos(t_p+t_j+t_k)\\
&\hspace{0.2in}+2s_ps_qs_{p+q}\cos(t_p+t_q+t_{-p-q})
-\1\{p \neq q\} \cdot 2s_ps_qs_{q-p}\cos(t_p+t_{-q}+t_{q-p}).
\end{align*}
We take $m \geq 8$.
Let us first consider $s_4 = \cdots = s_m = 1$ and $s_3 = 0$, with $s_1>0$ and
$s_2>1$ to be chosen later. Then, applying $\cos(t_i^{a,-}+t_j^{a,-}+t_k^{a,-})
=(-1)^{\1\{i=1\}+\1\{j=1\}+\1\{k=1\}}$, an explicit computation shows
that the diagonal terms of $\nabla^2 F^+(t^{a,-})$ are given by
\[{\tiny \partial_{pp} F^+(t^{a,-})
=\begin{cases}
4s_1^2s_2-(2m-7)s_1 & p=1 \\
4s_2^2+s_1^2s_2+(2m-8)s_2 & p=2 \\
0 & p=3 \\
s_2^2-2s_1+2s_2+2m-8 & p=4 \\
-4s_1+2s_2+2m-9 & p=5 \\
-4s_1+4s_2+2m-10 & p=6 \\
-4s_1+4s_2+2m-13 & p=m-1 \\
-6s_1+4s_2+2m-13 & p=m \\
-4s_1+4s_2+2m-11 & \text{for all other } p \\
\end{cases}}\]
and the off-diagonal terms (for $q>p$) are given by
\[{\tiny \partial_{pq} F^+(t^{a,-})
=\begin{cases}
-2s_1^2s_2 & (p,q)=(1,2) \\
-2s_1 & (p,q)=(1,4) \\
-2s_2^2+2s_2 & (p,q)=(2,4) \\
2s_2 & (p,q)=(2,5) \\
-2s_2 & (p,q)=(m-2,m) \\
2s_1+2s_2 & (p,q)=(m-1,m) \\
2s_1+2 & (p,q)=(p,p+1) \text{ for all } p=4,\ldots,m-2 \\
-2s_2+2 & (p,q)=(p,p+2) \text{ for all } p=4,\ldots,m-3 \\
2 & (p,q)=(p,p+3) \text{ for all } p=4,\ldots,m-3 \\
0 & \text{for all other } (p,q)
\end{cases}}\]
For $v = (0, 2s_2, 0, -s_2, 0, \ldots, 0)$ define
\[
X := \nabla^2 F^+(t^{a, -}) - v v^{\mathsf{T}},
\]
which removes the $s_2^2$ contributions from the entries $(2,2)$, $(2,4)$,
$(4,2)$, and $(4,4)$. Let $Y \in \R^{(m-1) \times (m-1)}$
be the minor of $X$ excluding the third row and column, indexed by
$\{1,2,4,\ldots,m\}$, and set
\[
\Delta_p := Y_{pp} - \sum_{q:q \neq p} |Y_{pq}|.
\]
Then the above expressions yield
  \[\Delta_1 = 2s_1^2 s_2 - (2m - 5)s_1 \qquad
  \Delta_2 = -s_1^2s_2+(2m - 12)s_2  \qquad 
  \Delta_4 = - 6s_1 - 2s_2 + 2m - 10\]
  \[\Delta_5 = - 8s_1 - 2s_2 + 2m - 13 \qquad
  \Delta_6 = - 8s_1 + 2m - 12 \qquad
\Delta_p = -8s_1+2m-15 \quad \text{ for } p=7,\ldots,m.\]
We now choose $s_1,s_2$ to ensure that each $\Delta_p$ above is strictly
positive: This is true if and only if
\[2m-12 > s_1^2 \quad \text{ and } \quad 2s_1s_2>2m-5 \quad \text{ and } \quad 2m-13>8s_1+2s_2
\quad \text{ and } \quad 2m-15>8s_1.\]
Setting $s_1 = \frac{1}{2} \sqrt{m}$ and $s_2 = \frac{2m - 5}{2s_1} + \eps$ for some small
$\eps > 0$, we may verify that these expressions hold for $m \geq 26$.
Then $Y$ is strictly diagonally dominant, and hence positive-definite.

This implies that all eigenvalues of $\nabla^2 F^+(t^{a,-})$ are strictly
positive, except for a single eigenvalue of 0 corresponding to the
eigenvector $e_3$.
We now increase $s_3$ from 0 a small constant $\delta$ to remove this 0
eigenvalue: Fixing $s_1,s_2$ and
$s_4=\ldots=s_m=1$ as above, denote by $h(s_3)$ the value of $\partial_{33}
F^+(t^{a,-})$ at $(s_1,s_2,s_3,\ldots,s_m)$. Then
\[
h(s_3) = -2s_1s_2 s_3 - 2s_1s_3 + 2s_2s_3 + 4 s_3^2 + s_3 (2m - 9).
\]
Since $e_3$ is the eigenvector of $\nabla^2 F^+(t^{a,-})$
corresponding to 0, the derivative of this 0 eigenvalue with respect to $s_3$ is
(see \cite[Eq.\ (67)]{cookbook})
\[h'(s_3)\big|_{s_3 = 0} = -2s_1s_2 - 2s_1 + 2s_2 + 2m - 9.\]
For $m \geq 26$ and the above choices of $s_1,s_2$, this derivative is positive.
Then for some sufficiently small $s_3=\delta$,
$\nabla^2 F^+(t^{a, -})$ is strictly positive definite.  We conclude that for
this choice of $(s_1,\ldots,s_m)$, each $t^{a,-}$ is a local minimum of
$F^+(t)$.
By continuity, this holds also for all $(s_1,\ldots,s_m) \in U_s$ and some open
set $U_s \subset \R^m$. Then Theorem \ref{thm:MRA}(a) implies that for each
$\theta_* \in \R^d$ where $(r_{1,*}^2,\ldots,r_{m,*}^2) \in U_s$, $R(\theta)$
has $d$ local minima (for sufficiently large $\sigma$), and these do not belong
to the orbit $\cO_{\theta_*}$. The condition $m \geq 26$ corresponds to
$d \geq 53$, and this shows part (c).

The analogous statements for the empirical landscape of $R_n(\theta)$
follow from Corollary \ref{cor:MRA}.
\end{proof}

\appendix

\section{Auxiliary lemmas and proofs}

\subsection{Cumulants and cumulant bounds}\label{appendix:cumulants}

The order-$\ell$ cumulant $\kappa_\ell(X)$ of a random variable $X$ is
defined recursively by the moment-cumulant relations
\[\E[X^\ell]=\sum_{\text{partitions } \pi \text{ of } [\ell]}\;\;
\prod_{S \in \pi} \kappa_{|S|}(X).\]
More generally, for random variables $X_1,\ldots,X_\ell$, the mixed cumulants
$\kappa_{|S|}(X_k:k \in S)$ for $S \subseteq [\ell]$ are defined recursively by
the moment-cumulant relations
\[\E\left[\prod_{i \in T} X_i\right]
=\sum_{\text{partitions } \pi \text{ of } T}\;\;
\prod_{S \in \pi} \kappa_{|S|}(X_k:k \in S).\]
These relations may be M\"obius-inverted to obtain the explicit definition
\begin{equation}\label{eq:scalarmomentcumulant}
\kappa_\ell(X_1,\ldots,X_\ell)
=\sum_{\text{partitions } \pi \text{ of } [\ell]}\;\;
(|\pi|-1)!(-1)^{|\pi|-1}\prod_{S \in \pi} \E\left[\prod_{i \in S}
X_i\right]
\end{equation}
where $|\pi|$ is the number of sets in $\pi$
(see \cite[Sec.~2.3.4]{mccullagh2018tensor}).
If $X_1=\ldots=X_\ell=X$, then $\kappa_\ell(X_1,\ldots,X_\ell)=\kappa_\ell(X)$.
The mixed cumulant $\kappa_\ell(X_1,\ldots,X_\ell)$ is multi-linear and
permutation-invariant in its $\ell$ arguments.
We have $\kappa_1(X)=\E[X]$, $\kappa_2(X)=\Var[X]$, and
$\kappa_2(X_1,X_2)=\Cov[X_1,X_2]$.

The cumulant generating function of a random variable $X$ is
the formal power series
\begin{equation}\label{eq:cumulantseries}
K_X(s)=\sum_{\ell=1}^\infty \kappa_\ell(X) \frac{s^\ell}{\ell!}.
\end{equation}
If $\log \E[e^{sX}]$ exists on a neighborhood of $0$, then its $\ell^\text{th}$ derivative
at 0 is $\kappa_\ell(X)$. 
Similarly, the cumulant generating function of a random vector $u \in \R^d$ is
the formal power series 
\[K_u(\theta)=\sum_{\ell_1,\ldots,\ell_d=1}^\infty
\frac{\theta_1^{\ell_1}\ldots \theta_d^{\ell_d}}{\ell_1!\ldots \ell_d!}
\kappa_{\ell_1+\ldots+\ell_d}(u_1,\ldots,u_1,\ldots,u_d,\ldots,u_d),\]
where in $\kappa_{\ell_1+\ldots+\ell_d}(u_1,\ldots,u_1,\ldots,u_d,\ldots,u_d)$,
each $u_j$ appears $\ell_j$ times. If $\log \E[e^{\langle \theta,u \rangle}]$ exists in
a neighborhood of $\theta = 0$, its $\ell^\text{th}$ derivative at 0 is
\[\kappa_\ell(u) \in (\R^d)^{\otimes \ell},\]
where $\kappa_\ell(u)$ denotes the order-$\ell$ cumulant tensor of $u$. This has
entries, for $i_1,\ldots,i_\ell \in [d]$,
\[\kappa_\ell(u)_{i_1,\ldots,i_\ell}=\kappa_\ell(u_1,\ldots,u_1,\ldots,u_d,\ldots,u_d)\]
where each coordinate $u_j$ appears $\ell_j$ times if $\ell_j$ of the indices
$i_1,\ldots,i_\ell$ equal $j$. The first two cumulant tensors are $\kappa_1(u)=\E[u]$ and
$\kappa_2(u)=\Cov[u]$.

More generally, if $\log \E[e^{\langle \theta,u \rangle}]$ exists in a neighborhood of $\theta$, a
reweighted exponential family law $p(u|\theta)$ may be defined by the expectation
\[\E[f(u) \mid \theta]=\E[f(u)e^{\langle \theta,u \rangle-K_u(\theta)}]
=\frac{\E[f(u)e^{\langle \theta,u \rangle}]}{\E[e^{\langle \theta,u \rangle}]}.\]
Then the $\ell^\text{th}$ derivative of $\log \E[e^{\langle \theta,u \rangle}]$ at $\theta$
is $\kappa_\ell(u \mid \theta)$, the order-$\ell$ cumulant tensor of this reweighted law
(see \cite[Theorem 1.5.10]{lehmann2006theory}).

The following result provides an upper bound for these cumulants when $X,X_1,\ldots,X_\ell$ are
bounded random variables. This bound is tight up to an
exponential factor in $\ell$, as may be seen for $X \sim \Unif([0,1])$ where
$\kappa_\ell(X)=B_\ell/\ell$ and $B_\ell$ is the $\ell^\text{th}$ Bernoulli number
(see \cite[Example 2.7]{billey2020asymptotic}), satisfying $|B_{2\ell}| \sim
4\sqrt{\pi \ell}(\ell/(\pi e))^{2\ell}$.

\begin{lemma}\label{lem:cumulantbounds}
\begin{enumerate}[(a)]
\item If $|X| \leq m$ almost surely, then $|\kappa_\ell(X)| \leq (m\ell)^\ell$.
\item If $|X_i| \leq m_i$ almost surely for each $i=1,\ldots,\ell$, then
$|\kappa_\ell(X_1,\ldots,X_\ell)| \leq \ell^\ell m_1\ldots m_\ell$.
\item If $|X| \leq m$ almost surely, then the series (\ref{eq:cumulantseries})
is absolutely convergent for $|s|<1/(me)$.
\end{enumerate}
\end{lemma}
\begin{proof}
We apply (\ref{eq:scalarmomentcumulant}). Enumerating over $v=|\pi|$, we have
\[\sum_{\text{partitions } \pi \text{ of } [\ell]}\;\; (|\pi|-1)!
=\sum_{v=1}^\ell \frac{(v-1)!}{v!} \sum_{\ell_1+\ldots+\ell_v=\ell}
\binom{\ell}{\ell_1,\ldots,\ell_v}=\sum_{v=1}^\ell \frac{1}{v} \cdot v^\ell
=\sum_{v=1}^\ell v^{\ell-1} \leq \ell^\ell,\]
so (b) follows from (\ref{eq:scalarmomentcumulant}).
Specializing to $X_1=\ldots=X_\ell$ yields (a), and (c) follows
from (a) and the bound $\ell! \geq \ell^\ell/e^\ell$.
\end{proof}

\subsection{Reparametrization by invariant polynomials}\label{appendix:transcendence}

We prove Lemmas \ref{lemma:phiconstruction} and \ref{lemma:polynomialreparam}.
Parts of these are well-known, but we provide a brief proof here for
convenience.

We recall the more usual definition of transcendence degree
for two fields $E \subset F$, where $\trdeg(F/E)$ is the
maximum number of elements in $F$ that are algebraically independent over $E$.
We verify also in the proof of Lemma \ref{lemma:phiconstruction} that our
definition of $\trdeg(A)$ for any subset $A \subseteq \cR^G$ coincides with
$\trdeg(\R(A)/\R)$, where $\R(A)$ is the field of rational functions generated by $A$.

\begin{proof}[Proof of Lemma \ref{lemma:phiconstruction}]
Consider any subsets $A' \subseteq A \subseteq \cR^G$, where $A'$ is
algebraically independent. Call $A'$ \emph{maximal} in $A$
if $A' \cup \{a\}$ is algebraically dependent for every
$a \in A \setminus A'$. Let $A'$ be maximal in $A$, and suppose $|A'|=k$.
Let $\R(A)$ and $\R(A')$ be the fields of $G$-invariant
rational functions generated by $A$ and $A'$. Algebraic independence
of $A'$ implies that $\trdeg(\R(A')/\R)=k$. Maximality of $A'$ implies that
each $a \in A$ is algebraic over $\R(A')$. Then $\R(A)$ is an algebraic
extension of $\R(A')$, so $\trdeg(\R(A)/\R(A'))=0$, hence $\trdeg(\R(A)/\R)=k$.
This verifies that every such maximal algebraically independent set $A'$ of $A$
has the same cardinality, which coincides with $\trdeg(\R(A)/\R)$.

Letting $\R(\theta_1,\ldots,\theta_d)$ and $\R(\cR^G)$ be the fields of all
rational functions and all $G$-invariant rational functions in $\theta$,
respectively, $\R(\theta_1,\ldots,\theta_d)$ is an algebraic extension of
$\R(\cR^G)$ (see \cite[Lemma
11]{Cox1992}), so $\trdeg(\R(\theta_1,\ldots,\theta_d)/\R(\cR^G))=0$.
Since $\trdeg(\R(\theta_1,\ldots,\theta_d)/\R)=d$, this shows
$\trdeg(\cR^G) = \trdeg(\R(\cR^G)/\R) = d$.  Thus $\trdeg(\cR_{\leq L}^G)=d$ for
some $L \geq 1$, and there exists a smallest such $L$.
To construct $\varphi$, let $\varphi^1$ be any maximal algebraically independent
subset of $\cR^G_1$. The above implies that the cardinality of $\varphi^1$ is
$d_1=\trdeg(\cR^G_1)$. These polynomials have degree exactly 1.
Now extend this to any maximal algebraically independent subset
$(\varphi^1,\varphi^2)$ of $\cR^G_2$. The above implies that the cardinality of
$\varphi^2$ is $d_2=\trdeg(\cR^G_2)-\trdeg(\cR^G_1)$. If $d_2>0$, then the
polynomials of $\varphi^2$ must have degree exactly 2, by maximality of
$\varphi^1$. We may iterate this procedure to obtain $(\varphi^1,\ldots,\varphi^L)$.
\end{proof}

\begin{proof}[Proof of Lemma \ref{lemma:polynomialreparam}]
For parts (a) and (b), recall by \cite[Theorem 2.3]{ehrenborg1993apolarity} that
$\varphi_1, \ldots, \varphi_k$ are algebraically independent if and only if
$\nabla \varphi_1, \ldots, \nabla \varphi_k$ are linearly independent over
the field of rational functions
$\C(\theta_1, \ldots, \theta_d)$.  For part (a), this linear independence means that
some maximal $k \times k$ minor of the $k \times d$ derivative $\der_\theta \varphi$ 
does not vanish in $\C(\theta_1, \ldots, \theta_d)$. Then that same maximal minor
does not vanish in $\C$ for generic $\theta \in \R^d$, showing linear independence for
generic $\theta$.  For part (b), linear independence at any point $\theta$ implies
that some maximal minor of $\der_\theta \varphi$
does not vanish and hence $\nabla \varphi_1, \ldots, \nabla \varphi_k$ are linearly
independent over $\C(\theta_1, \ldots, \theta_d)$, implying algebraic independence.

For part (c), let us arbitrarily extend $(\varphi_1,\ldots,\varphi_k)$ to a
system of coordinates $\varphi=(\varphi_1,\ldots,\varphi_d)$, where $\der
\varphi$ is non-singular in a neighborhood of $\wtheta$.
(Here, $\varphi_{k+1},\ldots,\varphi_d$ are
general analytic functions and need not belong to $\cR^G$.)
By the inverse function theorem, there is a
neighborhood $U$ of $\wtheta$ and corresponding neighborhood $\varphi(U)$ of
$\varphi(\wtheta)$ for which $\theta$ is an analytic function of
$\varphi \in \varphi(U)$. Then any polynomial $\psi \in \cR^G_{\leq \ell}$ is
such that $\psi(\theta)$ is also an analytic function of
$\varphi \in \varphi(U)$. Let us write this
function as $\psi=f(\varphi)$. Then $\psi(\theta)=f(\varphi(\theta))$ for all
$\theta \in U$, so by the chain rule,
\begin{equation}\label{eq:psiphi}
\der \psi(\theta)=\der_\varphi f(\varphi) \cdot \der \varphi(\theta).
\end{equation}
By part (b), since $(\varphi_1,\ldots,\varphi_k,\psi)$ are algebraically
dependent, the gradients $\nabla \varphi_1,\ldots,\nabla \varphi_k,\nabla \psi$
must be linearly dependent at every $\theta \in U$. So $\nabla \psi=\der
\psi^\top$ belongs to the span of $\nabla \varphi_1,\ldots,\nabla \varphi_k$ at
every $\theta \in U$. Since $\der \varphi(\theta)$ is a non-singular matrix,
this and (\ref{eq:psiphi}) imply that $\nabla_\varphi f=\der_\varphi f^\top$
has coordinates $k+1,\ldots,d$ equal to 0 for every $\varphi \in \varphi(U)$.
So $f$ is in fact an analytic function of only the first $k$ variables
$\varphi_1,\ldots,\varphi_k$ over $\varphi(U)$, which is the statement of part (c).
\end{proof}

\subsection{Concentration inequality for $\sum_i
\|\eps_i\|^3$}\label{appendix:sumcubes}

We prove the inequality (\ref{eq:sumcubes}).
We use the following concentration result, which specializes
\cite[Theorem 1.2]{adamczak2015concentration} to Gaussian random variables.

\begin{theorem}[\cite{adamczak2015concentration}]\label{thm:AW15}
Suppose $f:\R^m \to \R$ is $D$ times continuously-differentiable, and
$\nabla^D f(x)$ is uniformly bounded over $x \in \R^m$. Let $\eps \in \R^m$
have i.i.d.\ $\N(0,1)$ coordinates. Then for a constant $c \equiv c(D)>0$,
\[\P[|f(\eps)-\E f(\eps)| \geq t] \leq 2e^{-c\eta_f(t)}.\]
Here,
\begin{align*}
\eta_f(t)&=\min\Bigg(\min_{\text{partitions } \cJ \text{ of } [D]}
\left(\frac{t}{\sup_{x \in \R^m} \|\nabla^D f(x)\|_\cJ}\right)^{2/|\cJ|},\\
&\hspace{1in}\min_{1 \leq d \leq D-1} \min_{\text{partitions } \cJ \text{ of } [d]}
\left(\frac{t}{\|\E[\nabla^d f(\eps)]\|_\cJ}\right)^{2/|\cJ|}\Bigg)
\end{align*}
where $|\cJ| \equiv K$ is the number of sets in the
partition $\cJ=\{J_1,\ldots,J_K\}$ of $[d]$, and
\[\|A\|_\cJ=\sup\left(\sum_{i_1,\ldots,i_d=1}^m
a_{i_1,\ldots,i_d} \prod_{k=1}^K x_{(i_\ell:\ell \in J_k)}^{(k)}:
\|x^{(k)}\|_{\HS} \leq 1 \text{ for all } k=1,\ldots,K\right).\]
In this expression, $x^{(k)}$ denotes an order-$|J_k|$ tensor in $(\R^m)^{\otimes |J_k|}$,
and $x_{(i_\ell:\ell \in J_k)}^{(k)}$ is its entry at the
indices $(i_\ell:\ell \in J_k)$.
\end{theorem}

To show (\ref{eq:sumcubes}), let us
write the coordinates of $\eps_i$ as $\eps_{ij}$. We consider
\[f(\eps_1,\ldots,\eps_n)=\sum_{i=1}^n \|\eps_i\|^3\]
as a function of the $m=nd$ standard Gaussian variables $\eps_{ij}$, and apply the
above result with $D=3$ and this function $f:\R^{nd} \to \R$.
We analyze $\eta_f(t)$: Applying
$\partial_{\eps_{ij}}\|\eps_i\|=\eps_{ij}/\|\eps_i\|$, a direct computation yields
\begin{align*}
\partial_{\eps_{ij}} f&=3\|\eps_i\|\eps_{ij},\\
\partial_{\eps_{ij}}\partial_{\eps_{ik}} f&=3\|\eps_i\|\1\{j=k\}
+3\eps_{ij}\eps_{ik}/\|\eps_i\|,\\
\partial_{\eps_{ij}}\partial_{\eps_{ik}}\partial_{\eps_{i\ell}}
f&=3(\eps_{i\ell}\1\{j=k\}+\eps_{ik}\1\{j=\ell\}
+\eps_{ij}\1\{k=\ell\})/\|\eps_i\|-3\eps_{ij}\eps_{ik}\eps_{i\ell}/\|\eps_i\|^3,
\end{align*}
and all other partial derivatives up to order three are 0.
Taking expectations above
and applying sign invariance of $\eps_{ij}$, we have $\E[\nabla f]=0$ and
$\E[\nabla^2 f]=c\,\Id$ (in dimension $nd \times nd$) for a constant $c>0$.
Then $\|\E[\nabla f]\|_{\{1\}}=0$,
$\|\E[\nabla^2 f]\|_{\{1,2\}}=\|\E[\nabla^2 f]\|_{\HS}=c\sqrt{n}$, and
$\|\E[\nabla^2 f]\|_{\{1\},\{2\}}=\|\E[\nabla^2 f]\|=c$. Thus
\begin{equation}\label{eq:etaboundprelim}
\min_{1 \leq d \leq D-1} \min_{\text{partitions } \cJ \text{ of } [d]}
\left(\frac{t}{\|\E[\nabla^d f(\eps)]\|_\cJ}\right)^{2/|\cJ|}
\geq c'\min(t^2/n,t).
\end{equation}

The third derivative $A=\nabla^3 f$ has $n$ non-zero
blocks of size $d \times d \times d$, with entries uniformly bounded in the range
$[-12,12]$. We observe that for $\cJ=\{\{1,2,3\}\}$,
\[\|A\|_{\{1,2,3\}}=\|A\|_\HS \leq C\sqrt{n}.\]
For $\cJ=\{\{1,2\},\{3\}\}$, denote by $B_1,\ldots,B_n$ the $n$ blocks of $d$
consecutive coordinates in $[nd]$, and by $\|z_B\|_2^2=\sum_{i \in B} z_i^2$.
Then, since $a_{ijk}=0$ unless $i,j,k$ belong to the same such block,
\begin{align*}
\|A\|_{\{1,2\},\{3\}}&=\sup\left(\sum_{i,j,k=1}^{nd}
a_{ijk}y_{ij}z_k:\|Y\|_{\HS} \leq 1,\|z\|_2 \leq 1\right)\\
&=\sup\left(\sum_{i,j=1}^{nd}\left(\sum_{k=1}^{nd}
a_{ijk}z_k\right)^2:\|z\|_2 \leq 1 \right)^{1/2}\\
&=\sup\left(\sum_{\ell=1}^n \sum_{i,j \in B_\ell} \left(\sum_{k \in B_\ell}
a_{ijk}z_k\right)^2: \|z\|_2 \leq 1 \right)^{1/2}\\
&\leq C\sup \left(\sum_{\ell=1}^n \|z_{B_\ell}\|^2:\|z\|_2 \leq
1\right)^{1/2}=C.
\end{align*}
Similarly $\|A\|_{\{1,3\},\{2\}},\|A_{\{2,3\},\{1\}}\| \leq C$, and we also have
$\|A\|_{\{1\},\{2\},\{3\}} \leq \|A\|_{\{1,2\},\{3\}} \leq C$.
Combining with (\ref{eq:etaboundprelim}),
$\eta_f(t) \geq c'\min(t^{2/3},t,t^2/n)$ for a constant
$c'>0$. Then applying Theorem \ref{thm:AW15} with $t=n$,
\[\P\left[n^{-1}\Big(f(\eps_1,\ldots,\eps_n)-
\E[f(\eps_1,\ldots,\eps_n)]\Big) \geq 1\right] \leq 2e^{-cn^{2/3}}.\]
As $n^{-1}\E[f(\eps_1,\ldots,\eps_n)]=C_1$ for a constant $C_1>0$, this shows
(\ref{eq:sumcubes}) for $C_0=1+C_1$.

\bibliography{orbitMLE}
\bibliographystyle{plain}

\end{document}